\documentclass[a4paper, 12pt]{amsart}%(needed for arXiv) 

\usepackage[pdftex]{graphicx} %for arXiv

\usepackage{amssymb}
\usepackage{amsmath}
\usepackage{amsthm}
\usepackage{amscd}
\usepackage{subcaption}
\usepackage{comment}
\usepackage[all,cmtip]{xy}
\usepackage{bm}

\usepackage{color}

\usepackage[utf8]{inputenc}

\usepackage{tikz}

\usepackage{tikz-cd}

\usetikzlibrary{positioning}
\usetikzlibrary{intersections}
\usetikzlibrary{calc, quotes, angles}

\theoremstyle{plain}
  \newtheorem{thm}{Theorem}[section]
  \newtheorem{lem}[thm]{Lemma}
  \newtheorem{slem}[thm]{Sublemma}
  \newtheorem{cor}[thm]{Corollary}
  \newtheorem{prop}[thm]{Proposition}

  \newtheorem{ass}[thm]{Assertion}

\theoremstyle{definition}
  \newtheorem{defn}[thm]{Definition}
  
  \newtheorem{ex}[thm]{Example}
  
  \newtheorem{prob}[thm]{Problem}

  \newtheorem{rem}[thm]{Remark}

\theoremstyle{remark}

\makeatletter

\newcommand{\wangle}[0]{\tilde{\angle}}
\newcommand{\diam}[0]{\mathrm{diam}\,}

\newcommand{\e}[0]{\epsilon}

\newcommand{\vol}[0]{\mathrm{vol}}

\newcommand{\supp}[0]{\mathrm{supp}}

\newcommand{\ca}{\mathcal}
\newcommand{\ve}{\varepsilon}

\newcommand{\pa}{\partial}

\newcommand{\tri}{\triangle}
\newcommand{\inrad}[0]{{\rm inrad}}

\newcommand{\pmed}[0]{\par\medskip}
\newcommand{\psmall}[0]{\par\smallskip}
\newcommand{\pbig}[0]{\par\bigskip}
\newcommand{\n}[0]{\noindent}

\newcommand{\C}[0]{\mathbb C}
\newcommand{\R}[0]{\mathbb R}

\newcommand{\Z}[0]{\mathbb Z}
\newcommand{\N}[0]{\mathbb N}

\newcommand{\beq}[0]{\begin{equation}}
\newcommand{\eeq}[0]{\end{equation}}
\newcommand{\beqq}[0]{\begin{equation*}}
\newcommand{\eeqq}[0]{\end{equation*}}
\newcommand{\bali}[0]{\begin{align}}
\newcommand{\eali}[0]{\end{align}}
\newcommand{\balii}[0]{\begin{align*}}
\newcommand{\ealii}[0]{\end{align*}}

\newcommand{\benu}[0]{\begin{enumerate}}
\newcommand{\eenu}[0]{\end{enumerate}}

\begin{document}
%%%%%%%%%%%%%%%%%%%%%%%%%%

\title[Limits of manifolds with boundary]{Limits of manifolds with boundary I
%\\-- Infinitesimal structure
}

\author{Takao Yamaguchi \and Zhilang Zhang }
%\author[T.Yamaguchi]{Takao Yamaguchi}
%\author[Z.Zhang]{Zhilang Zhang}

\address{Takao Yamaguchi, Department of Mathematics, University of Tsukuba, Tsukuba 305-8571, Japan}

\email{takao@math.tsukuba.ac.jp}

\address{Zhilang Zhang, School of Mathematics, Foshan University, Foshan, China}
 \email{zhilangz@fosu.edu.cn}


\thanks{This work was supported by JSPS KAKENHI Grant Numbers 18H01118, 21H00977  and  NSFC No.11901089}

\date{\today}

\begin{abstract}
In this  paper, 
we develop the infinitesimal geometry of the limit spaces of compact Riemannian  manifolds with boundary, 
where we assume lower bounds on the sectional curvatures of manifolds and boundaries and the second fundamental forms of boundaries and an upper diameter bound.  
We mainly focus on the 
case when inradii of manifolds are uniformly bounded away from zero. 
In this case, many limit spaces have wild geometry, which arise  as the boundary 
 singular points of the limit spaces.
 We determine the infinitesimal structure at  those boundary singular points.
We also determine the Hausdorff dimensions of the boundary singular sets.
\end{abstract}

\maketitle

\setcounter{tocdepth}{2}

\tableofcontents
\section{Introduction} \label{sec:intro}
%%%%%%%% %%%%%%%%

%%%%%%%% Background on the study of Riemannian manifolds with boundary %%%%%%%%%%%%%%%%
The study of Riemannian manifolds with boundary related to collapsing began with the 
works due to Gromov
\cite{G:synthetic} and Alexander-Bishop \cite{AB}
on thin Riemannian manifolds.
In \cite{Kodani}, Kodani investigated 
the Lipschitz convergence of Riemannian manifolds with boundary. From a different point of view,  Anderson-Katsuda-Kurylev-Lassas and Tayler \cite{AKKLT} proved a precompactness theorem for a certain family of Riemannian manifolds with uniformly bounded Ricci curvature.  
In \cite{Kodani} and \cite{AKKLT}, 
they assumed lower bounds on some geometric invariants like injectivity radii.
See also Knox \cite{Kn}.
There is also a pioneering work by
J. Wong (\cite{wong0}, \cite{wong2}) 
on this subject based on the gluing construction
explained below.
%%%
In the three dimension,  Mitsuishi and Yamaguchi
\cite{MY2:dim3bdy} has made clear all the collapses of three-dimensional Alexandrov spaces with boundary.
%%%%
 For other approaches to the convergence 
of Riemannian manifolds with boundary,
see Perales \cite{Perales}, \cite{Perales2} and  Perales-Sormani \cite{PerSor}. 
\cite{Perales} is 
 based on Sormani-Wenger intrinsic flat distance  (see \cite{SW})
and \cite{PerSor}, \cite{Perales2} are based on  the convergences of open domains away from boundaries.
In Yamaguchi and Zhang \cite{YZ:inrdius},
we made clear the structure of manifolds 
with boundary whose inradii are sufficiently small. 

 Here is a short comment on recent researches
on the convergence of manifolds with boundary.
In \cite{Xu:precpt},  Xu obtained the precompactness of certain families of domains of Riemannian manifolds whose Ricci curvatures are uniformly bounded below.
%%%%%%%
He showed that 
the precompactness of the boundaries implies the precompactness of the closed domains of the manifolds. Our condition \eqref{eq:curv-assum} is related with this result.
In \cite{HY},  Huang and Yamaguchi have
discussed inradius collapsed manifolds whose Ricci curvatures are uniformly bounded below.

In the general dimension, not much is known 
about the geometric structure of  the
general limit spaces, especially in the collapsing situation.
The purpose of this paper is to develop
the geometry on the limit spaces in a 
certain family of Riemanian manifolds with 
boundary.
In the study of convergence and collapsing Riemannian manifolds with boundary,
the main problem is to control the boundary behavior in a geometric way.
In \cite{wong0}, Wong carried out a nice extension procedure over the boundary to
study collapsed manifolds with boundary under a lower sectional or Ricci curvature bound.

In the present paper, we are  
concerned with the infinitesimal structure of
the  limit spaces of complete 
Riemannian manifolds $M$ with boundary.
For $n\ge 2, \kappa, \nu\in\R$ and 
$\lambda\ge 0$,
let us consider the following  lower bounds
on the  sectional curvatures of $M$ and $\pa M$ and the second fundamental forms of $\pa M$:
\beq\label{eq:curv-assum}
      K_M \ge \kappa, \,\,\, K_{\pa M} \ge \nu,\,\,\,
 \Pi_{\partial M} \ge -\lambda.
\eeq
Here when $n=\dim M=2$,
we do not need the condition of the lower curvature bound $K_{\pa M} \ge \nu$.
%%%%%%%
For an additional constant $d>0$,
let $\ca M(n,\kappa,\nu, \lambda, d)$ denote the set of all isometry classes of 
$n$-dimensional compact  Riemannian manifolds $M$ with boundary    
satisfying \eqref{eq:curv-assum} and 
$\diam(M)\le d$.

%%%%%%%%%%% The previous family %%%%%%%%%
As indicated above,  Wong \cite{wong0} carried out the gluing construction for manifolds satisfying
\beq \label{eq:curv-assum2}
 K_M \ge \kappa, \,\,\,
 |\Pi_{\partial M}|\le \lambda.
\eeq
Note that \eqref{eq:curv-assum2} is stronger than 
\eqref{eq:curv-assum}  since \eqref{eq:curv-assum2}
 implies $K_{\pa M} \ge \nu(\kappa,\lambda)$ via 
the Gauss equation.
However, this gluing construction still works  
under the assumption \eqref{eq:curv-assum}.
Thus, the following results in \cite{wong0}
still hold with no additional argument:

\begin{itemize}
\item The family $\ca M(n,\kappa,\nu,\lambda,d)$ is precompact 
with respect  to the Gromov-Hausdorff distance.
\item  The set consisting of  all 
elements $M$ in $\ca M(n,\kappa,\nu,\lambda,d)$ having volume $\vol(M)\ge v>0$  contains only finitely many homeomorphism classes.
\end{itemize}

%%%%%% Inradius and the previous paper %%%% 
The {\it inradius} of a Riemannian manifold $M$ with boundary is defined as 
\[
                    \inrad(M):=\sup_{x\in M}{d(x,\pa M)}.
\]
Suppose that a sequence $M_i$ in $\ca M(n,\kappa,\nu,\lambda, d)$ converges to a compact  geodesic space $N$.
We say that $M_i$ {\it inradius collapses} if and only if 
$\lim_{i\to\infty} \inrad(M_i)=0$.
In \cite{YZ:inrdius}, 
we investigated the structure of 
inradius collapsed manifolds under \eqref{eq:curv-assum2}.
It should also be pointed out that
the main results in \cite{YZ:inrdius} also hold 
in $\ca M(n,\kappa,\nu, \lambda, d)$ 
without any change in the argument.
In particular, 
it was  proved in \cite{YZ:inrdius} that every limit space $N$ of inradius collapsed manifolds
is an Alexandrov space with curvature uniformly bounded below (actually curvature 
$\ge\nu$ in the present assumption \eqref{eq:curv-assum}).
%%%%

%%%%% Towards main results %%%%%%
Let $m:=\dim N$ denote the topological dimension of $N$.  It will become clear (Lemma  \ref{lem:dimNdimN0}) that $m\le n$.
In the present paper, we consider the  convergence 
\beq \label{eq:conv-intro}
\ca M(n,\kappa,\nu,\lambda, d)\ni M_i\to N.
\eeq
If $m=n$, then $\inrad(M_i)$ must be uniformly bounded away from $0$.
We call the convergence \eqref{eq:conv-intro} a {\it non-inradius convergence} in this case.
On the other hand, 
we call  \eqref{eq:conv-intro} a {\it non-inradius collapse} when $m<n$ and 
$\inrad(M_i)$ are uniformly bounded away from $0$.
Therefore,  we simply call  \eqref{eq:conv-intro} a {\it non-inradius collapse/convergence} in any case when 
$\inrad(M_i)$ are uniformly bounded away from $0$.
We also say that $M_i$  non-inradius collapses /converges  to $N$ in that case.
The purpose of this  paper is to 
determine  the infinitesimal  structure of  
the limit spaces 
$N$ under the
non-inradius collapse/convergence.
In that case,  $N$ is not necessarily an Alexandrov space. 
This is because of the presence of  the {\em boundary}  $N_0$ of $N$,
which is defined as the limit of $\pa M_i$ under the convergence \eqref{eq:conv-intro}.
%$M_i \to N$. 
It is easily seen that the  {\em interior}  ${\rm int}N:=N\setminus N_0$
satisfies the local Alexandrov curvature condition $\ge \kappa$,
although it is not complete. 
Obviously, it is not the case for points of $N_0$.
Thus our main concern is on the infinitesimal  structure at points
of $N_0$.

\psmall 
%%%%%%%%%%%%%%   Main results %%%%%%%%%
 Now we state our main results.

%%%%%%%%  Infinitesimal structure %%%%%%%%
\pmed\n
{\bf  Infinitesimal Alexandrov structure}.
We say that a 
geodesic space $X$ is {\it infinitesimally Alexandrov} if  for any 
 $x\in X$,  the space of directions $\Sigma_x(X)$  can be defined in such a way that it  is an Alexandrov space with curvature $\ge 1$
and the tangent cone $T_x(X)$ is isometric to 
the Euclidean cone over $\Sigma_x(X)$
(for the precise definition, see Section \ref{sec:int-ext}).
The nonnegative integer defined as 
\[
      {\rm rank}(X):= \sup_{x,y \in X}|\dim \Sigma_x(X)-\dim \Sigma_y(X)|
\]
is called the {\it rank} of $X$.

%%%%%%%%%%
We also say that a closed subset $X_0$ of  an infinitesimally Alexandrov space $X$ is {\it infinitesimally sub-Alexandrov} if 
for any $x\in X_0$, 
the space of directions $\Sigma_x(X_0)$ 
can be defined as a closed subset of 
$\Sigma_x(X)$ an
the intrinsic metric 
$\Sigma_x(X_0)^{\rm int}$ of $\Sigma_x(X_0)$
induced from $\Sigma_x(X)$ is 
an Alexandrov space with 
curvature $\ge 1$.
The {\it rank} of $X_0$ is similarly defined.
From here on, in the present paper, we use the terminology
intrinsic/extrinsic metrics instead of interior/exterior metrics
used in \cite{YZ:inrdius}.

 \pmed

%%%%%%%%%%%%%%%%%%%%%%%
Let $\ca M(n,\kappa,\nu,\lambda)$ (resp. $\ca M_{\rm pt}(n,\kappa,\nu,\lambda)$) denote the set of all isometry classes of (resp. pointed) 
$n$-dimensional complete  Riemannian manifolds
$M$ (resp. $(M,p)$) with boundary satisfying    
$K_M \ge \kappa$, $K_{\pa M} \ge \nu$, $\Pi_{\partial M}\ge-\lambda$.
Let a sequence $(M_i,p_i)$ in 
$\ca M_{\rm pt}(n,\kappa,\nu,\lambda)$ converge to a pointed  geodesic  space $(N,x_0)$  under a
non-inradius collapse/convergence.

\begin{thm} \label{thm:alex}
Under the above situation, we have  the following$:$
\begin{enumerate}
\item $N$ is infinitesimally Alexandrov with ${\rm rank}(N)\le 1\,;$
\item  $N_0$ is infinitesimally sub-Alexandrov with ${\rm rank}(N_0)=0$.
\end{enumerate}
\end{thm}

Theorem \ref{thm:alex} yields the possibility of developing the Alexandrov geometry at least  infinitesimally at 
any point of $N$.
\pmed\n
%%%%%%%  Boundary points and Boundary singular points %%%%%%%
{\bf Infinitesimal structure at boundary singular points.}\,
As observed above, our main task is to
make clear the infinitesimal structure at boundary points,
i.e.,  points of $N_0$.
The {\it boundary singular set} $\ca S\subset N_0$ is defined in a 
natural way as follows.
By \eqref{eq:curv-assum}, we may assume that 
the intrinsic metric $(\pa M_i)^{\rm int}$ converges to some
Alexandrov space $C_0$. We may also assume that the natural map
$\eta_i: (\pa M_i)^{\rm int} \to (\pa M_i)^{\rm ext}$
to the extrinsic metric,
converges to a surjective $1$-Lipschitz map 
$$
     \eta_0:C_0\to N_0.
$$
It is shown in \cite{YZ:inrdius}  that 
$\# \eta_0^{-1}(x)\le 2$  for 
all $x\in N_0$, and 
that $\eta_0$  preserves  the length of curves, and therefore induces a $1$-Lipschitz map  $d\eta_0:\Sigma_p(C_0) \to \Sigma_x(N)$ for any $p\in \eta_0^{-1}(x)$.
%%%%

% 
We call a point $x\in N_0$ {\it single}  (resp. {\it double})    if $\# \eta_0^{-1}(x)=1$
(resp. $\# \eta_0^{-1}(x)=2$).
We denote by $N_0^1$ (resp. by $N_0^2$) 
the set of all single points (resp. double points)
in $N_0$.
For $k=1$ or $2$,
let ${\rm int}\, N_0^k$ denote the interior of $N_0^k$ in $N_0$, and let 
$\pa N_0^k=\bar N_0^k\setminus {\rm int}\, N_0^k$ be the topological boundary of 
$N_0^k$ in $N_0$.
We set  
$$
\mathcal S^k :=\pa N_0^k \cap N_0^k.
$$ 
We call a point of $\mathcal S^1$ (resp. of 
$\mathcal S^2$)
a {\it single  singular point} 
(resp. a {\it double singular point}).
Then we define the {\it boundary singular set} $\ca S$ as $\ca S = \ca S^1 \cup \ca S^2$.
Here is a simple figure illustrating a non-inradius convergence of a surface 
$M_\e$  to the limit $N$ as $\e\to 0$:
\pbig
\begin{center}
\begin{tikzpicture}
[scale = 0.4]

\draw [-, thick] (-3.2,0) to [out=180, in=-90] (-3.7,0.5);
\draw [-, thick] (-2.8,0) to [out=180, in=00] (-3.2,0);
\draw [-, thick] (-2.1,0.4) to [out=180, in=00] (-2.8,0);
\draw [-, thick] (-1.6,0) to [out=180, in=00] (-2.1,0.4);
\draw [-, thick] (-0.8,0.6) to [out=180, in=00] (-1.6,0);
\draw [-, thick] (0,0) to [out=180, in=00] (-0.8,0.6);
\draw [-, thick] (0,0) to [out=00, in=180] (0.8,0.6);
\draw [-, thick] (0.8,0.6) to [out=00, in=180] (1.6,0);
\draw [-, thick] (1.6,0) to [out=00, in=180] (2.1,0.4);
\draw [-, thick] (2.1,0.4) to [out=00, in=180] (2.8,0);
\draw [-, thick] (2.8,0) to [out=00, in=180] (3.2,0);
\draw [-, thick] (3.2,0) to [out=00, in=-90] (3.7,0.5);

%[shift={(0,0.6)}]
\draw [-, thick] (-3.2,1) to [out=180, in=90] (-3.7,0.5);
\draw [shift={(0,1)}][-, thick] (-2.8,0) to [out=180, in=00] (-3.2,0);
\draw[shift={(0,1)}] [-, thick] (-2.1,0.4) to [out=180, in=00] (-2.8,0);
\draw[shift={(0,1)}] [-, thick] (-1.6,0) to [out=180, in=00] (-2.1,0.4);
\draw[shift={(0,1)}] [-, thick] (-0.8,0.6) to [out=180, in=00] (-1.6,0);
\draw[shift={(0,1)}] [-, thick] (0,0) to [out=180, in=00] (-0.8,0.6);
\draw[shift={(0,1)}] [-, thick] (0,0) to [out=00, in=180] (0.8,0.6);
\draw[shift={(0,1)}] [-, thick] (0.8,0.6) to [out=00, in=180] (1.6,0);
\draw[shift={(0,1)}] [-, thick] (1.6,0) to [out=00, in=180] (2.1,0.4);
\draw[shift={(0,1)}] [-, thick] (2.1,0.4) to [out=00, in=180] (2.8,0);
\draw[shift={(0,1)}] [-, thick] (2.8,0) to [out=00, in=180] (3.2,0);
\draw[-, thick] (3.2,1) to [out=00, in=90] (3.7,0.5);
  
\draw[-, thick] (-3.7,0.5) to [out=-90, in=90] (-3.7,-0.5);
\draw[-, thick] (3.7,0.5) to [out=-90, in=90] (3.7,-0.5);

\draw[-, thick] (-3.7,-0.5) to [out=-90, in=180] (-3.4,-0.8);
\draw[-, thick] (3.7,-0.5) to [out=-90, in=00] (3.4,-0.8);

\draw[-, thick] (-3.4,-0.8) -- (3.4,-0.8);

\fill (-3.6,-0.2) circle (0pt) node [left] {$\e$};
\fill (-5.5,0.5) circle (0pt) node [left] {$M_\e$};

\draw[->, thick] (0,-1.5) to  (0,-3);
\fill (0,-2.2) circle (0pt)  node[right] {\scriptsize{$GH$}};

\draw[shift={(0,-4)}]  [-, thick] (-2.6,0) to [out=180, in=00] (-3,0.2);
\draw[shift={(0,-4)}]  [-, thick] (-3,0.2) to [out=180, in=00] (-3.3,0);
\draw[shift={(0,-4)}]  [-, thick] (-3.3,0) to [out=180, in=00] (-3.5, 0.1);
\draw[shift={(0,-4)}]  [-, thick] (-3.5,0.1) to [out=180, in=00] (-3.7,0);
\draw[shift={(0,-4)}]  [-, thick] (-2.1,0.4) to [out=180, in=00] (-2.6,0);
\draw[shift={(0,-4)}]  [-, thick] (-1.6,0) to [out=180, in=00] (-2.1,0.4);
\draw[shift={(0,-4)}]  [-, thick] (-0.8,0.6) to [out=180, in=00] (-1.6,0);
\draw[shift={(0,-4)}]  [-, thick] (0,0) to [out=180, in=00] (-0.8,0.6);
\draw[shift={(0,-4)}]  [-, thick] (0,0) to [out=00, in=180] (0.8,0.6);
\draw[shift={(0,-4)}]  [-, thick] (0.8,0.6) to [out=00, in=180] (1.6,0);
\draw[shift={(0,-4)}]  [-, thick] (1.6,0) to [out=00, in=180] (2.1,0.4);
\draw[shift={(0,-4)}]  [-, thick] (2.1,0.4) to [out=00, in=180] (2.6,0);
\draw[shift={(0,-4)}]  [-, thick] (2.6,0) to [out=00, in=180] (3,0.2);
\draw[shift={(0,-4)}]  [-, thick] (3,0.2) to [out=00, in=180] (3.3,0);
\draw[shift={(0,-4)}]  [-, thick] (3.3,0) to [out=00, in=180] (3.5, 0.1);
\draw[shift={(0,-4)}]  [-, thick] (3.5,0.1) to [out=00, in=180] (3.7,0);

%%%%%%

\draw[shift={(0,-4)}]  [-, thick] (-3.5,-0.1) to [out=180, in=00] (-3.7,0);
\draw[shift={(0,-4)}]  [-, thick] (-3.3,0) to [out=180, in=00] (-3.5, -0.1);
\draw[shift={(0,-4)}]  [-, thick] (-3,-0.2) to [out=180, in=00] (-3.3,0);
\draw[shift={(0,-4)}]  [-, thick] (-2.6,0) to [out=180, in=00] (-3,-0.2);

\draw[shift={(0,-4)}]  [-, thick] (-2.1,-0.4) to [out=180, in=00] (-2.6,0);
\draw[shift={(0,-4)}]  [-, thick] (-1.6,0) to [out=180, in=00] (-2.1,-0.4);
\draw[shift={(0,-4)}]  [-, thick] (-0.8,-0.6) to [out=180, in=00] (-1.6,0);
\draw[shift={(0,-4)}]  [-, thick] (0,0) to [out=180, in=00] (-0.8,-0.6);
\draw[shift={(0,-4)}]  [-, thick] (0,0) to [out=00, in=180] (0.8,-0.6);
\draw[shift={(0,-4)}]  [-, thick] (0.8,-0.6) to [out=00, in=180] (1.6,0);
\draw[shift={(0,-4)}]  [-, thick] (1.6,0) to [out=00, in=180] (2.1,-0.4);
\draw[shift={(0,-4)}]  [-, thick] (2.1,-0.4) to [out=00, in=180] (2.6,0);
\draw[shift={(0,-4)}]  [-, thick] (2.6,0) to [out=00, in=180] (3,-0.2);
\draw[shift={(0,-4)}]  [-, thick] (3,-0.2) to [out=00, in=180] (3.3,0);
\draw[shift={(0,-4)}]  [-, thick] (3.3,0) to [out=00, in=180] (3.5, -0.1);
\draw[shift={(0,-4)}]  [-, thick] (3.5,-0.1) to [out=00, in=180] (3.7,0);

\fill [shift={(0,-4)}](-6,0) circle (0pt) node [left] {$N$};
\fill (3.7, -4) circle (3pt)  node [right] {\small{$\ca S^1$}};
\fill (-3.7, -4) circle (3pt)  node [left] {\small{$\ca S^1$}};
%\draw [->](0,-5.2) to (0,-4.3);
\draw [->,thick,dotted](0,-5.2) to (0,-4.3);
\draw [thick,dotted](-2.6,-5.2) to (2.6,-5.2);
\draw [->,thick,dotted](1.6,-5.2) to (1.6,-4.3);
\draw [->,thick,dotted](-1.6,-5.2) to (-1.6,-4.3);
\draw [->,thick,dotted](2.6,-5.2) to (2.6,-4.3);
\draw [->,thick,dotted](-2.6,-5.2) to (-2.6,-4.3);
\fill (0,-5.1) circle (0pt)  node [below] {\small{$\ca S^2$}};

\draw[-, thick] (13,-4) circle [x  radius=4,y  radius=0.6];
%\fill (12, -3.1) circle (0pt)  node [above] {\small{$C_0$}};
\draw [->, thick] (8, -4) -- (6,-4);
\fill (7, -4) circle (0pt)  node [above] {\small{$\eta_0$}};

\fill (11, 0) circle (0pt)  node [right] {\small{$(\pa M_\e)^{\rm int}$}};
\fill (17,-4) circle (0pt) node [right] {$C_0$};
\fill (12.5,-2.2) circle (0pt)  node[right] {\scriptsize{$GH$}};
\draw[->, thick] (12.5,-1.5) to  (12.5,-3);

\end{tikzpicture}
%\caption{Topology of $(M_i,\ba M_i)$}
	%\label{table:FB}
\end{center}
\vspace{-1cm}  
\begin{figure}[htbp]
  \centering
  \caption{}
  \label{fig:conv-surface}  
\end{figure}

The existence of the boundary singularities
defined above affects the geometry of the limit space. 
Therefore it is quite important to determine the 
infinitesimal structure at those boundary singular points.

 The infinitesimal structure at  a point of $N_0^2$ is rather simple (Lemma \ref{lem:double-sus}, see also Sublemma \ref{slem:eta-inter}). 
On the other hand, the infinitesimal structure
at points of  $\ca S^1$ is  a priori unclear.
Thus towards the infinitesimal characterization and classification of boundary singular points,  it is a key to describe the infinitesimal  structure  at  points of  
$\ca S^1$. 

To carry out it, for any $x\in N_0^1$ and $p\in C_0$ with $\eta_0(p)=x$,
consider the differential $d\eta_0:\Sigma_p(C_0)\to \Sigma_x(N_0)$.
It is turned out that this map is realized as the quotient map
via some isometric involution $f_*:\Sigma_p(C_0)\to \Sigma_p(C_0)$ (Theorem \ref{thm:X1-f}).
Let $\tilde{\ca F}_p$ be the fixed point set of $f_*$, and set 
$\ca F_x:=d\eta_0(\tilde{\ca F}_p)$.

The following result shows that any point $x\in\ca S^1$ is actually very singular.

\begin{thm}\label{thm:S1-nontrivial-f*}
For any $x\in\ca S^1$,  we have 
\begin{enumerate}
\item $f_*:\Sigma_p(C_0)\to \Sigma_p(C_0)$ is not the identity\,$;$
\item $\Sigma_x(N)=\Sigma_x(N_0)$ and it is isometric to the quotient
space 
$$\Sigma_p(C_0)/f_*.
$$
\end{enumerate}
\end{thm}  

Surprisingly, $f_*:\Sigma_p(C_0)\to \Sigma_p(C_0)$ can be nontrivial 
at some point $x\in {\rm int}\, N_0^1$.
We call such a point a {\it cusp}, and denote by $\ca C$ the set of all cusps (see Examples \ref{rem:X1-cust}
and \ref{ex:general-cusp}).
Cusp points have properties similar to those of 
$\ca S^1$.

Although $\ca C$ is not necessary closed
(see Example \ref{ex:smallS3}), it is easy to verify that $\ca S^1$ is closed in 
$N_0$ (see Lemma \ref{lem:closed-S1}).
Using the proof of Theorem \ref{thm:S1-nontrivial-f*}, we  prove that $\ca S^1\cup\ca C$ is closed in $N_0$ (see Theorem \ref{thm:CS1=closed}).
Moreover we show that $\ca S^1\cup\ca C$ is  contained in 
an  ``extremal subset'' of $N_0$ in the infinitesimal sense.  
Note that $\ca F_x$ is a 
proper extremal subset of the Alexandrov space $\Sigma_x(N_0)$ (see Perelman-Petrunin \cite{PtPt:extremal}).
\pmed

\begin{thm}\label{thm:S1extremal}
Let a sequence $M_i$ in $\ca M(n,\kappa,\nu,\lambda, d)$  non-inradius collapses /converges
 to a compact geodesic space $N$.
Then 
 $\Sigma_x(\ca S^1\cup\ca C)$ is  contained in $\ca F_x$  for any $x\in \ca S^1\cup\ca C$.
\end{thm}
\psmall

In Theorem \ref{thm:S1extremal},
$\Sigma_x(\ca S^1\cup\ca C)$ does not coincide with $\ca F_x$ in general. See Example \ref{ex:non-closed}(2).
\pmed

\pmed\n
{\bf Hausdorff dimensions.}\,
Finally we discuss  the Hausdorff dimensions
of the boundary singular sets together with 
the metric singular set of $N_0$.
%
%Let $m$ denote the topological dimension of $N$,
%$m:=\dim N$, and 
Let $\dim_HN$ denote the Hausdorff dimension of $N$.
For Alexandrov spaces, it is equal to
the topological dimension $\dim N$
 (\cite[Corollary 6.5]{BGP}). 
In our case of non-inradius
collapse/convergence, we have the following (see 
Lemma  \ref{lem:dimNdimN0} for more details):
$$
\dim N=\dim_H N=m, \qquad   \dim N_0=\dim_H N_0=m-1.
$$
We call a point $x\in N_0$ {\it metrically regular} if  the intrinsic metric
$\Sigma_x(N_0)^{\rm int}$ of $\Sigma_x(N_0)$ is isometric to $\mathbb S^{m-2}$. 
Otherwise we call $x$ {\it metrically singular}.
We denote by $N_0^{\rm reg}$ (resp. by $N_0^{\rm sing}$) the set of all metrically regular 
(resp. all metrically singular) points of  $N_0$.
%%%%%%%%%%%%%%%%%%%%%%%%%%

For $0\le k\le m-3$, let us denote
by $\ca S^1(k)$  (resp. $\ca C(k)$) the set of points $x\in\ca S^1$ (resp. $x\in\ca C$)
such that $\dim\ca F_x=k$.

 We define  the {\it interior} ${\rm int}\,N_0$ 
of $N_0$ as the set of all points $x\in N_0$
such that the Alexandrov space 
$\Sigma_x(N_0)^{\rm int}$ has no boundary
 (see Definition \ref{defn:int-pa-X0}).
Throughout the paper, let $\dim_H$ denote the 
Hausdorff dimension with respect to the metric of 
$N$ unless otherwise stated.

\begin{thm} \label{thm:dim(metric-sing)}
Let a sequence $M_i$ in $\ca M(n,\kappa,\nu,\lambda, d)$ non-inradius collapses /converges
to a compact geodesic space $N$.
% while ${\rm inrad}(M_i)$ have a positive lower bound.
Then we have $:$
\begin{enumerate}
\item $\dim_H N_0^{\rm sing} \le m-2$. In particular,
$\dim_H (\ca S^1\cup\ca C)\le m-2\,;$
\item  $\dim_H (\ca S^1(k)\cup\ca C(k))\le k+1$ for each $0\le k\le m-3\,;$
\item $\dim_H (N_0^{\rm sing} \cap {\rm int}\, N_0) \le m-3$.
\end{enumerate}
\end{thm}

See Burago-Gromov-Perelman \cite{BGP}, Otsu-Shioya \cite{OS} for the results in
 Alexandrov geometry corresponding to Theorem \ref{thm:dim(metric-sing)}(1),  (3) (see also Theorem \ref{thm:dim-sing}).

\begin{thm}\label{thm:nowhere}
If \,${\rm int}\, N_0^2$ is nonempty, then
 $\dim_H \ca S^2\ge m-2$.
%\begin{enumerate}  
%\item $\dim_H \ca S^2\ge m-2$  if \,${\rm int}\, N_0^2$ is nonempty\,$;$
% \item $\ca S^2$ is a nowhere dense subset of $N_0$.
%\end{enumerate}
\end{thm}

\begin{rem} \label{rem:randomS2}   
(1)\, The dimension estimates in Theorems \ref{thm:dim(metric-sing)} and \ref{thm:nowhere} are sharp. See Examples \ref{ex:non-closed}(1).
See also Corollary \ref{cor:nowhere}.
\par\n
(2)
In contrast with $\ca S^1$ and $\ca C$, 
the $(m-1)$-dimensional Hausdorff measure $\ca H^{m-1}(\ca S^2)$ could be positive
in some cases.  On the other hand, the case $\dim_H \ca S^2< \dim_H \ca S^1$
may occur in some other cases. See Examples \ref{ex:Cantor} and \ref{ex:smallS2}.
\end{rem}

\pmed\n
 {\bf Basic strategy.}\,
For the limit space $N$, we consider an Alexandrov
space $Y$ extending $N$ due to \cite{wong0}.
Developing  Alexandrov geometry of $Y$ with the help of
the infinitesimal Alexandrov structure of 
$N$ and the infinitesimal sub-Alexandrov structure of $N_0$,  we determine the infinitesimal 
structures of $N$ and $N_0$. 
This kind of geometry seems of independent interest in contrast with
the  geometry of submanifolds in Riemannian manifolds.

%%%%%%%%%%%   The organization of the paper  %%%%%
The organization of the paper 
is as follows.

In section \ref{sec:prelim},  we first recall basic materials on 
Alexandrov spaces with curvature bounded below and Wong's extension procedure for 
Riemannian manifold with boundary. 

 We begin Section \ref{sec:non-inradius} with the study of  the limit spaces of  
non-inradius collapse/convergence,
%
%the convergence where the inradii have a positive lower bound,  
following the basic approach employed in \cite{YZ:inrdius}.

In Section \ref{sec:gluing}, we describe the 
basic  properties of the gluing map $\eta_0:C_0\to N_0$, which will be 
needed in later sections.

The study of the infinitesimal structure of 
$N$ starts from 
Section \ref{sec:int-ext}.
In Section \ref{sec:int-ext}, we 
give the proof of Theorem \ref{thm:alex}.

In Section \ref{sec:Isometric},
we first construct the isometric involution $f_*$ on 
$\Sigma_p(C_0)$ with $\eta_0(p)\in N_0^1$
%%%
such that the quotient map
$\Sigma_p(C_0)\to\Sigma_p(C_0)/f_*$ coincides with the differential 
$d\eta_0:\Sigma_p(C_0)\to\Sigma_x(X_0)$.
Then we introduce the notions of cusps and 
boundary points of $X_0$.

In Section \ref{sec:inf-SC},  we prove
Theorem \ref{thm:S1extremal}, where the notion 
of almost parallel domains plays an important role.

%%%%
Section \ref{sec:local-str-S1} is a key section of the paper. Towards the proof of Theorem \ref{thm:S1-nontrivial-f*},
we refine the notion of almost parallel domains used in Section \ref{sec:inf-SC}, and 
develop the geometry of almost parallel domains.
This contains a new and key idea in the present paper.

In Section \ref{sec:dimension}, we discuss the Hausdorff dimensions of the metric singular set  and the boundary singular sets of 
$N_0$, and prove Theorems \ref{thm:dim(metric-sing)} and \ref{thm:nowhere}.

We provide several examples, which plays  important 
roles throughout  the paper.

In the continuation \cite{YZ:partII}  of 
the present paper, 
we describe the local structure of the limit spaces, and discuss some global convergence/collapsing in  
$\mathcal M(n,\kappa,\nu,\lambda, d)$  including 
Lipschitz homotopy stability.
% and diffeomorphism stability.

\psmall\n 
{\bf Acknowledgements}. The authors would like to thank Raquel Perales  for informing us 
the paper \cite{Perales2}.

 %%%%%%%%%%%%%%%%%%%%%%%%%%%%%%%%
\setcounter{equation}{0}

\section{Preliminaries}\label{sec:prelim}

%%%%%%%%%%%%%%%%%%%%% Alexandrov spaces   
\par\n
{\bf Notations and conventions} 
\psmall
We fix some notations and terminologies used in the paper.
\begin{itemize}
\item $\mathbb R^n_+$ denotes the $n$-dimensional Euclidean half space,
 \[
 \R^n_+=\{ (x_1,\ldots, x_n)\in\R^n\,|\, x_n\ge 0\} \,;
\] 
\item $I$ denotes a closed interval;
\item $D^n=\{ x\in \R^n\,|\,||x||\le 1\}$ and $D^n_+=D^n\cap \R^n_+$\,;
\item  $\mathbb S^n$ (resp. 
$\mathbb S^n_+$) denotes the $n$-dimensional unit 
sphere  (resp.  $n$-dimensional unit 
hemisphere);
\item $I_\ell$ denotes a closed interval of length $\ell$, 
and $S^1_{\ell}$ (resp. $S^1(r)$) denotes the circle of length $\ell$ (resp. of radius $r$) in 
$\C$ around the origin;
\item For a closed subset $K$ of a metric space $X$, let  $B(K,r)$ denote
the  closed metric ball around $K$ of radius $r$.
We also use the symbol $S(K,r)$ to denote the metric $r$-sphere around $K$;
$S(K,r)=\{ x\in X| |x, K|=r\}$. Sometimes, we write as $B^X(K,r)$ to 
emphasize that it is a ball in $X$;
\item For a metric space $\Sigma$, let $K(\Sigma)$ denote the Euclidean cone over $\Sigma\,;$
\item For a subset $A$ of a topological space, $\mathring A$ denotes the interior of $A\,;$
\item  $L(c)$ denotes the length of a curve $c\,;$
\item $o_i$  denotes a sequence of positive numbers
satisfying  $\lim_{i\to\infty}o_i=0\,;$
\item $\tau_{a_1,\ldots,a_m}(\e_1,\ldots,\e_n)$  denotes a function depends on $a_1,\ldots,a_m$ 
satisfying  $\lim_{\e_1,\ldots,\e_n\to0}
\tau_{a_1,\ldots,a_m}(\e_1,\ldots,\e_n)=0$.

\end{itemize}

\subsection{Alexandrov spaces and GH-convergence.}\label{ssec:Alex}
For basics of Alexandrov spaces, we refer to \cite{BGP},  \cite{BBI}, \cite{AKP}.

Let $X$ be a geodesic space.  
The distance between two points $p, q\in X$ is denoted
by $|p,q|$ or $|p,q|_X$.
Sometime we use $d=d_X$ to denote the distance of 
$X$.  For $\lambda>0$, we often use 
the symbol $\lambda X$ to denote the rescaling
$(X,\lambda d)$.

For a fixed real number $\kappa$, $M_\kappa^2$ denotes
the complete simply connected surface  with constant curvature
$\kappa$.
For  a geodesic triangle
$\tri pqr$ in $X$ with vertices $p$, $q$ and $r$, we
denote by $\tilde\tri pqr$ 
 a {\em comparison triangle} in 
$M_\kappa^2$ having
the same side lengths as the corresponding ones in $\tri pqr$.
Here we suppose that the perimeter of $\tri pqr$ is less than $2\pi/\sqrt\kappa$
if $\kappa> 0$.
The metric space $X$ is called an {\em Alexandrov space
with curvature $\geq \kappa$}
if each point of $X$
has a neighborhood $U$ satisfying the following:
For any geodesic triangle in $U$ with vertices
$p$, $q$ and $r$ and for any point $x$ on the segment $qr$,
we have $|p,x| \geq |\tilde{p}, \tilde{x}|$,
where $\tilde x$ is the point on the geodesic $\tilde{q}\tilde{r}$ corresponding to $x$.
From now on we assume that an Alexandrov space is always finite dimensional.

 Let $X$ be an $m$-dimensional Alexandrov space $X$ with curvature $\ge \kappa$. 
For two geodesics  $\alpha:[0,s_0]\to X$ and $\beta:[0,t_0]\to X$ starting from a point $x\in X$,
the {\em angle} between $\alpha$ and $\beta$ is defined by 
\[
      \angle(\alpha,\beta)  :=\lim_{s, t\to0} \tilde\angle \alpha(s)x\beta(t),
\]
where $\tilde\angle \alpha(s)x\beta(t)$ denotes the angle of a 
comparison triangle $\tilde \tri \alpha(s)x\beta(t)$ at the point 
$\tilde x$.
For $x,y,z\in X$, we denote by $\angle xyz$ (resp. $\tilde\angle xyz$) the angle
between  geodesics $yx$ and $yz$ at $y$ (resp. the geodesics $\tilde y\tilde x$ and 
$\tilde y\tilde z$ at $\tilde y$ in the comparison triangle 
$\tilde\triangle xyz=\triangle \tilde x \tilde y \tilde z$).
We often use the symbol $\tilde\angle^X xyz$
to emphasize that it is a comparison angle 
with respect to $X$. 
We denoted by $\Sigma_x'(X)$ the set of equivalent classes of geodesics $\alpha$ emanating from $x$,
where  $\alpha$ and  $\beta$ are called {\it equivalent}  if
 $\angle(\alpha,\beta)=0$. 
The {\em space of directions} at $x$, denoted by $\Sigma_x=\Sigma_x(X)$, is the completion of
$\Sigma_x'(X)$ with the angle metric, which is known to be an $(m-1)$-dimensional compact
Alexandrov space with curvature $\ge 1$.
A direction of minimal geodesic from $p$ to $x$ is also denoted by $\uparrow_p^x$ or $\dot\gamma_{p,x}(0)$. 
The set of all directions $\uparrow_p^x$ 
from $p$ to $x$ is denoted by 
$\Uparrow_p^x$.
A minimal geodesic $\gamma$ joining $p$
to $x$  has unit speed and is  parametrized on $[0, |p,x|]$
with $\gamma(0)=p$ unless otherwise stated.
 As indicated above, we denote a $\gamma$ by $\gamma_{p,x}$, $\gamma^X_{p,x}$ or simply $px$ as above. 
Such a $\gamma$ is often called 
$X$-minimal.
%%%%
The direction at $p$ represented by $\gamma$
is denoted by $\dot\gamma(0)$.
 For $v\in\Sigma_x(X)$, we denote by $-v$
the unique element of $\Sigma_x(X)$ such that
$\angle(v,-v)=\pi$ if it exists.
In this case, $v$ and $-v$ are called {\it opposite}
to each other.

The {\em tangent cone} at $x\in X$, denoted by $T_x(X)$, 
is the Euclidean cone $K(\Sigma_x)$ over $\Sigma_x$.
It is known that $T_x(X)$ coincides with 
the Gromov-Hausdorff limit
$\lim_{r\to0}\left( \frac{1}{r}X,x \right)$
(see the later paragraph for the Gromov-Hausdorff convergence).

For a closed subset $A$ of $X$ and $p\in A$, the space of directions  $\Sigma_p(A)$
of $A$ at $p$ is defined as the set of 
all $\xi\in \Sigma_p(X)$ which can be written as the limit of 
directions $\uparrow_p^{p_i}\in\Sigma_p(X)$ from $p$ to points $p_i\in A$ with $|p,p_i|\to 0$:
\[
         \xi =\lim_{i\to\infty} \uparrow_p^{p_i}.
\]

A point $x\in X$ is called {\em regular} if $\Sigma_x$ is isometric to the unit sphere  $\mathbb{S}^{m-1}$.
Otherwise we call $x$ {\em singular}. We denote by $X^{\rm reg}$ (resp. $X^{\rm sing}$)
the set of all regular points (resp. singular points) of $X$. 

For  $\delta>0$ and $1\le k\le m$, a system of $k$ pairs of points, $\{(a_i,b_i)\}_{i=1}^k$ is called a
$(k,\delta)$-\emph{strainer} at $x\in X$ if it satisfies
\begin{align*}
   \tilde\angle a_ixb_i > \pi - \delta, \quad &
                \tilde\angle a_ixa_j > \pi/2 - \delta, \\
      \tilde\angle b_ixb_j > \pi/2 - \delta,\quad &
                \tilde\angle a_ixb_j > \pi/2 - \delta,
\end{align*}
for all $1\le i\neq j\le k$.
If  $x\in X$ has a $(k,\delta)$-strainer, then 
it is called $(k,\delta)$-strained.
If  $x\in X$  is $(m,\delta)$-strained, it is called 
a {\em $\delta$-regular} point.
We call $X$ {\em almost regular} if any point of $X$ is $\delta_m$-regular 
for some $\delta_m \ll 1/m$.
It is known that a small neighborhood of any almost regular point is  biLipschitz homeomorphic to an open subset in $\mathbb R^m$  with biLipschitz constant close to $1$ (see \cite{BGP}).

The boundary $\partial X$ is  inductively defined as the set of points $x\in X$ 
such that $\Sigma_x$ has non-empty boundary $\partial\Sigma_x$.
We denote by $D(X)$ the double of $X$, which is also
an Alexandrov space with curvature $\ge \kappa$
(see \cite{Pr:alexII}). By definition, $D(X)=X\cup_{\partial X} X$, where two copies of $X$ are glued along their boundaries.
A boundary point $x\in\partial X$ is called {\em $\delta$-regular} if $x$ is $\delta$-regular
in $D(X)$. We say that $X$ is {\em almost regular with almost regular boundary}
if every point of $D(X)$ is $\delta$-regular with $\delta\ll 1/m$.

In Section \ref{sec:non-inradius}, we need the following result on the dimension of the interior singular point sets.

\begin{thm} [\cite{BGP}, \cite{Per}, cf. \cite{OS}]   \label{thm:dim-sing}
$$
            \dim_H(X^{\rm sing}\cap {\rm int}X) \le n-2,  \,\,\dim_H(\partial X)^{\rm sing} \le n-2,
$$  
where $(\partial X)^{\rm sing} :=D(X)^{\rm sing}\cap \pa X$.
\end{thm}
 \pmed
A subset $E$ of an Alexandrov space $X$ is called {\em extremal} (\cite{PtPt:extremal})
if every distance function $f={\rm dist}_q$ with $q\in M\setminus E$ has the property
that if $f|_E$ has a local minimum at $p\in E$, then $df_p(\xi)\le 0$
for all $\xi\in\Sigma_p(X)$. Extremal subsets possess quite important properties.

Suppose that  a compact group $G$ acts on $X$ as isometries. Then the quotient space
$X/G$ is an Alexandrov space (\cite{BGP}). Let $F$ denote the set of $G$-fixed points.

\begin{prop} [\cite{PtPt:extremal}] \label{prop:fixed-ext}
$\pi(F)$ is an extremal subset of $X/G$, where $\pi:X\to X/G$ is the projection.
\end{prop}

Boundaries of  Alexandrov spaces are typical examples of extremal subsets. 

\pmed\n 
{\bf Gromov-Hausdorff convergence.}\,

For metric spaces $X$ and $X'$,
a not necessarily continuous map $f:X\to X'$ is called
$\e$-approximation if 
\begin{itemize}
\item $||f(x), f(x')|-|x,x'|| < \epsilon$ for all $x, x' \in X\,;$ 
\item $f(X)$ is $\e$-dense in $Y$.
\end{itemize} 
The Gromov-Hausdorff distance $d_{GH}(X,Y)<\e$
iff there are $\e$-approximations $\varphi:X\to Y$ and 
$\psi:Y\to X$ (see \cite{GLP} for more details).
 
For compact subsets $A_1,\ldots,A_k\subset X$
and $A'_1,\ldots,A'_k\subset X'$,
the Gromov-Hausdorff distance 
$d_{GH}((X,A_1,\ldots,A_k),(X',A'_1,\ldots,A'_k))<\e$ 
iff there are $\e$-approximations $\varphi:X\to X'$ and 
$\psi:X'\to X$ such that the restrictions
$\varphi|_{A_i}$ and $\psi|_{A'_i}$ give
$\e$-approximations between $A_i$ and $A'_i$ for any
$1\le i\le k$.
%%%%
See \cite{GoChe}  for more detail on the 
Gromov-Hausdorff convergence of metric pairs.

\begin{thm} [\cite{Pr:alexII}, \cite{Per}, cf.\cite{Kap:stab}]      \label{thm:stability}
If a sequence $X_i$ of $n$-dimensional compact Alexandrov spaces 
with curvature $\ge \kappa$ Gromov-Hausdorff converges to an $n$-dimensional compact Alexandrov space $X$,
then $X_i$  is homeomorphic to $X$ for large enough $i$.
\end{thm}

%%%%%%%%%%%%%%%%%%%%%%%%%%%%% Gluing    %%%%%%%%%%%%%%%%%%%%%%%%%%%%%%%%%

\subsection{Manifolds with boundary and gluing}\label{ssec:gluing}

A Riemannian manifold with boundary is not necessarily an Alexandrov space.
Based on Kosovski$\breve{i}$'s gluing theorem \cite{Kos}, 
assuming \eqref{eq:curv-assum2}, Wong (\cite{wong0}) carried out a gluing of $M$ and 
 warped  cylinders along their boundaries  in such a way that 
the resulting manifold  becomes an Alexandrov space having totally geodesic boundary.
Here  the upper bound $\Pi_{\partial M}\le\lambda$
was used only to have the lower bound on $K_{\pa M}$ via the Gauss equation.
This is a key in this construction,
and is used to have the lower sectional curvature bound on the cylindrical part.
%%%
Therefore under the new weaker curvature condition
\eqref{eq:curv-assum},
we  can still carry out the gluing construction
without  additional argument,
which we recall below.
Note also that from the Gauss equation,  
$K_M \ge \kappa$ and 
$\Pi_{\partial M}\ge 0$ imply
$K_{\pa M} \ge\kappa$.
%%%%%%%%%

%%%%%%%%%%%%%%%%%%%%

Suppose $M$ is an $n$-dimensional complete  Riemannian manifold satisfying 
\eqref{eq:curv-assum}.
Then for arbitrarily  $t_0>0$ and  $0<\e_0<1$
there exists a monotone non-increasing function  $\phi: [0,t_0]\to (0,1]$ 
satisfying
\beq       \label{eq:phi}
\left\{
 \begin{aligned}
     \phi''(t)+K\phi(t)\leq0, \,\,&\phi(0)=1, \,\,\phi(t_0)=\e_0,\\
     -\infty<\phi'(0)\leq-\lambda,\,\, &\phi'(t_0)=0,
 \end{aligned}
   \right.
\eeq
for some constant  $K=K(\lambda,\ve_0,t_0)$.
Now consider the warped product metric on  
$[0, t_0]\times\partial M$ defined by 
\[
   g(t,x)=dt^2+\phi^2(t)g_{\partial M}(x)
\]
where $g_{\partial M}$ is the Riemannian metric of $\partial M$ induced from 
that of $M$. We denote by   
$C_M:=[0, t_0]\times_\phi\partial M$  the warped product.
It follows from the construction that
\begin{itemize}
 \item the sectional curvature $K_{C_M}$ of $C_M$ is greater than a constant $c(\nu,\lambda,\e_0, t_0);$
 \item the second fundamental form of 
 $\{ t\}\times \partial M$ is given by 
   \[
          II^{C_M}_{\{ t\}\times \partial M}(V,W)=\frac{\phi'(t)}{\phi(t)}g(V,W),
\]
for vector fields $V, W$ on $\{ t\}\times \partial M$. In particular, we have 
\[
  II^{C_M}_{\{ 0\}\times \partial M}\geq \lambda,\,\,\,
 II^{C_M}_{\{t_0\}\times \partial M}\equiv0.
\]
\end{itemize} 
\pmed
\n
We can glue $M$ and $C_M$
along $\partial M$ and  $\{0\}\times \partial M$.
From the construction, we have 
\beq  \label{eq:gluing-cond}
\begin{cases}
\begin{aligned}
&K_M\ge\kappa, \,\,\, K_{C_M}\ge c(\nu,\lambda,\e_0, t_0), \\ 
&II^M_{\pa M}+II^{C_M}_{\{ 0\}\times \partial M}\ge 0
\,\,\, \text{along \,$\pa M=\{ 0\}\times \partial M$}.
\end{aligned}
\end{cases}
\eeq
It follows from \eqref{eq:gluing-cond} 
and \cite{Kos} that the resulting space
\[
    \tilde M := M \cup_{\pa M} C_M
\]
carries the structure of differentiable manifold of class $C^{1,\alpha}$ with $C^0$-Riemannian metric.
%(\cite{Kos}).
Obviously  $M$ is diffeomorphic to $\tilde M$.
%Moreover we have the following.

\begin{prop} [\cite{wong0}] \label{prop:extendAS}
For  any $M$  satisfying \eqref{eq:curv-assum}, 
the following holds$:$
\begin{enumerate}
 \item $\tilde M$  is an  Alexandrov space with curvature 
  $\ge \tilde{\kappa}$,  
         where $\tilde{\kappa}=\tilde{\kappa}(\kappa,\nu, \e_0,\lambda,t_0);$
 \item the extrinsic metric $M^{\rm ext}$ of $M$ in $\tilde M$ is $L$-bi-Lipschitz homeomorphic to $M$ 
         for the uniform constant $L=1/\e_0;$
\item  $\diam(\tilde M)\le \diam(M) + 2t_0$.
\end{enumerate}
\end{prop}

%%%%%%%%%%%%%%%

Let $\partial^\alpha M$ be a component of $\pa M$.
We denote by 
$(\pa^\alpha M)^{\rm int}$ the intrinsic  length metric
of $\pa^\alpha M$.

\begin{prop} [\cite{wong0}]\label{prop:cpn-diam}
For any $M\in\ca M(n,\kappa,\nu,\lambda,d)$,
we have the following:
\begin{enumerate}
 \item There exists a constant $D=D(n,\kappa,\nu,\lambda,d)$ such that  any boundary component $\pa^\alpha M$ has intrinsic diameter bound
    \[
             \diam((\pa^\alpha M)^{\rm int})\leq D;
    \]
 \item $\partial M$ has at most $J$ components, where $J = J(n, \kappa,\nu, \lambda, d)$.
\end{enumerate}
\end{prop}

It follows from  Proposition \ref{prop:cpn-diam} that  the set of boundary components of $M$ when 
$M$ runs over the family $\ca M(n,\kappa,\nu,\lambda,d)$ 
forms a precompact family with respect to the Gromov-Hausdorff distance.
%%%%%%%%%%%%%%%%%%%%%%%%%

We denote by $\ca M_b(n,\kappa,\lambda,d)$ 
the set of all $n$-dimensional compact Riemannian manifolds $M$ with boundary
satisfying \eqref{eq:curv-assum2} and 
$\diam(M)\le d$.
%%%
For the construction of examples, it is convenient to work in the family 
$\ca M_b(n,\kappa,\lambda,d)$  rather than 
$\ca M(n,\kappa,\nu,\lambda,d)$  to check the 
geometric bounds. Therefore  in what follows 
we mainly provide examples of convergence
in $\ca M_b(n,\kappa,\lambda,d)$.

%%%%%%%%%%
The notion of warped product also works for geodesic spaces.

Let $X$ and $Y$ be geodesic spaces, and let 
 $\phi:X\to\mathbb R_+$ be a positive continuous 
function. Then the warped product $X\times_{\phi}Y$ is defined as follows (see \cite{wong1}).
For a curve $\gamma=(\sigma,\nu):[a,b]\to X\times Y$, the length of $\gamma$ is 
defined as  
\[
 L_{\phi}(\gamma)=
   \sup_{|\Delta|\to 0} \sum_{i=1}^k
 \sqrt{|\sigma(t_{i-1}),\sigma(t_i)|^2 +
      \phi^2(\sigma(s_i)) |\nu(t_{i-1}),\nu(t_i)|^2},
\]
where $\Delta:a=t_0<t_1<\cdots<t_k=b$ and $s_i$ is any element of $[t_{i-1},t_i]$.
The  warped product $X\times_{\phi}Y$ is defined as the topological space 
$X\times Y$ equipped with the length metric induced from $L_{\phi}$.

\begin{prop}$($\cite[B.2.6]{wong1}$)$
\label{prop:warped}
Let $Y_i$ be a convergent sequence of geodesic  spaces.
If $X$ is a compact geodesic space, then we have 
\[
                  {\lim}_{GH}(X \times_{\phi} Y_i) = 
X\times_{\phi} ({\lim}_{GH}Y_i).
\]
\end{prop}
\pmed

%%%%%%%%Basic structure  %%%%%%%%%%%%%%%%%%%%%%%%%%%%%%%%%%%%%%%%

\setcounter{equation}{0}

\section{Basic structure of limit spaces} \label{sec:non-inradius}

In this section, we consider the situation that 
$M_i \in \ca M(n,\kappa,\nu,\lambda, d)$ converges to a compact geodesic space $N$ 
while ${\rm inrad}(M_i)$ has a positive lower bound. 
%%% 
Let $\tilde M_i$ be the extension of $M_i$ as constructed in Section \ref{ssec:gluing}.
Passing to a subsequence, we may assume that 
\begin{enumerate}
 \item $\pa M_i$ converges to $N_0$ under the convergence $M_i\to N;$
 \item $(\tilde M_i, M_i,\partial M_i)$ converges to a triple $(Y,X,X_0)$ with $Y\supset X\supset X_0$
in the sense that $\tilde M_i$ converges to $Y$ and 
$M_i$ (resp. $\partial M_i$)  also converges to $X$ (resp. to $X_0$), under the convergence  $\tilde M_i\to Y$  (in the sense of Section 
\ref{ssec:Alex}).
\end{enumerate}
Note that $Y$ is an Alexandrov space with curvature $\ge \tilde\kappa$.

\pmed
 In  view of  Proposition \ref{prop:extendAS}, 
passing to a subsequence, we may assume that $C_{M_i}$ converges 
to some compact Alexandrov space $C$ with curvature $\ge c(\nu,\lambda,\ve_0, t_0)$. 
%
%\tilde\kappa=\tilde\kappa(\kappa,\nu,\lambda,t_0)$. 
Here  $C_{M_i}$ is not necessarily connected, and therefore 
the convergence $C_{M_i}\to C$ should be understood componentwisely. 
From Proposition \ref{prop:warped}, we have  
\beq \label{eq:defC}
                    C=[0, t_0]\times_{\phi} C_0, \,\,   C_0=\lim_{i\to\infty} (\pa M_i)^{\rm int}.
\eeq
It follows that $C_0$ is an Alexandrov space with curvature $\ge \nu$.

For simplicity we set  
\[
     C_0:=\{ 0\}\times C_0,\,\,  C_t:=\{ t\}\times C_0\subset C,
\]
and 
\[
       {\rm int}\, X:= X\setminus X_0.
\]
Let $C_{M_i}^{\rm ext}$ denote the extrinsic metric 
on $C_{M_i}$
induced from $\tilde M_i$, which is defined as 
the restriction of the metric of $\tilde M_i$.
Since the identity map $\iota_i: C_{M_i}\to C_{M_i}^{\rm ext}$ is 1-Lipschitz,
we have a 1-Lipschitz map $\eta: C\to Y$ in the limit.
Note that $\eta:C\to Y\setminus {\rm int}\, X$ is surjective.

From now on,  we consider 
\[
        \eta_0: = \eta|_{\{ 0\}\times C_0} : C_0\to X_0,
\]
which is a surjective $1$-Lipschitz map with respect to the extrinsic metrics of 
$C_0$ and $X$, and hence with respect to the intrinsic metrics, too.
Note that $\eta_0$ is the limit of $1$-Lipschitz map $\eta_{0,i}:(\pa M_i)^{\rm int} \to (\pa M_i)^{\rm ext}$, where $(\pa M_i)^{\rm ext}$ denotes the extrinsic  
metric induced from $\tilde M_i$.

Next we recall some basic facts already shown in \cite{YZ:inrdius}.

\begin{lem} [{\cite[Lemma 3.1]{YZ:inrdius}}]  \label{lem:loc-isom}
The map $\eta:C\setminus C_0\to Y\setminus X$ is a bijective local isometry.
\end{lem}

Let $\tilde\pi:C\to C_0$ and $\pi:Y\to X$ be the canonical projections.
To be precise, $\pi$ is defined as
\[
          \pi(y) = \eta_0\circ \tilde \pi(\eta^{-1}(y)), \quad (y\in Y\setminus X).
\]
For any $p\in C_0$ with $x:=\eta_0(p)$, let 
$$
\tilde\gamma_p^{+}(t) := (p,t), \quad
\gamma_x^+(t):=\eta(\tilde\gamma(t)), \quad t\in [0,t_0].
$$
We call the geodesic $\tilde\gamma_p^{+}$
 (resp. $\gamma_x^+$) the {\it perpendicular to} $C_0$  at $p$
(resp. a perpendicular to $X_0$ at $x$).
The maps $\tilde \pi$ and $\pi$ are the projections along perpendiculars.
We also call 
$$
\tilde\xi_p^+:=\dot{\tilde\gamma}_p^+(0), \quad
\xi_x^+:=\dot\gamma_x^+(0)
$$
the perpendicular directions at $p$ and  $x$ respectively.

Obviously, we have  $|\eta(p,t), X|=t$  for  every $(p,t)\in C$.  We set
\beq \label{eq:CtY}
            C_t^Y:= \eta(C_t).
\eeq

The projections along the perpendiculars also provides
a Lipschitz strong deformation of $Y$ to $X$. Thus we have 

\begin{prop}[\cite{wong0}] \label{prop:retractionYX}
$Y$ has the same Lipschitz homotopy type as $X$.
\end{prop}

\psmall
For the multiplicities of the gluing map $\eta_0$, we have 

\begin{lem}  [{\cite[Lemma 3.3]{YZ:inrdius}}]\label{lem:preimage}
For every $x\in X_0$, we have the following:
\begin{enumerate}
 \item $\#\eta_0^{-1}(x)\leq 2;$ 
 \item Suppose $\#\eta_0^{-1}(x)=2$, and take $p_+,p_-\in C_0$ with 
$\eta_0(p_\pm)=x$.  
 If $t\le \phi(t_0)|p_+,p_-|/2$, then 
\[
    |\eta(p_+,t),\eta(p_-,t)|=2t.
\]
Therefore, $\Sigma_x(Y)$ is isometric to the spherical suspension with the 
two vertices $\{ \xi_x^+, \xi_x^-\}$, where 
 \[
     \xi_x^+ :=\uparrow_x^{\eta(p_+, t_0)}, \quad
   \xi_x^- :=\uparrow_x^{\eta(p_-, t_0)}.
 \]
\end{enumerate}
\end{lem}
\pmed
 When $\#\eta_0^{-1}(x)=2$, we use the notation 
\[
   \gamma_x^+(t):=\eta(p_+,t),\quad
\gamma_x^-(t):=\eta(p_-,t)
\]
to denote the two perpendiculars at $x$.

Let $X^{\rm int}$ denote the intrinsic metric of $X$.  We now state the relations between $(N, N_0)$ and $(X, X_0)$. The following lemma is 
immediate from Proposition \ref{prop:extendAS}(2).

\psmall
\begin{lem}  \label{lem:XbiLipX}
The canonical map $X\to X^{\rm int}$ is an
$L$-bi-Lipschitz homeomorphism.
\end{lem}

\begin{prop} [{\cite[Proposition 3.9, Remark 3.12]{YZ:inrdius}}]\label{prop:intrinsic}
There exists an isometry $f: X^{\rm int} \to N$ such that 
$f(X_0) =N_0$.
\end{prop}

From now on, we often identify 
$N=X^{\rm int}$ and  $N_0=X_0$,
or shortly $N=X$ when no confusion 
arises.

\psmall
We always assume $m=\dim Y$.
From a uniform positive lower bound for  ${\rm inrad}(M_i)$, we immediately have 
\beq
      \dim X = m,   \label{eq:dimX} \,\,\, {\rm int}\, X \neq\emptyset.
\eeq
See also Lemma \ref{lem:dimNdimN0}.
It is also easy to see 

\begin{lem} \label{lem:bdy}
$X_0$ coincides with the topological boundary $\partial X$ of $X$ in $Y$.
\end{lem}

Let $X\cup_{\eta_0}([0,t_0]\times_{\phi} C_0)$ denote the geodesic space obtained by the result of
gluing of the two geodesic spaces $X^{\rm int}$ and $[0,t_0]\times_{\phi} C_0$ by the map 
$\eta_0:\{ 0\}\times C_0\to X^{\rm int}_0$.
We have a description of $Y$:

\begin{lem} [{\cite[Proposition 3.11]{YZ:inrdius}}]   \label{lem:YX0}
 $Y$ is isometric to the geodesic space 
\[
      X \cup_{\eta_0}([0,t_0]\times_{\phi} C_0),
\]
where $(0,x)\in \{ 0\}\times C_0$ is identified with $\eta_0(x)\in X_0$ for each $x\in C_0$.  
\end{lem}
\pmed

%%%%%%%%%%%%%%%%%%%%%%%%%%%%%%%

\begin{defn}
For $k=1,2$, set
\[
    X_0^k:=\{ x \in X_0\,|\, \# \eta_0^{-1}(x)=k\,\}, \,\,\, C_0^k:=\eta_0^{-1}(X_0^k).
\] 
By Proposition \ref{prop:intrinsic}, we often identify $N_0^k \equiv X_0^k$.
We denote by ${\rm int}\, X_0^k$ the interior of $X_0^k$ in $X_0$.
Let $\partial X_0^k$ and $\partial C_0^k$ denote the topological boundaries of 
$X_0^k$ and $C_0^k$ in $X_0$ and $C_0$ respectively.
We denote by  ${\tilde{\mathcal S}}$ the topological boundary of
$C_0^1$ and $C_0^2$  in $C_0$, and by $\mathcal S$ the
topological boundary of $X_0^1$ and $X_0^2$ in $X_0$:
\begin{align*}
       {\tilde{\mathcal S}}:=\partial C_0^1=\partial C_0^2, \,\, \,\,\, \mathcal S:=\partial X_0^1=\partial X_0^2.
\end{align*}
In view of $N=X^{\rm int}$, we often identify 
$N_0 =X_0,\,\,\, N_0^k = X_0^k$.
We call  $\mathcal S$ the {\it topological boundary singular set} of $X$ or $N$.
For each $k=1,2$, a point of 
\[
                \mathcal S^k:=\mathcal S\cap X_0^k
\]
is called a {\it topological boundary singular point of $X$ of type $k$.}
We also set 
\[
    \tilde{\mathcal S}^k :=\tilde{\mathcal S}\cap C_0^k.
\]
\end{defn}
 It is easy to verify $\tilde{\ca S}^1=\eta_0^{-1}(\ca S^1)$. The corresponding statement of $\tilde{\ca S}^2$
also holds true (see Lemma \ref{lem:lift-S2}).

\begin{rem} \upshape
In the case of inradius collapse, since the subset $N_0^2$ is open in $X_0=X$ (see Theorem
\ref{thm:inradius-collapse}), there exists no boundary singular point of type 2. However, in the case of non-inradius collapse, as seen in later examples, both the boundary singular sets $\ca S^1$ and  $\ca S^2$ may arise.
\end{rem}
           
\pmed
The results of Section 4.2 of \cite{YZ:inrdius}
still hold if $X$ is replaced by $X_0$ by the same arguments.
We show those results with only rough outline 
if necessary.

The following lemma is immediate from
the warped product structure of $C$.

\begin{lem} [{\cite[Lemma 4.5]{YZ:inrdius}}]
For any $p\in C_0$, 
%let $\tilde\gamma_p^{+}(t) = (p,t)$.Then
$\Sigma_p(C)$ is isometric to the half-spherical suspension $\{ \dot{\tilde\gamma}_p^{+}(0) \}*\Sigma_p(C_0)$.
\end{lem}

For $x\in X_0$, let $\Sigma_x(X_0)$ and $\Sigma_x(X)$ be defined as the closed subsets of $\Sigma_x(Y)$ defined in Section \ref{ssec:Alex}.

\begin{lem} \label{lem:not-perp}
For any $x\in X_0$, 
and for any $\xi\in\Sigma_x(Y)$
with $\angle(\xi_x^+,\xi)<\pi/2$ and $Y$-geodesic $\gamma$ with
$\dot\gamma(0)=\xi$, we have the following$:$
\begin{enumerate}
\item $\gamma(t)\in Y\setminus X$ for all small enough $t>0\,;$
\item  The curve $\sigma(t)=\pi\circ \gamma(t)$ defines a 
unique direction $v\in\Sigma_x(X_0)$ in the sense
$v=\lim_{t\to 0} \uparrow_x^{\sigma(t)}$, and satisfies 
\[
 \angle(\xi_x^{+},\xi)+\angle(\xi, v)=\angle(\xi_x^{+}, v)=\pi/2\,;
\]
\item Under the convergence $(\frac{1}{t} Y, x)\to (K_x(Y),o_x)$ as 
$t\to 0$, $\sigma(t)$ converges to a minimal geodesic
$\sigma_\infty$ from $o_x$ in the direction $v$.
\end{enumerate}
\end{lem}
\begin{proof} 
(1) easily follows from the first variation formula.
(2) and (3) follow from \cite[Lemma 4.6]{YZ:inrdius}
\end{proof}

\begin{lem} \label{lem:double-sus}
If $x\in X_0^2$, then $\Sigma_x(Y)$ is isometric to the spherical suspension $\{ \xi_x^+, \xi_x^-\}*\Sigma$, where $\xi_x^+, \xi_x^-$ are as in 
Lemma \ref{lem:preimage} (2), and $\Sigma$
is isometric to both $\Sigma_x(X_0)$ and $\Sigma_x(X)$. 
\end{lem}
\begin{proof}
 From Lemma \ref{lem:preimage}(2), we have 
$\Sigma_x(Y)=\{ \xi_x^+, \xi_x^-\}*\Sigma$ for some $\Sigma$.
By Lemma \ref{lem:not-perp}, 
$\Sigma=\{ v\in\Sigma_x(Y)\,|\, \angle(v,\xi_x^+)=\pi/2\}=\Sigma_x(X_0)$.
Since $\Sigma=\{ v\in\Sigma_x(Y)\,|\, \angle(v,\xi_x^-)=\pi/2\}$,
we also have $\Sigma=\Sigma_x(X)$.
\end{proof}

For $x\in X_0^1$, we consider the {\it radius}
of $\Sigma_x(Y)$ viewed from $\xi_x^+$ defined as 
\[
{\rm rad}(\xi_x^+):=\sup \{ d(\xi_x^+, \xi)\,|\,\xi\in\Sigma_x(Y)\}.
\]

\begin{lem} \label{lem:single-interior}
For  any $x\in X_0^1$ with  ${\rm rad}(\xi_x^+)>\pi/2$,  let $\xi\in\Sigma_x(Y)$ be  chosen in such a way  that  $\angle(\xi_x^+,\xi)>\pi/2$ and 
the geodesic
$\gamma$ in the direction $\xi$ is defined.
Then we have 

\begin{enumerate}
\item $x\in {\rm int}\,X_0^1\,;$
\item 
$\gamma((0,\e))\subset {\rm int} X$
for some $\e>0$.
\end{enumerate}
\end{lem}

\begin{proof} 
%Take $\xi\in\Sigma_x(Y)$ with $\angle(\xi_x^+,\xi)>\pi/2$.
Note that $\xi\in\Sigma_x(X)$. 
%We may assume that there is a geodesic
%$\gamma$ in the direction $\xi$.

\n
(1)\, Suppose  $x\notin {\rm int}\,X_0^1$. Then we have a sequence $w_m\in X_0^2$
with $\lim_{m\to\infty} w_m = x$.
Fix sufficiently small $\e>0$ such 
that $\tilde\angle \gamma_x^+(\e)x \gamma(\e)>\pi/2$.
Choose  perpendiculars $\gamma_{w_m}^{\pm}$ at $w_m$.
Note that 
$$
\lim_{m\to\infty} \tilde\angle\gamma_{w_m}^{\pm}(\e)w_m\gamma(\e)
=\tilde\angle \gamma_x^+(\e)x \gamma(\e)>\pi/2,
$$
which is a contradiction to 
$\angle(\gamma_{w_m}^{+}, \gamma_{w_m}^{-})=\pi$.

\n
(2)\,
Set $\hat \gamma:=\gamma\setminus\{ x\}$.
If the conclusion does not hold, 
we would have points $z_n\in\hat\gamma \cap X_0$
converging to $x$.
Let $\gamma_{z_n}^+$ be a perpendicular at $z_n$, and set $s_n:=|x,z_n|_Y$.
From the lower semicontinuity of angles, we have 
$\angle_{z_n}(\dot\gamma^+_{z_n}(0), \dot\gamma(s_n))> \pi/2+\e$ for large enough $n$, 
where $\e>0$ is independent of $n$, 
and hence
\[
     \wangle \gamma^+_{z_n}(s) z_n x \le
 \angle_{z_n}(\dot\gamma^+_{z_n}(0), -\dot\gamma(s_n))< \pi/2-\e
\]
for a fixed $s>0$.
Thus for large enough $n$,  we get
\[
  |x,\gamma_{z_n}^+(s)|<|z_n, \gamma_{z_n}^+(s)|=|\gamma_{z_n}^+(s),X|,
\]
which is a contradiction since $x\in X_0$.  
\end{proof}

 From Lemma \ref{lem:single-interior}, 
 we immediately have the following.

\begin{lem}\label{lem:intX01}
If $M_i\in \ca M(n,\kappa,\nu,\lambda,d)$ converges to a compact geodesic space 
$N$ while ${\rm inrad}(M_i)$ has a positive lower bound independent of $i$,
then ${\rm int}\, N_0^1$ is nonempty.
\end{lem}
\begin{proof}
For any $z\in {\rm int} \,X$, take $x\in X_0$ such that 
$|z,x|=|z, X_0|$. Lemma \ref{lem:single-interior}
implies $x\in {\rm int} \,X_0^1$.
\end{proof}

\pmed

Recall that  $X_0^k\setminus \mathcal S^k$ is open in $X_0$.
Moreover using Lemma \ref{lem:single-interior}, we have the following. 

\begin{lem} \label{lem:X02}
 ${\rm int}\, X_0^2=X_0^2\setminus \mathcal S^2$ is open in $X$. In particular, we have 
 \begin{enumerate}
 \item for every $x\in X_0^2\setminus \mathcal S^2$, there is an $\epsilon>0$ such that
 $\dim_H X \cap B(x, \epsilon) = m-1;$
 \item  $\mathcal S$ is empty if and only if $X_0=X_0^1$.
\end{enumerate}
\end{lem} 
\begin{proof}  
For any $x\in {\rm int}\, X_0^2$,  take 
an open neighborhood 
$U$ of $x$ in  $X_0$ such that $U\subset  {\rm int}\, X_0^2$. 
Since $\pi^{-1}(U)$ is the union of all perpendiculars 
$\gamma_z^+$ with $z\in U$, we certainly have 
$\pi^{-1}(U)\cap X=U$.
It suffices to show that $\pi^{-1}(U)$ is open in $Y$.
For any $y\in\pi^{-1}(U)$, let $z:=\pi(y)$.

If $y\in Y\setminus X$, then the $\pi$-image of a 
small neighborhood of $y$ in $Y$ is contained in $U$
by the continuity of $\pi:Y\setminus {\rm int}\,X\to
X_0$.

Suppose $y\in X$. Then we have $y=\pi(y)\in U$.
We assert that a small neighborhood of $y$ in $Y$
is contained in $\pi^{-1}(U)$.
Otherwise, there is a sequence $y_i$ converging to $y$
such that $\pi(y_i)\notin U$.
Since $\pi(y_i)\to y$,  in view of the above argument, $y_i$ must be in ${\rm int}\,X$. Let $z_i\in X_0$ be a nearest point of $X_0$ from 
$y_i$. Since $z_i\to y$, $z_i\in U$ for large enough $i$.
However $z_i\in {\rm int}\,X_0^1$ by Lemma \ref{lem:single-interior}. This is a contradiction.

Now $(1)$ is obvious. To show $(2)$, suppose $\ca S$ is empty.
If $X_0^2$ is not empty, then $X_0=X_0^2$. It follows from the above that 
$X=X_0$. This is a contradiction to \eqref{eq:dimX}.
\end{proof}  

\psmall

\begin{rem} \label{rem:X1-cust}
As the following example shows, 
the condition  $x\in {\rm int}\,X_0^1$
does not imply   
${\rm rad}(\xi_+)>\pi/2.$
See also Example \ref{ex:X01cosh}.
\end{rem}

\begin{ex} \label{ex:cusp} 
Let $I:=[0,2]$, and 
let $g:I\to \R_+$ be a smooth function satisfying the following properties:
\begin{enumerate}
 \item $g(1-x)=g(1+x)$ for $x\in[0,1],$
 \item $g^{-1}(0)=\pa I$, $g'=0$ on $\pa I$.
\end{enumerate}
Let $B_1$, $B_2$ be two copies of $\{(x,y)|\, x\in I,0\leq y\leq g(x)\}$. Let $A_\epsilon$ be the intersection of $\pa B(I\times\{(0, 0)\}, \e)\subset\mathbb{R}^3$ and the half space $z\leq0$ of $\mathbb{R}^3$.  
Then we glue $A_\epsilon$, $B_1$ and $B_2$ together by gluing $I\times\{ (-\epsilon,0)\}\subset A_\epsilon$ with $I\times\{0\}\subset B_1$,
and identifying $I\times\{ (\epsilon,0)\}\subset A_\epsilon$ with $I\times\{0\}\subset B_2$ respectively. Obviously, the resulting space, denoted as $M_\epsilon$, is in $\mathcal{M}_b(2,0,\lambda,3)$ for some $\lambda$.
Let $N$ be the gluing of $B_1$ and $B_2$ along $[0,2]\times\{0\}$. 
Then $M_\epsilon$ converges to $N$ as $\epsilon\to0$. Note that $N_0=N_0^1$.
At any cusp $p$ of the boundary $N_0$ of $N$,
$\Sigma_p(X)=\Sigma_p(X_0)$ is a point,
and $\Sigma_p(Y)$ is a circle of length $\pi$.
In particular, ${\rm rad}(\xi_+)=\pi/2$.
\end{ex}

\pmed
%%%%%%%%%%%%%%%%%%%%
\begin{center}
\begin{tikzpicture}
[scale = 0.5]

\filldraw[fill=lightgray] 
%[fill=gray, opacity=.1] 
 (-4,0.3) to [out=270, in=180] (-3,0)
 to [out=0, in=180] (0, 1)
to [out=0, in=180] (3,0)
to [out=0, in=270] (4,0.3)
 to [out=270, in=0] (3,-0.4)
to   [out=180, in=0] (-3,-0.4)
to   [out=180, in=270] (-4,0.3);

\filldraw[fill=lightgray, opacity=.1] 
 (-4,0.3) to [out=90, in=180] (-3,0.6)
 to [out=0, in=180] (0, 1.5)
to [out=0, in=180] (3,0.6)
to [out=0, in=90] (4,0.3)
 to [out=270, in=0] (3,-0.4)
to   [out=180, in=0] (-3,-0.4)
to   [out=180, in=270] (-4,0.3);

\filldraw[fill=lightgray] 
%\filldraw[fill=gray, opacity=.1] 
(7,0.5) to [out=0, in=180] (10, 1.3)
to [out=0, in=180] (13,0.5)
to [out=180, in=0] (10, -0.3)
to [out=180, in=0] (7,0.5);

%\draw [dotted] (-3,0.6) -- (3,0.6);
%\draw [dotted] (-3,0) -- (3,0);
\draw [-, thick] (-3,0) to [out=180, in=270] (-4,0.3);
\draw [-, thick] (3,0) to [out=0, in=270] (4,0.3);
\draw [-, thick] (-3,0.6) to [out=180, in=90] (-4,0.3);
\draw [-, thick] (3,-0.4) to [out=0, in=270] (4,0.3);
\draw [-, thick] (-3,-0.4) to [out=180, in=270] (-4,0.3);
\draw [-, thick] (-3,-0.4) -- (3,-0.4);
\draw [-, thick] (3,0.6) to [out=0, in=90] (4,0.3);
\draw [-, thick] (-3,0) to [out=0, in=180] (-0, 1);
\draw [-, thick] (3,0) to [out=180, in=0] (0, 1);

\draw [-, thick] (-3,0.6) to [out=0, in=180] (0, 1.5);
\draw [-, thick] (3,0.6) to [out=180, in=0] (0, 1.5);
\fill (-4.5, 0.3) circle (0pt) node [left] {$M_\e$};

\draw [very thick] (7,0.5) -- (13,0.5);
\draw [-, thick] (7,0.5) to [out=0, in=180] (10, 1.3);
\draw [-, thick] (10,1.3) to [out=0, in=180] (13,0.5);
\draw [-, thick] (13,0.5) to [out=180, in=0] (10, -0.3);
\draw [-, thick] (10,-0.3) to [out=180, in=0] (7,0.5);

\fill (13.5, 0.5) circle (0pt) node [right] {$N$};
\end{tikzpicture}
\end{center}
\vspace{-0.5cm}  
\begin{figure}[htbp]
  \centering
  \caption{}
  \label{fig:conv-boat}  
\end{figure}  

\psmall

From Lemmas \ref{lem:not-perp}, \ref{lem:double-sus} and 
\ref{lem:single-interior}, we obtain the following.

\begin{prop}    \label{prop:perp+horiz}
Let $x\in X_0$.
\begin{enumerate}
 \item For every  $\xi\in\Sigma_x(Y)\setminus \Sigma_x(X)$ which is not a perpendicular direction, there is a unique perpendicular direction  $\xi_x^{+}\in \Sigma_x(Y)$ at $x$
and a unique  $v\in\Sigma_x(X_0)$ such that
      \begin{equation*}
         \angle(\xi_x^{+},\xi)+\angle(\xi, v)=\angle(\xi_x^{+}, v)=\pi/2;
                      % \label{eq:sus3}
     \end{equation*}
 \item For any perpendicular direction $\xi_x^+\in\Sigma_x(Y)$, we have
     \[
              \Sigma_x(X_0) = \{ v\in\Sigma_x(Y)\,|\,\angle(\xi_x^{+}, v)=\pi/2\}.
    \]
 \item If $x\in X_0^1$ and $\xi_x^+$ is the perpendicular direction at $x$, then  
    \[
              \Sigma_x(X) = \{ v\in\Sigma_x(Y)\,|\,\angle(\xi_x^{+}, v)\ge \pi/2\}.
    \]
\end{enumerate}
\end{prop}

Now we roughly describe the spaces of directions at the points of $\ca S^1$.
In Section \ref{sec:Isometric}, we give more details on it.

\pmed
\begin{lem}  \label{lem:S1-basic1} 
For $x\in \ca S^1$, let $\xi_x^+$ be the perpendicular direction at $x$.
Then we have the following:
 \begin{enumerate}
  \item $x\in Y^{\rm sing}$. More precisely,
$\Sigma_x(Y)=\{ \xi_x^+\} \star \Sigma$ holds, where 
$\Sigma=\Sigma_x(X_0)=\Sigma_x(X)\,;$
  \item $\Sigma$ is an Alexandrov space with curvature $\ge 1$,
 \end{enumerate}
where $\{ \xi_x^+\} \star \Sigma$ denotes the
union of all geodesics
joining $\xi_x^+$ and the points of $\Sigma$.
\end{lem}

It should be noted that $\{ \xi_x^+\} \star \Sigma$ is not the half suspension.

\begin{proof}[Proof of Lemma \ref{lem:S1-basic1}]
By Lemma \ref{lem:single-interior},
we have ${\rm rad}(\xi_x^+)=\pi/2$, 
and hence $x\in Y^{\rm sing}$.
It follows that $\Sigma_x(X_0)$ is the farthest set from $\xi_x^+$ with 
distance $\pi/2$  and therefore it is convex in 
$\Sigma_x(Y)$. Now (2) is obvious.
Proposition \ref{prop:perp+horiz} then implies (1).
\end{proof}

\psmall

From now on, we denote by $B^{X_0}(x,r)$ the $r$-ball around $x$ with respect to the 
extrinsic metric induced from $X$
(see also Lemma \ref{lem:same-topology}).

\pmed
\begin{lem} \label{lem:S2character}
For any $x\in\ca S^2$, there exists $\e>0$ such 
that for all
$y\in B^{X_0}(x, \e)$, we have   ${\rm rad}(\xi_y^+)>\pi-\tau_x(\e)>\pi/2$.

In particular,  
$B^{X_0}(x, \e)\cap X_0^1\subset {\rm int }\, X_0^1$.
\end{lem}

\begin{proof}
Take the perpendiculars $\gamma_x^{\pm}$ at $x$, and fix $t>0$
such that $\tilde\angle \gamma_x^+(t) x\gamma_x^{-}(t) =\pi$.
Choose a sequence $x_i\in X_0^1$ with $x_i\to x$.
Then we have $\tilde\angle \gamma_x^+(t) x_i\gamma_x^{-}(t) >\pi-o_i$.
%for all large enough $i$.
We may assume that the perpendicular $\gamma_{x_i}^+$ converges to 
$\gamma_x^+$ as $i\to\infty$.
It follows that $\angle(\dot\gamma_{x_i}^+(0), \uparrow_{x_i}^{\gamma_x^{-}(t)})>\pi-o_i$.
%If $x_n\in X_0^1$,
Lemma \ref{lem:single-interior} shows that 
$x_i\in {\rm int}\, X_0^1$.
This completes the proof by contradiction.
\end{proof}
\psmall
The following lemma is an immediate consequence
of Lemma \ref{lem:S2character}.

\begin{lem} \label{lem:closed-S1}
$\ca S^1$ is closed in $X_0$. 
\end{lem}

Lemma \ref{lem:S2character} also  
implies the following immediately.

\begin{cor} \label{cor:nowhere}
$\ca S^2$ is a nowhere dense subset of $N_0$.
\end{cor}
\pmed
\n 
{\bf Tangent cones $T_x(X_0)$ and $T_x(X)$.}\,

\begin{defn} For  $x\in X_0$, 
let $\sigma:[0, \e]\to X_0$ be a Lipschitz curve starting from $x$ in $X_0$
and having a unique direction $[\sigma]\in \Sigma_x(X_0)$  in the sense
that for any sequence $t_i\to 0$, $\uparrow_x^{\sigma(t_i)}$ converges to $[\sigma]$
in $\Sigma_x(Y)$.
We call such a direction $[\sigma]$ an {\it intrinsic direction} at $x$.  
Let $\Sigma_x^0(X_0)$ denote the set of all intrinsic directions $[\sigma]\in\Sigma_x(X_0)$.
The set $\Sigma_x^0(X)$ of all intrinsic directions $[\sigma]\in\Sigma_x(X)$
is defined similarly.
\end{defn}

The next result shows that every element in $\Sigma_x(X_0)$ or $\Sigma_x(X)$ can be approximated by
intrinsic directions.

\begin{prop} \label{prop:closure} 
For every $x\in X_0$, 
 $\Sigma_x(X_0)$ and $\Sigma_x(X)$ coincide with the closure of $\Sigma_x^0(X_0)$ and 
$\Sigma_x^0(X)$ in $\Sigma_x(Y)$ respectively.
\end{prop}
\begin{proof}
By Lemma \ref{lem:not-perp}, the set $\Sigma_0$ of all intrinsic directions $[\sigma]$ at $x$ such that $\sigma$ is written as $\sigma=\pi(\gamma)$ for $Y$-geodesics $\gamma$ in  $Y\setminus {\rm int}\, X$
is dense in $\Sigma_x(X_0)$. 
Moreover, if $\Sigma$ denotes the set of all directions of $Y$-geodesics
$\gamma:[0,\e)\to X$ such that $\gamma((0,\e))\subset {\rm int}\, X$
for some $\e>0$, 
then the union of $\Sigma$ and $\Sigma_0$ is dense in $\Sigma_x(X)$.
\end{proof}
\pmed

\begin{lem}$($\cite{YZ:inrdius}$)$ \label{lem:deviation}
For arbitrary $x, y\in X$ and any  minimal geodesic
$\mu:[0,\ell]\to Y$ joining them, let $\sigma=\pi\circ\mu$
and set $\rho(t)=|\mu(t), X|$. Then we have
\begin{enumerate}
 \item  $\rho(t)\le Ct|x,y|_Y$, where $C=C(\lambda)$. In particular,
\[ \max \rho\le O(|x,y|_Y^2);\]
 \item $\angle(\mu'(0), [\sigma])\le O(|x,y|_Y);$
 \item$\displaystyle {\left| \frac{L(\sigma)}{L(\mu)} - 1\right|< O(|x,y|_Y^2)}$.
\end{enumerate}
\end{lem}

\begin{proof}
(1)\, We may assume $\mathring{\mu}$
does not meet $X$. Let $\tilde\mu:=\eta^{-1}\circ\mu$.
% joining $p, q\in C_0$.
We apply \cite[Lemmas 4.1, 4.2]{YZ:inrdius}
to $\tilde\rho(s):=|\tilde\mu(s), C_0|=\rho(s)$ to 
obtain $\rho(t)\le t\rho'(0)\le Ct|x,y|_Y$.

(2) and (3) are  identical with
\cite[Lemma 4.14]{YZ:inrdius}.
\end{proof}

\psmall
\begin{prop} \label{prop:tang-cone}
For every $x\in X_0$, under the convergence  
\[
    \lim_{\e\to 0}\biggl(\frac{1}{\e} Y, x\biggr) =(T_x(Y),o_x),
\]
$(X_0, x)$ and $(X, x)$ converge to the Euclidean cones $K(\Sigma_x(X_0)),o_x)$ and $K(\Sigma_x(X)),o_x)$ respectively as $\e\to 0$.
\end{prop}
\begin{proof} 
We abbreviate $|\,\,,\,\,|_{Y/\e}$ as 
$|\,\,,\,\,|_\e$.
%Note that 
%$\Sigma_x(X)=\Sigma_x(Y)\setminus\mathring{B}(\xi_x^+,\pi/2)$ is convex, and hence 
%is an Alexandrov space with curvature $\ge 1$.
For $\delta>0$, choose a $\delta$-dense set $\{ w_j\}_{j=1}^J$ of $\Sigma_x(X)$ in such a way that:
\begin{itemize}
\item $\{ w_j\}_{j=1}^K$ \,$(K\le J)$ is a  
 $\delta$-dense subset of $\Sigma_x(X_0)$
for some $K\,;$
\item $\{ w_j\}_{j=K+1}^J\subset \Sigma_x(X)\setminus \Sigma_x(X_0)$. 
\end{itemize}
Let $\zeta>0$ be small 
enough, which will be determined later.
For each $1\le j\le J$, choose $\xi_j\in\Sigma_x(Y)$ with $|\xi_j,w_j|<\zeta$  satisfying
\begin{itemize}
\item the geodesic $\gamma_j$ in the direction $\xi_j$ is defined on an interval $[0,\e_j]\,;$
\item $\gamma_j((0,\e_j])\subset Y\setminus X$
for all $1\le j\le K\,;$
\item $\gamma_j((0,\e_j])\subset {\rm int}X$
for all $K+1\le j\le J$.
\end{itemize}
%%%%%%%%%%%%%%%%%%%
\pmed
\n
For any fixed $R>0$, choose small enough 
$\e >0$ such that $R\e<\min \{ \e_j\}$.
We consider the curve 
$\sigma_j:=\pi\circ\gamma_j$ in 
$X$\,$(1\le j\le J)$.  Note that $\sigma_j\subset X_0$ for all $1\le j\le K$ and 
$\dot\sigma_j(0)$
is $2\zeta$-close to $w_j$ for all $1\le j\le J$.
By Lemma \ref{lem:not-perp}, 
we have 
\beq\label{eq:diff-sigma-gamma}
   |\sigma_j(t),\gamma_j(t)|_\e\le \tau_{x}(\e,\zeta)t
\eeq
for all $t\in [0,R]$ and $1\le j\le J$.
Note that 
\begin{align*}
&\sup \{ \angle(\dot\gamma_{x,y}^Y(0),\Sigma_x(X))\,|\,y\in B^{X/\e}(x, R)\} <\tau_{x,R}(\e), \\
&\sup \{ \angle(\dot\gamma_{x,y}^Y(0),\Sigma_x(X_0))\,|\,y\in B^{X_0/\e}(x, R)\} <\tau_{x,R}(\e).
\end{align*}

We define a map  $\varphi_\e:B^{X/\e}(x,R)\to B^{K(\Sigma_x(X))}(o_x,R)$ as follows.
For each $y\in B^{X/\e}(x, R)$,
%choose a $Y$-minimal geodesic $\gamma_{x,y}^Y$
%joining $x$ to $y$.
choose $1\le j\le J$ such 
\[
      \angle(\dot\gamma^Y_{x,y}(0),w_j)<\tau_{x,R}(\e)+\delta,
\]
and define 
% and set
%$\sigma_{x,y}:=\pi\circ\gamma_{x,y}^Y$. Then we define
\[
        \varphi_\e(y):=|x,y|_{\e} w_j.              
\]
Here we may assume $w_j\in\Sigma_x(X_0)$
for all $y\in B^{X_0/\e}(x, R)$.
Taking small enough $\e\ll\delta$, we may assume that 
\[
 \varphi_\e(y)=|x,y|_{\e} w_j \quad \text{for all 
$y\in\sigma_j$ and $1\le j\le J$.}
\]
%%%%%%

We show that $\varphi_{\e}$ provides a  $\tau_{x,R}(\e,\delta,\zeta)$-approximation.
Now we use the symbol $a\sim b$ if $|a-b|<
\tau_{x,R}(\e,\zeta)$. 
%%%%%
For arbitrary $y_i\in\sigma_i$ and $y_j\in\sigma_j$ with $1\le i, j\le J$, let $u_i\in \gamma_i$ and $u_j\in \gamma_j$ be
such that $|x,u_i|=|x,y_i|$ and $|x,u_j|=|x,y_j|$.
Using \eqref{eq:diff-sigma-gamma}, we have 
\begin{align*} \label{eq:|yiyj|}
   |y_i,y_j|_{\e} &\sim |u_i,u_j|_{\e} \\
 &\sim ||x,y_i|_\e w_i, 
|x,y_j|_\e w_j|_{K(\Sigma_x(X))}\\ 
&= |\varphi_{\e}(y_i),\varphi_{\e}(y_j)|_{K(\Sigma_x(X))}.
\end{align*}

For any $z\in B^{X/\e}(x,R)$,
take $j$ such that $\angle(\dot\gamma^Y_{x,z}(0),
w_j)<\tau_{x,R}(\e)+\delta$.
Set 
$z_j:=\sigma_j(|x,z|_\e)$. 
Let $\tilde z_j\in\gamma_j$ be the point with
$z_j=\pi(\tilde z_j)$.
%%%%%
Since $\angle^Yzx\tilde z_j<\tau_{x,R}(\e)+\delta+2\zeta$,
we have $|z,\tilde z_j|_\e<\tau_{x,R}(\e,\delta,\zeta)$.
It follows from \eqref{eq:diff-sigma-gamma} that
$$
|z,z_j|_\e\le |z,\tilde z_j|_\e + |\tilde z_j, z_j|_\e<
\tau_{x,R}(\e,\delta,\zeta).
$$
From construction, we also have
$\varphi_{\e}(z)=\varphi_{\e}(z_j)$.
For any other point  $z'\in B^{X/\e}(x,R)$,
we take $j'$ and $z'_{j'}$ in a way similar to the above
$j$ and $z_j$ for $z$.
Then we have 
 \begin{align*}
  &| |z,z'|_{\e} - |z_j, z'_{j'}|_{\e} | < \tau_{x,R}(\e,\delta,\zeta),\\
  &  |\varphi_{\e}(z), \varphi_{\e}(z')|=|\varphi_{\e}(z_j), \varphi_{\e}(z'_{j'})|.
\end{align*}
It follows that 
\[
    | |\varphi_{\e}(z), \varphi_{\e}(z')|-|z,z'|_{\e}|< 
\tau_{x,R}(\e,\delta,\zeta).
\]

Finally
% for any $u\in B^{K(\Sigma_x(X))}(o_x,R)$, take 
%$1\le j\le J$ with 
%$\angle(u, w_j)<\delta$.
%Then obviously we have $|u,\varphi_{\e}(\sigma_j(|u.|)|<\tau_R(\e,\delta)$.
note that the image of $\varphi_\e$ is the union of the segments joining $o_x$ to $Rw_j$, which is 
$R\delta$-dense in $B(o_x,R)$
Thus $\varphi_\e$ is a
$\tau_{x,R}(\e,\delta,\zeta)$-approximation. This shows the 
convergence $(X,x)\to (T_x(X),o_x)$.

The above argument also shows the 
convergence $(X_0,x)\to (T_x(X_0),o_x)$.
This completes the proof.
\end{proof}

We set
\[
       T_x(X_0) := K(\Sigma_x(X_0)), \quad T_x(X) := K(\Sigma_x(X))
\]
and call them the {\it tangent cones}  of $X_0$ and $X$ at $x$ respectively.

%%%%%%%%%%%%%%%%%%%%%%
\pmed
Now we consider the {\it differential} of the map $\eta:C\to Y$, which is defined as a rescaling limit. 
Fix $p\in C_0$ and $x=\eta_0(p)\in X_0$, and
let $t_i$ be an arbitrary sequence of positive numbers with $\lim_{i\to\infty} t_i=0$.
Passing to a subsequence, we may assume that
\beq \label{eq:deta}
     \eta_i:=\eta:  \biggl(\frac{1}{t_i}C, p\biggr) \to \left(\frac{1}{t_i}Y, x\right)
\eeq
converges to a $1$-Lipschitz map
\beq\label{eq:eta-infty}
   \eta_{\infty}:  (T_p(C),o_ p) \to (T_x(Y), o_x)
\eeq
between  the tangent cones of the Alexandrov spaces.

By \cite[Lemma 4.6(2)]{YZ:inrdius},
the limit 
$\eta_{\infty}:  (T_p(C),o_ p) \to (T_x(Y), o_x)$ is uniquely determined.
We denote it by $d\eta_p$, and call it the {\it differential} of $\eta$ at $p$.
Note that $d\eta_p$ induces a $1$-Lipschitz map
\[
            (d\eta_0)_p:T_p(C_0) \to T_x(X_0),
\]
and an injective local isometry
\[
  (d\eta)_p:T_p(C)\setminus T_p(C_0) \to T_p(Y)\setminus T_x(X).
\]
In what follows, we we simply write as $d\eta$ and $d\eta_0$ for $d\eta_p$ and $(d\eta_0)_p$ respectively, when 
$p$ is fixed.

For the perpendicular $\gamma_x^+(t):=\eta(p,t)$,
let $\xi_x^+:=   \dot\gamma_x^+(0)$, and set
\[
    \Sigma_x(Y)^+:=\{ \xi\in\Sigma_x(Y)\,|\,\angle(\xi,\xi_x^+)\le\pi/2\}.
\]

\begin{prop} \label{prop:length'}
For every $\tilde v\in T_p(C)$, we have  
\beq\label{eq:deta=}
                |d\eta(\tilde v)| = |\tilde v|.
\eeq
In other words, if $p_n\in C_0$ converges to $p$, then we have 
\[
      \lim_{n\to\infty} \frac{|\eta_0(p_n),\eta_0(p)|_Y}{|p_n,p|_{C_0}}=1.
\]
In particular,  $\eta:C\to Y$  and $\eta_0:C_0\to X_0$ preserve the length of Lipschitz curves, and they induce  $1$-Lipschitz maps
$$
d\eta:\Sigma_p(C) \to \Sigma_x(Y),\quad
d\eta_0:\Sigma_p(C_0) \to \Sigma_x(X_0)
$$
such that 
\begin{itemize}
\item both $d\eta:\Sigma_p(C) \to \Sigma_x(Y)^+$ and
$d\eta_0:\Sigma_p(C_0) \to \Sigma_x(X_0)$
are surjective$\,;$
\item $d\eta:\Sigma_p(C)\setminus\Sigma_p(C_0) \to \mathring{\Sigma}_x(Y)^+$
is a bijective local isometry.
\end{itemize}
\end{prop}

\begin{proof}
In view of Proposition \ref{prop:tang-cone},
\cite[Sublemma 4.4]{YZ:inrdius} shows 
that
$d\eta:T_p(C)\setminus T_p(C_0) \to T_x(Y)\setminus T_x(X)$
is an injective local isometry. 
By \cite[Lemmas 4.6, 4.16]{YZ:inrdius}, 
this map satisfies $|d\eta(\tilde v)|=|\tilde v|$
for all $\tilde v\in T_p(C)\setminus T_p(C_0)$, 
and therefore induces the
bijective and locally isometric map  
$d\eta:\Sigma_p(C)\setminus \Sigma_p(C_0) \to \mathring{\Sigma}_x(Y)^+$.
From the continuity,  this extends to the
surjective  map 
sending $\Sigma_p(C_0)$ to $\Sigma_x(X_0)$
and satisfying \eqref{eq:deta=}.  
\end{proof}

The proof of the following proposition is identical with \cite[Propostion 4.19]{YZ:inrdius}, and hence omitted.

\begin{prop}$($\cite[Proposition 4.19]{YZ:inrdius}$)$ \label{prop:isometry'}
For every $p\in C_0^2$,  $d\eta_p$ provides an isometry $d\eta_p : T_p(C)\to T_x^{+}(Y)$ which preserves
    the half suspension structures of both $\Sigma_p(C)=\{\xi_{+}\}*\Sigma_p(C_0)$ and
   $\Sigma_x^{+}(Y)=\{ \xi_{+}\}*\Sigma_x(X_0)$, where $T_x^{+}(Y):=T_x(X_0)\times\mathbb R_{+}$.
 \end{prop}

Proposition \ref{prop:isometry'} does not hold
for $p\in \tilde{\ca S^1}$.
For $p\in C_0^1$, the relation between $\Sigma_p(C)$ and 
$\Sigma_x(Y)$ will be made clear in 
Theorem \ref{thm:X1-f}.
\pmed
%%%%%% summary of notations %%%%
Here we present a summary of basic notations
and convergences
used so far for readers' convenience.

\begin{table}[h]
\begin{center}
\begin{tikzpicture}[auto]
\node (a) at (0, 0) {$M_i$}; 
\node (x) at (2, 0) {$N$};
\node (b) at (0, -1) {$\pa M_i$};   
\node (y) at (2, -1) {$N_0$};
\draw[->] (a) to node {$\scriptstyle GH$} (x);
\draw[->] (b) to node {$\scriptstyle GH$} (y);
\fill (0, -0.3) circle (0pt) node [below] {$\cup$};
\fill (2, -0.3) circle (0pt) node [below] {$\cup$};
\draw[thick] (-1.6,0.5) rectangle (3.5,-1.4);
\end{tikzpicture}
\caption{\small{Original convergence}}
\end{center}
%\end{subtable}
\end{table}

\begin{table}[h]
%\begin{subtable}{.5\linewidth}
    \centering

\begin{center} 
\begin{tikzpicture}[auto]

\fill (-1, 1) circle (0pt) node [right] {$
\tilde M_i=M_i\cup_{\pa M_i} C_{M_i}, \quad
C_{M_i}=[0,t_0]\times_{\phi}(\pa M_i)^{\rm int}$};
\node (a) at (0, 0) {$\tilde M_i$}; 
\node (x) at (2, 0) {$Y$};
\node (x2) at (4, 0) {$C$};
\node (x3) at (6.2, 0) {$C_{M_i}$};
\node (b) at (0, -1) {$M_i$};   
\node (y) at (2, -1) {$X$};
\node (c) at (0, -2) {$\pa M_i$};   
\node (z) at (2, -2) {$X_0$};
\node (z2) at (4, -2) {$C_0$};
\node (z3) at (6.2, -2) {$(\pa M_i)^{\rm int}$};
\fill (-0.4, -3) circle (0pt) node [right] {$N=X^{\rm int}, \quad N_0=X_0,\quad  {\rm int} N=N\setminus N_0$};

\fill (0.3, -3.8) circle (0pt) node [right] 
  {$Y=X\cup_{\eta_0} C,\quad 
C= [0,t_0]\times_{\phi}C_0$};
\draw[thick] (-1.1,1.5) rectangle (7.6,-4.2);
\draw[->] (a) to node {$\scriptstyle GH$} (x);
\fill (0, -0.3) circle (0pt) node [below] {$\cup$};
\draw[->] (x) to node {$\scriptstyle \pi$} (y);
\draw[->] (b) to node {$\scriptstyle GH$} (y);
\fill (0, -1.3) circle (0pt) node [below] {$\cup$};
\fill (2, -1.3) circle (0pt) node [below] {$\cup$};
\draw[->] (c) to node {$\scriptstyle GH$} (z);
\draw[->] (x2) to node[swap] {$\scriptstyle \eta$} (x);
\draw[->] (z2) to node[swap] {$\scriptstyle \eta_0$} (z);
\draw[->] (x2) to node {$\scriptstyle \tilde\pi$} (z2);
\draw[->] (x3) to node[swap] {$\scriptstyle GH$} (x2);
\draw[->] (z3) to node[swap] {$\scriptstyle GH$} (z2);
\draw[->] (x3) to  (z3);
\end{tikzpicture}  
\end{center}
  \caption{\small{Convergence of extensions}}
\end{table}

\begin{table}[h]
\begin{center}
\begin{tikzpicture}[auto]
\fill (-1, 0) circle (0pt) node [right] {$ X_0^k=\{ x\in X_0\,|\, \# \eta_0^{-1}(x)=k\}, \quad
C_0^k=\eta_0^{-1}(X_0^k)$};
\fill (-1, -0.8) circle (0pt) node [right]{$\ca S^k=\pa X_0^k\cap X_0^k$,  \quad $\tilde{\ca S}^k=\pa C_0^k\cap C_0^k$ \,\, $(k=1,2)$};

\fill (0.5, -1.6) circle (0pt) node [right]{
$\ca S=\ca S^1 \cup \ca S^2$,\quad $\tilde{\ca S}=
\tilde{\ca S}^1\cup\tilde{\ca S}^2$};

\draw[thick] (-1.2,0.5) rectangle (7.8,-2);

\end{tikzpicture}
\caption{\small{Boundary singular sets}}
\end{center}
\end{table}

%%%%%%%%%%%%%%%%%%%%%%%%%%%%%%
\pmed
\setcounter{equation}{0}

\section{Gluing maps} \label{sec:gluing}
\psmall
In this section, we still assume that $M_i\in \ca M(n,\kappa,\nu,\lambda,d)$ converges to a compact geodesic  space 
$N$ while ${\rm inrad}(M_i)$ has a positive lower bound independent of $i$.
In what follows, we describe the properties of 
of the gluing map $\eta_0:C_0\to X_0$.

From Lemma \ref{lem:preimage}, we can define a map $f:C_0\to C_0$ as  follows: For an arbitrary point $p\in C_0$,
let 
\begin{align} \label{eq:defn-f}
    \text{$f(p):=q$ \, if\, $\{p,q\}=\eta_0^{-1}(\eta_0(p))$}, 
\end{align}
where  $q$ is equal to $p$ if $\eta_0(p)\in X_0^1$.
Note that $f$ is an involution,  i.e.,  $f^2={\rm id}$.

First we investigate the continuity of $f$.
By definition,   $f$ is the identity on $C_0^1$.
Note that $f:C_0\to C_0$ is not continuous at every point of $\tilde{\ca S}^2$.
\pmed
\begin{rem}
In the case of inradius collapse, $f:C_0\to C_0$ is an isometry
(\cite[Proposition 4.27]{YZ:inrdius})
\end{rem}

\begin{lem} \label{lem:contin} We have the following.
 \begin{enumerate}
  \item  $f:C_0\to C_0$ is continuous on $C_0 \setminus \tilde{\mathcal S}^2\,;$
  \item $f:C_0^2\to C_0^2$  is continuous.
 \end{enumerate}
\end{lem}

\begin{proof} 
(1)\, Let $p_i\in C_0 \setminus \tilde{\mathcal S}^2$
converge to $p\in C_0 \setminus \tilde{\mathcal S}^2$.
Note that 
\[
 C_0 \setminus \tilde{\mathcal S}^2=
 C_0^1\cup{\rm int}C_0^2.
\]
a)\,\, Suppose $p\in C_0^1$.
If $f(p_i)$ does not converge to $p$, we have a contradiction to $p\in C_0^1$.
\par\n
b)\,\,Suppose $p\in {\rm int}C_0^2$, and hence
$p_i\in{\rm int} C_0^2$ for large $i$. If $f(p_i)$ does not converge to $f(p)$, we may assume that it converges to $p$.
Set $x:=\eta_0(p)$, $x_i:=\eta_0(p_i)$.
Let $\gamma_x^{\pm}$ be the two perpendiculars 
at $x$.
Choose $s_0>0$ such that 
$\wangle \gamma_x^+(s_0) x \gamma_x^-(s_0)=\pi$.
By the assumption,  the two perpendiculars $\gamma_{x_i}^+$ and 
$\gamma_{x_i}^-$  at $x_i$ converge to a perpendicular,
say $\gamma_x^+$, at $x$. This implies
\[
     \angle(\uparrow_{x_i}^{\gamma_x^-(s_0)},
     \dot\gamma_{x_i}^{\pm}(0)) >\pi-o_i,
\]
and hence $\angle(\dot\gamma_{x_i}^{+}(0),
\dot\gamma_{x_i}^{-}(0)) <o_i$,
This is a contradiction.
\par\n
(2)\, This is similarly discussed as the above (1)-b).
\end{proof}

\begin{ex} \label{ex:1}
For $0<\epsilon\ll \delta<1$, let $g_{\epsilon}:(-3,3)\to (0, \infty)$ be a smooth function 
satisfying 
\begin{enumerate}
 \item $g_{\epsilon}(-x)=g_{\epsilon}(x)$, $\lim_{x\to \pm 3}g_{\epsilon}(x)=0$,
 \item $g_{\epsilon}(x)$ does not depend on $\epsilon$ on $|x|\ge 1$, 
 \item $g_{\epsilon}(x)=\epsilon$ on $[-\delta, \delta]$,
 \item $g_{\epsilon}' >0$ on $(\delta, 2)$, and $g_{\epsilon}' <0 $ on $(2,3)$
 \item $|g_{\epsilon}''|$ is uniformly bounded on $[\delta, 1]$,
 \item the domain $N_{\epsilon}$ on $xy$-plane bounded by the graphs $y=\pm g_{\epsilon}(x)$
         is smooth.
\end{enumerate}
Note that  $\partial N_{\epsilon}$ has a uniformly bounded absolute geodesic curvature.
therefore $M_{\epsilon}:=N_{\epsilon} \times S^1_{\epsilon}$ belongs to $\mathcal M_b(3, 0,\lambda,d)$
for some $\lambda$ and $d$, and it converges to the limit $N$ of $N_{\epsilon}$ with respect to the Hausdorff distance. 
In this case, $N_0$ coincides with the topological boundary of $N$ in $\mathbb R^2$, 
$N_0^2=[-\delta,\delta]\times \{ 0\}$ and 
$\mathcal S=\mathcal S^2$ consists of the two points $\{ (\pm\delta,0)\}$.
\end{ex}

\begin{lem} \label{lem:f'}
$f:C_0 \setminus {\tilde{\mathcal S}} \to C_0\setminus {\tilde{\mathcal S}} $ is a  local isometry.
\end{lem}
\begin{proof}
Since $C_0 \setminus \tilde{\mathcal S}={\rm int}\, C_0^1\cup {\rm int}\, C_0^2$, it suffice to show that $f$ is a local isometry on 
${\rm int}\, C_0^2$. This follows in a way similar to 
\cite[Proposition 4.27]{YZ:inrdius}.
\end{proof}
\pmed

From now on, we investigate the properties of 
$\eta_0$ in more detail.
In a way similar to Corollary 4.21 and Lemmas 4.25, 4.26 of \cite{YZ:inrdius},  
we have the following.

\begin{lem} \label{lem:eta'}
The restrictions of $\eta_0$  to 
${\rm int}\,C_0^k$\,$(k=1,2)$ 
have the following properties:
 \begin{enumerate}
 \item $\eta_0 :{\rm int}\,C_0^1\to {\rm int}\,X_0^1$ is a bijective local isometry  with 
respect to the intrinsic metrics$;$
 \item $\eta_0:{\rm int}\,C_0^2\to {\rm int}\,X_0^2$ is a locally isometric double covering with respect to the intrinsic metrics.
 \end{enumerate}
\end{lem} 

\pmed
 In what follows, we study the local 
property of $\eta_0$ at the general points 
of $C_0^k$\, $(k=1,2)$.
Recall that $\tau_p(r)$ is a positive function depending 
on $p\in C$ that  satisfies $\lim_{r\to 0}\tau_p(r)=0$.

\begin{lem} \label{lem:eta1}
For every $x\in X_0^1$ and $p\in C_0$ with 
$x=\eta_0(p)$, there is an $r_0>0$ such that the map
$\eta_0:B^{C_0}(p,r)\to B^{X_0}(x,r)$ is almost surjective for all $0<r\le r_0$
in the sense that 
\[  
     \eta_0(B^{C_0}(p,r))\supset B^{X_0}(x, (1-\tau_p(r))r).
\]
\end{lem}

\begin{proof}
Since $\eta_0:C_0\to X_0$ is length-preserving by Proposition \ref{prop:length'},
$\eta_0:B(p,r)\to B(x,r)$ is well-defined.
Suppose the conclusion does not hold. Then we have a sequence $r_i>0$ with
$\lim_{i\to\infty} r_i=0$ and a sequence $x_i\in B(x, (1-c)r_i)\setminus \eta_0(B(p,r_i))$ for a constant $c>0$.
Clearly
$p_i:=\lim_{t\to 0}\eta^{-1}(\gamma_{x_i}^+(t))$ is well-defined,
and satisfies $\eta_0(p_i)=x_i$.  From $x\in X_0^1$, we have 
$p_i\to p$. Proposition \ref{prop:length'} implies  
$\lim_{i\to\infty} |x,x_i|/|p,p_i|=1$. Thus $p_i\in B(p,r_i)$ for large $i$, which 
is a contradiction to the choice of $x_i$. 
\end{proof}
\begin{rem}  \label{rem:not-sur2}  \upshape
Lemma \ref{lem:eta1} does not hold for $x\in \ca S^2$.
See Figure \ref{fig:conv-surface}  in Section \ref{sec:intro}.

\end{rem}

\begin{lem} \label{lem:eta-2}
For every $x\in X_0^2$ with 
$\{ p_1, p_2\}=\eta_0^{-1}(x)$, there are neighborhoods 
$U_k$ of $p_k$\, $(k=1,2)$\, in $C_0$ such that 
each restriction $\eta_0: U_k\to \eta_0(U_k)$ is an
isometry with respect to their intrinsic metrics.
\end{lem}
\begin{proof}
First we show that  $\eta_0: U_k\to \eta_0(U_k)$ is injective
for a neighborhood $U_k$ of $p_k$.
Suppose the contrary for $k=1$. Then we have sequences $q_i^+\neq q_i^-\in C_0$ 
converging $p_1$ such that $\eta_0(q_i^+)=\eta_0(q_i^-)$. 
Let $\tilde\gamma_i^+$ and $\tilde\gamma_i^-$ be 
perpendiculars at $q_i^+$ and $q_i^-$ respectively.
Then $\gamma_i^{\pm}:=\eta(\tilde\gamma_i^{\pm})$ are the 
perpendiculars at $x_i:=\eta_0(q_i)$.
Let $\gamma_x^{\pm}$ be the perpendiculars at $x$. We may assume that
$\gamma_i^+$ converges to $\gamma_x^+$.
Since $\tilde \angle \gamma_x^+(\delta) x\gamma_x^-(\delta)=\pi$ for small  $\delta>0$, it follows that  
$\lim_{i\to\infty}\tilde\angle \gamma_i^+(\delta) x_i\gamma_x^-(\delta)
=\pi$, and therefore
 $\lim_{i\to\infty} \angle \gamma_x^-(\delta) x_i\gamma_i^-(\delta)=0$.
This implies that $\gamma_i^-\to \gamma_x^-$.
However since $\tilde\gamma_i^{\pm}$ converge to the perpendicular 
$\tilde\gamma$ at $p_1$, $\gamma_i^{\pm}$ must converge to $\eta(\tilde\gamma)=\gamma_x^+$.
This is a contradiction.

Now it follows from Proposition \ref{prop:length'}  that each 
$\eta_0: U_k\to \eta_0(U_k)$ is an isometry 
with respect to the intrinsic metrics.
\end{proof}

\begin{prop} \label{prop:eta-2cover}
The restriction $\eta_0:C_0^2\to X_0^2$ is a locally 
almost isometric double covering. Namely, for every $x\in X_0^2$, there
is a neighborhood $V$ of $x$ in $X_0^2$ such that 
$\eta_0^{-1}(V)$ consists of disjoint open subsets
$W_1$ and $W_2$ of $C_0^2$ and each restriction $\eta_0:W_k\to V$\,\,$(k=1,2)$\,
 is almost isometric in the sense that for all $\tilde y, \tilde z\in W_k$ we have  
\[
    1-\tau_x(r)< \frac{|\eta_0(\tilde y), \eta_0(\tilde z)|_Y}{|\tilde y,\tilde z|_{C_0}} \le 1,
\]
where $r=\diam(V)$.
\end{prop}
\begin{proof}
For every $x\in X_0^2$, let $\{ p_1, p_2 \}:=\eta_0^{-1}(x)$. 
For $0<r< |p_1,p_2|_{C_0}/2$, we set
\[
         U_k(r):= \mathring{B}^{C_0}(p_k, r), \quad (k=1,2).
\]

\begin{slem} \label{slem:eta-ball}
For a small enough $r$, we have 
\begin{align*}
       \eta_0(U_1(r))\cup &\eta_0(U_2(r)) \supset B^{X_0}(x,(1-\tau_x(r))r) \\
        &      \supset \eta_0(U_1((1-\tau_x(r))r))\cup  \eta_0(U_2((1-\tau_x(r))r)) 
\end{align*}
In particular, $\eta_0(U_1(r))\cup \eta_0(U_2(r))$ contains an open neighborhood of $x$ in $X_0$.
\end{slem}
\begin{proof} 
The second inclusion is obvious from the $1$-Lipschitz property of $\eta_0$.
Suppose that the first inclusion does not hold.
Then we have $r_i\to 0$,  $c>0$ and 
$x_i\in  B^{X_0}(x,(1-c)r_i) \setminus  \eta_0(U_1(r_i))\cup \eta_0(U_2(r_i)) $.
Take $p_i\in C_0$ with $\eta_0(p_)=x_i$.
We may assume that 
$p_i\to p_1$.
Since Proposition \ref{prop:length'} implies that 
$|x,x_i|/|p_1,p_i|\to 1$, we have 
$|p_1,p_i|<r_i$ for large $i$, and hence $x_i\in \eta_0(U_1(r_i))$.
This is a contradiction.
\end{proof}

\begin{slem} \label{slem:eta-inter}
For a small enough $r$, we have  
\begin{align*}
      X_0^2&\supset \eta_0(U_1(r))\cap \eta_0(U_2(r))\\
   &\supset X_0^2\cap (\eta_0(U_1((1-\tau_x(r))r))\cup \eta_0(U_2((1-\tau_x(r))r)))\\
   &\supset  \eta_0(U_1((1-\tau_x(r))r))\cap \eta_0(U_2((1-\tau_x(r))r)).
\end{align*} 
In particular, $ X_0^2\cap (\eta_0(U_1(r/2))\cup\eta_0(U_2(r/2)))$ is a closed subset of $X_0$.
\end{slem}
\begin{proof} 
The first and the third inclusions are obvious.
Suppose that the second inclusion does not hold.
Then we have a sequence $r_i\to 0$,  $c>0$ and 
$x_i\in X_0^2\cap (\eta_0(U_1((1-c)r_i))\cup \eta_0(U_2((1-c)r_i))\setminus \eta_0(U_1(r_i))\cap \eta_0(U_2(r_i))$.
We may assume $x_i\in X_0^2\cap \eta_0(U_1((1-c)r_i))$
for all $i$.
Take $p_i^1\in U_1((1-c)r_i)$ with $\eta_0(p_i^1)=x_i$.
Note $|x,x_i|_{X_0}<(1-c)r_i$.
From $x_i\in X_0^2$
we  have a point $p_i^2\in C_0\setminus \{ p_i^1\}$  with $\eta_0(p_i^2)=x_i$.
We may assume that  $p_i^2$ converges to a point $q\in \eta_0^{-1}(x)$.
Since $p_i^1\to p^1$, by Lemma \ref{lem:contin}  (2), we have $q=p_2$.
 Proposition \ref{prop:length'} then implies that 
$|x,x_i|/|p_2,p_i^2|\to 1$.
This implies $p_i^2 \in U_2(r_i)$, and hence
$x_i\in \eta_0(U_1(r_i))\cap \eta_0(U_2(r_i))$,
a contradiction. 
\end{proof}

\begin{rem} \label{rem:S2=notclosed}\upshape
Concerning Sublemma \ref{slem:eta-inter},
 $\ca S^2$ is not necessarily closed in $X_0$.
 See Figure 1 in Section \ref{sec:intro}.
See also  Lemma \ref{lem:closed-S1}.
\end{rem}

\begin{slem} \label{slem:almost-2cover}
%%%%%%
For any $p\in C_0^2$, if $r>0$ is small enough, then $\eta_0:B^{C_0}(p,r)\to X_0$ is injective, and 
for all $\tilde y,\tilde z\in B^{C_0}(p,r)$ we have 
\[
    1-\tau_p(r)\le 
 \frac{|\eta_0(\tilde y), \eta(\tilde z)|_Y}{|\tilde y, \tilde z|_{C_0}}\le 1.
\]
\end{slem}
\begin{proof}
Suppose the sublemma does not hold. Then 
%for some $x\in X_0^2$
we have sequences $\tilde y_i$ and $\tilde z_i$ in $C_0$ converging to $p$ satisfying 
%there are  
%$\tilde x$ and $\tilde y_i, \tilde z_i\in B(\tilde x, 2/i)$ 
%satisfying $\eta(\tilde x)=x$, $\eta(\tilde y_i)=y_i$, $\eta(\tilde z_i)=z_i$ and 
\begin{align} \label{eq:d(yz)-quot}
  \frac{|\eta_0(\tilde y_i), \eta_0(\tilde z_i)|_Y}{|\tilde y_i, \tilde z_i|_{C_0}}   \to 0.
\end{align}
%for some uniform constant $c>0$.
Let $\gamma_{y_i}^+(t):=\eta(\tilde y_i, t)$,
$\gamma_{z_i}^+(t):=\eta(\tilde z_i, t)$ be the perpendiculars
at $y_i:=\eta_0(\tilde y_i)$ and $z_i:=\eta_0(\tilde z_i)$ respectively.
Set $s_i:=|y_i,z_i|_Y$, and let $\sigma_i:[0,1]\to C_{s_i}^Y$ be a minimal geodesic 
from $\gamma_{y_i}^+(s_i)$ to $\gamma_{z_i}^+(s_i)$
in $C_{s_i}^Y$.
Now we may assume that 
$(\frac{1}{s_i} Y, y_i)$ converges to an Alexandrov space $(Y_\infty,y_\infty)$ with nonnegative curvature.
%%%%%%%%%%%%%%%%%
Combining  Lemma \ref{lem:S2character}
and the splitting theorem implies that $Y_\infty$ is isometric to a 
product $\mathbb R\times X_\infty$.
Under this convergence, 
we may also assume that 
$\gamma_{y_i}^+$, $\gamma_{z_i}^+$ and 
$\sigma_i$ converge to geodesic rays $\gamma_{y_\infty}^+$, $\gamma_{z_\infty}^+$ and a minimal geodesic segment $\sigma_\infty$ in $Y_\infty$ joining  $\gamma_{y_\infty}^+(1)$
and $\gamma_{z_\infty}^+(1)$ respectively.
Now we conclude that 
\[
 \lim_{i\to\infty}\frac{\phi(s_i)}{s_i}|\tilde y_i,\tilde z_i|_{C_0}
    =|\sigma_\infty(0), \sigma_\infty(1)|=|y_\infty, z_\infty|=1.
\]
This is a contradiction to \eqref{eq:d(yz)-quot}.
\end{proof}

For small enough $r>0$, set $V_0:=\eta_0(U_1(r))\cap \eta_0(U_2(r))$.
By Sublemmas \ref{slem:eta-ball} and \ref{slem:eta-inter},
 $V_0$ contains an open neighborhood $V$ of $x$
in $X_0^2$.
Put $W_k:= (\eta_0|_{U_k(r)})^{-1}(V)$\, $(k=1,2)$.
Note that $\eta_0^{-1}(V)$ is the disjoint union of $W_1$ and $W_2$.   By Sublemma \ref{slem:almost-2cover}, $\eta_0:W_k\to V$ are almost isometric.
This completes the proof of Proposition \ref{prop:eta-2cover}.
\end{proof}

Now we discuss some topics using the results
proved so far.
 The first two are about $\ca S^2$. 

\begin{lem}\label{lem:lift-S2} It holds that 
 $\tilde{\ca S}^2=\eta_0^{-1}(\ca S^2)$.
\end{lem}
\begin{proof}
For any $\tilde x \in\tilde{\ca S}^2$, choose a sequence $\tilde x_i\in C_0^1$
converging to $\tilde x$. 
Then $x_i:=\eta_0(\tilde x_i)\in X_0^1$
converges to $x:=\eta_0(\tilde x)\in X_0^2$.
Thus we have $x\in \ca S^2$ and hence $\tilde x\in
\eta_0^{-1}(\ca S^2)$.

Conversely, for any $\tilde x\in\eta_0^{-1}(\ca S^2)$,
set $x:=\eta_0(\tilde x)\in\ca S^2$.
Take a sequence $x_i\in X_0^1$ converging to $x$.
If the lift $\tilde x_i\in C_0^1$ of
$x_i$ converges to $\tilde x$, then we certainly have 
$\tilde x\in \tilde{\ca S^2}$.
Otherwise, $\tilde x_i$ converges to the other lift $\tilde x'$
of $x$. Let $U_1(r)$ and $U_2(r)$ be as in Sublemma 
\ref{slem:eta-ball} with $\tilde x\in U_1(r)$ and 
$\tilde x'\in U_2(r)$.

%%%%%%%%%
\begin{slem} \label{slem:spherical-U}
Let $\Sigma$ be such that $\Sigma_x(Y)$ is the spherical suspension 
$\{\xi_x^{\pm}\}*\Sigma$, and let $\Sigma_x^k(r):=\{\uparrow_x^y\,|\,y\in  \eta_0(U_k(r/2))\}$ for each $k\in\{ 1,2\}$.
Then we have 
\[
d_H^{\Sigma_x(Y)} ( \Sigma_x^k(r),\Sigma)<\tau_x(r),
\]
where $d_H^{\Sigma_x(Y)}$ denotes the Hausdorff distance in $\Sigma_x(Y)$.
\end{slem}
\begin{proof} For any $\e>0$, we easily have
$\Sigma_x^k(r)\subset B(\Sigma, \e)$ for small enough $r$ from a limit argument.
Suppose that $\Sigma\subset B(\Sigma_x^k(r),\e)$ 
does not hold for any small enough $r$ and some $\e$. Then we have sequences $r_i\to 0$ 
and $v_i\in \Sigma\setminus B(\Sigma_x^k(r_i),\e)$.
We may assume $v_i\to v\in\Sigma$.
%For $v\in\Sigma\setminus\Sigma_x^k(r)$, 
Take $\xi^\pm\in\Sigma_x(Y)$ such that  
$\angle(\xi_x^+,\xi^\pm)+\angle(\xi^\pm,v)=\pi/2$ and $\angle(\xi^\pm,v)=\pi/4$.
In what follows,  we may assume that the geodesics $\gamma^\pm$ in the directions 
$\xi^\pm$ are defined on some interval
$[0, s]$.
By Lemma \ref{lem:not-perp}, for any small
enough  $0<t<s$,
we have 
$\angle(\uparrow_x^{\pi\circ\gamma^\pm(t)}, v)
<\e/2$ and hence 
$\angle(\uparrow_x^{\pi\circ\gamma^\pm(t)}, v_i)
<\e$ for large $i$.
This is a contradiction.
\end{proof}

%%%%%%%%%
Note that $x_i\in \eta_0(U_2(r))\setminus\eta_0(U_1(r))$ for large $i$.
Let $y_i$ be a nearest point of $\eta_0(U_1(r))$
from $x_i$. Let $r_i:=|x,x_i|$.
Sublemma \ref{slem:spherical-U} implies 
that $\lim_{i\to\infty}|x_i,y_i|/r_i=0$.
%
%Since $\Sigma_x(\eta_0(U_1(r)))=\Sigma_x(\eta_0(U_2(r)))$,
%$\eta_0(U_1(r))$ and $\eta_0(U_2(r))$  are tangent 
%at $x$.
Lemma \ref{lem:single-interior} implies
$y_i\in X_0^1$. From $y_i\to x$, we have
$\eta_0^{-1}(y_i)\to \tilde x$, and hence $\tilde x\in
\tilde{\ca S}^2$.
\end{proof}

%\begin{proof}
%Let a sequence $x_i$ in $\ca S^1$ converge to a point $x\in X_0$.
%Since $x\in\ca S$, we only have to show  $x\in X_0^1$.
%Since ${\rm rad}(\xi_{x_i}^+)=\pi/2$,
%we have ${\rm rad}(\xi_{x}^+)=\pi/2$.
%This implies  $x\in X_0^1$.
%\end{proof} 

%%%%%%%%%%%%%%%%%%%%%%%
 
\begin{lem} \label{lem:non-extend}
Let $\gamma:[0,\ell]\to X_0$ be an $X_0$-minimal geodesic starting from a point 
$x\in {\rm int}\, X_0^2$ such that
\[
  \gamma([0,t_0))\subset {\rm int}\, X_0^2, \,\,\, \gamma(t_0)\in\pa({\rm int}\,X_0^2),
\]
for some $0<t_0<\ell$.
Then we have 
\begin{enumerate}
\item $\gamma(t_0)\in \ca S^2\,;$
\item there is a unique limit
\[
     \lim_{t\to t_0+}\dot\gamma^Y_{\gamma(t_0),\gamma(t)}\in\Sigma_{\gamma(t_0)}(Y).
\]
\end{enumerate}
\end{lem}

\begin{slem} \label{slem:non-branch}
Let $x_i$ and $z_i\in X_0$ be sequences converging to a point $y\in X_0$ such that 
\begin{align} \label{eq:xyz=zyx}
  &  |x_i,y|_{X_0^{\rm int}}=|z_i,y|_{X_0^{\rm int}}, \quad
   \lim_{i\to\infty}\frac{|x_i,y|_{X_0^{\rm int}}+
|y,z_i|_{X_0^{\rm int}}}{|x_i,z_i|_{X_0^{\rm int}}}=1.
\end{align}
Take respective lifts $\tilde y,\tilde x_i,\tilde z_i\in C_0$ of $y,x_i, z_i$ such that 
$\tilde x_i\to \tilde y$ and $\tilde z_i\to \tilde y$.
Then we have 
\[
     \lim_{i\to\infty}\wangle^{C_0}
        \tilde x_i\tilde y\tilde z_i=\pi.
\]
\end{slem}
\begin{proof} 
Since $\eta_0$ is $1$-Lipschitz, Proposition \ref{prop:length'} implies 
\beq \label{eq:yx/yx=1}
\lim_{i\to\infty} |y,x_i|_{X_0^{\rm int}}/|\tilde y,\tilde x_i|_{C_0}=1, \quad
\lim_{i\to\infty} |y,z_i|_{X_0^{\rm int}}/|\tilde y,\tilde z_i|_{C_0}=1.
\eeq
From \eqref{eq:xyz=zyx} and 
the $1$-Lipschitzness of 
$\eta_0$, we have   
\beq \label{eq:xz>xz}
     |\tilde x_i,\tilde z_i|_{C_0}\ge |x_i,z_i|_{X_0^{\rm int}} \ge 2|x_i,y|_{X_0^{\rm int}}(1-o_i).
\eeq
Combining \eqref{eq:yx/yx=1}, \eqref{eq:xz>xz}
and \eqref{eq:xyz=zyx},
we have 
\beqq \label{eq:xz=zy+yz}
\lim_{i\to\infty}\frac{|\tilde x_i,\tilde y|_{C_0}+|\tilde y,\tilde z_i|_{C_0}}{|\tilde x_i,\tilde z_i|_{C_0}}=1,
\eeqq
from which the conclusion follows immediately.
\end{proof}

\begin{proof}[Proof of Lemma \ref{lem:non-extend}] (1)\,
Let $y:=\gamma(t_0)$, and $\sigma:[0,t_0]\to X_0$ be 
defined as $\sigma(t)=\gamma(t _0-t)$.
Suppose $y\in\ca S^1$, and set $q:=\eta_0^{-1}(y)$.
%%%
Since $\sigma((0, t_0])\subset {\rm int}\,X_0^2$, we obtain 
two distinct lifts $\tilde\sigma_k:[0,t_0]\to C_0$ of $\sigma$ from $q$\,
$(k=1,2)$.
Note that $\tilde\sigma_k$ is $C_0$-minimal.
In fact, if it is not the case, we would have a curve
$\tilde \rho$ joining $q$ and $\tilde\sigma_k(t_0)$
shorter than $\tilde \sigma_k$,
Then the curve $\eta_0\circ\tilde\rho$ is 
shorter than $\sigma$, which is a contradiction.

In particular, we have  
$\dot{\tilde\sigma}_1(0)\neq\dot{\tilde\sigma}_2(0)$.
Take a subsequence $t_i\to t_0+$ such that 
$\dot\gamma^Y_{y,\gamma(t_i)}(0)$ converges to a 
direction, say $v\in\Sigma_y(N_0)$.
Choose 
$\tilde v\in\Sigma_{q}(C_0)$ such that $d\eta_0(\tilde v)=v$.
Sublemma \ref{slem:non-branch} yields 
\[
\angle(\dot{\tilde\sigma}_k(0),\tilde v)=\pi.
\]
This is a contradiction to the non-branching property of geodesics in 
the Alexandrov space $\Sigma_q(C_0)$.
\par\n 
(2)\,Suppose there are distinct limits
\[
   v_1:=\lim_{t_i\to t_0+}\dot\gamma^Y_{\gamma(t_0),\gamma(t_i)}\in\Sigma_{\gamma(t_0)}(Y),\quad
v_2:=\lim_{s_i\to t_0+}\dot\gamma^Y_{\gamma(t_0),\gamma(s_i)}\in\Sigma_{\gamma(t_0)}(Y).
\]
Choose any lift $q\in\eta_0^{-1}(y)$ and 
let $\tilde\sigma$ be the lift of $\sigma$ from 
$q$.
Take $\tilde v_k\in\Sigma_q(C_0)$ such that 
$d\eta_0(\tilde v_k)=v_k$\, $(k=1,2)$.
Then as in (1), Sublemma \ref{slem:non-branch}
yields
\[
\angle(\dot{\tilde\sigma}(0),\tilde v_k)=\pi,
\]
which is a contradiction.
\end{proof}
%%%%%%%%%%%%%%%%%%%%%%%
\pmed

Let $X_0^{\rm int}$ and $X_0^{\rm ext}$
denote $X_0$ equipped with the intrinsic metric and
the extrinsic metric induced from $X$ respectively.

 We now make clear that
both $X_0^{\rm int}$ and $X_0^{\rm ext}$ define
the same topology and the same Hausdorff
dimension for $X_0$.

\begin{lem}\label{lem:same-topology}
For any fixed $x\in X_0$, we have 
\[
       B^{X_0^{\rm int}}(x,r) \subset  B^{X_0^{\rm ext}}(x,r) \subset  B^{X_0^{\rm int}}(x,(1+\tau_x(r))r).
\]

In particular, both $X_0^{\rm int}$ and $X_0^{\rm ext}$
define the same topology on $X_0$.
\end{lem}
\begin{proof}
The first inclusion is obvious.
The second inclusion follows from Lemma 
\ref{lem:eta1} for $x\in X_0^1$
and from Sublemma 
\ref{slem:eta-ball} for $x\in X_0^2$.
\end{proof}

Recall that  $m$ is the topological dimension of $X$,
$m:=\dim X$.

\begin{lem}\label{lem:dimNdimN0} 
For any nonempty open subset $U_0$ of $X_0$, we have 
$$
 \dim U_0=\dim_H^{X_0} U_0=\dim_H^{X} U_0=m-1,
$$
where $\dim_H^{X_0}$ and $\dim_H^{X}$ denote the Hausdorff dimension with respect to the intrinsic metric  and the extrinsic metric of $X_0$ respectively.
\end{lem}
\begin{proof} 
Since $\eta_0:C_0\to X_0^{\rm int}$ is surjective and $1$-Lipschitz, we have 
$\dim U_0\le \dim_H ^{X_0} U_0\le \dim_H C_0=m-1$.
%%%
We show that $U_0$ meets either 
${\rm int}\, X_0^1$
or ${\rm int}\, X_0^2$.
If $U_0$ meets $\ca S^2$, then Lemma \ref{lem:S2character} implies that 
$U_0$ meets ${\rm int}\, X_0^1$.
Suppose $U_0$ meets $\ca S^1$.
Since we may assume that $U_0$ does not meet 
$\ca S^2$, we have $U_0\cap X_0^2\subset 
{\rm int}\, X_0^2$.
Since $C_0$ is an Alexandrov space, 
it follows from  Lemma \ref{lem:eta'} and \cite[Theorem 9.5]{BGP}
that $U_0$ contains an open subset $V_0$ that is homeomorphic to $\R^{m-1}$. 
Together with Lemma \ref{lem:same-topology},
this implies that 
 $\dim_H ^{X_0} U_0\ge \dim_H ^{X} U_0\ge \dim U_0\ge \dim V_0=m-1$.
This completes the proof.
\end{proof} 

\pmed\n 
{\bf Limits of  local inradius collapsing.}
Finally we state some results on the limits of local inradius 
collapsing.

\begin{defn} \label{defn:local-inradius}
We say that  an open set $D$ of $X$ is  a {\it  locally inradius 
collapsed part} if and only if $D$ is contained in $X_0$.  
In this case,  $D$ does not meet $\ca S^2$ by
Lemma \ref{lem:S2character}.
\end{defn}

The following result follows essentially from \cite[Propositions 4.27,4.30 and Corollary 4.31]{YZ:inrdius}.

\begin{thm}[\cite{YZ:inrdius}] \label{thm:inradius-collapse}
For  a locally inradius 
collapsed part  $D$ of $X$, set 
$C_{0,D}:=\eta_0^{-1}(D)$.
Then we have the following.
\begin{enumerate}
\item $f:C_{0,D}\to C_{0,D}$ is locally isometric$\,;$
\item $D$ is isometric to the quotient space $C_{0,D}/f$. In particular, $D$ is  locally an Alexandrov space with 
curvature $\ge\nu\,;$
\item If both $X_0^1\cap D$ and $X_0^2\cap D$ are nonempty, then
$X_0^1\cap D$ is extremal in $D$, and $X_0^2\cap D$ is open dense in $D$.
\end{enumerate}
\end{thm}

\begin{rem} \label{rem:S1-inrad}
(1)\,In Theorem \ref{thm:inradius-collapse} (3),
$X_0^1\cap D\subset \ca S^1$ in our present terminology.

(2)\,  The notion of extremal subsets in a local
Alexandrov space $D$ is defined in a way similar to that in Section \ref{ssec:Alex}: A closed 
subset $E$ of $D$ is called {\it extremal} if  each point
$x\in E$ has a neighborhood that satisfies
the definition of extremal subsets in 
Section \ref{ssec:Alex}.
This definition coincides with the original one
when $D$ is (globally) an Alexandrov space.
\end{rem}

\begin{lem}\label{lem:loc-inradX2S1}
Let $U$ be an open  connected neighborhood of a point $x\in\ca S^1$ in $X$. Then $U$ is a part of local inradius collapse if and only if
$U\cap X_0\subset {\rm int} X_0^2\cup\ca S^1$.
\end{lem}
\begin{proof}
If  
 $U$ is a part of local inradius collapse,
then Theorem \ref{thm:inradius-collapse}
yields the conclusion: Or more directly,
Lemma \ref{lem:S2character} implies that $U$ never meets $\ca S^2$,
and Lemma \ref{lem:non-extend}
 implies that $U$ never meets ${\rm int} X_0^1$ by the connectedness of $U$.

Next suppose $U\cap X_0\subset {\rm int} X_0^2\cup\ca S^1$  and it is not a part of local inradius collapse.
Then for some $y\in U$,
there is a sequence 
$y_i\in X\setminus U$
converging to $y$.
Let $z_i$ be a nearest point of $X_0$ from
$y_i\in X\setminus X_0$.
For large enough $i$, we have $z_i\in U$.
However Lemma \ref{lem:single-interior}
implies $z_i\in {\rm int} X_0^1$.
This is a contradiction to the hypothesis on $U$.
\end{proof}

%
%\begin{lem}\label{lem:loc-inradX2S1}
%If $U$ is a nonempty open subset of $X_0$
%such that $U\subset {\rm int} X_0^2\cup\ca S^1$, then it is a part of local inradius collapse.
%\end{lem}
%\begin{proof}
%Suppose the conclusion does not hold. Then for some $y\in U$,
%there is a sequence 
%$y_i\in X\setminus U$
%converging to $y$.
%Let $z_i$ be a nearest point of $X_0$ from
%$y_i\in X\setminus X_0$.
%For large enough $i$, we have $z_i\in U$.
%However Lemma \ref{lem:single-interior}
%implies $z_i\in {\rm int} X_0^1$.
%This is a contradiction to the hypothesis on $U$.
%\end{proof}

\begin{thm}\label{thm:restf=isometry}
$f:C_0\setminus\overline{\tilde{\ca S}^2}\to C_0\setminus\overline{\tilde{\ca S}^2}$ is locally isometric.
\end{thm}

\begin{proof}
By Lemma \ref{lem:lift-S2}, we have
\[
      \overline{\tilde{\ca S}^2}=\eta_0^{-1}(\overline{{\ca S}^2}).
\]
For any $p\in C_0\setminus\overline{\tilde{\ca S}^2}$,
take a connected neighborhood $U$ of $x:=\eta_0(p)$ in $X_0$ such 
that $U$ never meets $\overline{{\ca S}^2}$.
If $U$ meets  both ${\rm int}X_0^1$ and ${\rm int}X_0^2$, then Lemma \ref{lem:non-extend} implies
$U$ meets $\ca S^2$, which is a contradiction.
Thus we have either $U\subset {\rm int}X_0^2\cup\ca S^1$ 
or  $U\subset {\rm int} X_0^1$. In the former case,  by  Lemmas \ref{lem:X02} and 
\ref{lem:loc-inradX2S1},
$U$ is a locally inradius collapsed part.
The conclusion follows from Lemma \ref{lem:f'} and Theorem \ref{thm:inradius-collapse}.
\end{proof}

 Next we exhibit an example having two locally inradius collapsed parts $U_1, U_2$ with $U_k\subset {\rm int} X_0^k$ \, $(k=1,2)$.

\begin{ex}\label{ex:intX01-open}
For $\e>0$, let $D^2(\e)$ be a nonnegatively curved two-disk
with $\diam(D^2(\e))\le 10\e$ 
such that $\pa D^2(\e)$ has a collar neighborhood isometric to 
$S^1_\e\times [0,\delta)$ for some $\delta>0$.
Note that $\lim_{\e\to 0}\delta=0$.
We consider the product $D^2(\e)\times\mathbb R$, and extend it as follows.
Choose a smooth function $f_\e:S^1_\e\times\mathbb R \to \R_+$
satisfying
\benu
 \item $\max f_\e=1\,;$
 \item $\displaystyle f_\e^{-1}(0)=S^1_\e\times\left(\bigcup_{m\in\mathbb Z}[m-0.1,m+0.1]\right)\,;$
 \item $\displaystyle f_{\e}^{-1}(1)=S^1_\e\times\left(\bigcup_{m\in\mathbb Z}[m+0.3, m+0.7]\right)\,;$
 \item $|\nabla\nabla f_\e|\le C\,;$
 \item $f_\e(x,y+1)=f_\e(x,y)$ for all $(x,y)\in S^1_\e\times\mathbb R$.
\eenu
Let $W_\e'$ denote the gluing of $D^2(\e)\times\mathbb R$ and the set
$$
G_\e:=\{ (x,y,z)\in S^1_\e\times\mathbb R\times [0,\infty)\,|\,0\le z\le f_\e(x,y)\}
$$
along their boundaries $\pa D^2(\e)\times\mathbb R$ and 
$S^1_\e\times\mathbb R\times\{ 0\}$.
The translation $y\to y+1$ defines 
an isometric $\mathbb Z$-action on $W_\e'$. We set 
$W_\e:=W_\e'/\mathbb Z$.
Let $U_\e$ denote  the $0.1$-neighborhood of the segment 
$[0.4,0.6]\times \{ 0.9-\e\}$ in $\mathbb R^2$. 
Note that $S^1_\e\times U_\e$ is isometrically embedded in $W_\e$.
Now consider $M_\e:=W_\e\setminus (S^1(\e)\times U_\e)$.
By a slight modification of $M_\e$ around $S^1_\e\times \pa U_\e$, 
we may assume $M_\e\in \ca M_b(3,0,\lambda, d)$
for some constants $\lambda, d$.
As $\e\to 0$, $M_\e$ converges to 
$$
X:=\{ (y,z)\in\mathbb R\times [0,\infty)|\, 0\le z\le f_0(y)\}/\mathbb Z
      \setminus {\rm int}\, \bar U_0,
$$
where $f_0$ and $\bar U_0$ are the limits of $f_\e$ and $U_\e$ respectively.
Remark that 
$$
X_0=\{ (y,z)\in\mathbb R\times [0,\infty)|\, z=f_0(y)\}/\mathbb Z\cup \pa \bar U_0,
$$

In this example, $U_1:=(-0.1,0.1)\times \{ 0\}$ and
$U_2:=(0.4,0.6)\times \{ 1\}$ are locally inradius collapsed parts
with $U_k\subset {\rm int} X_0^k$.
\end{ex} 

\begin{rem} \label{rem:deform-unbound}
All examples given so far were about the Gromov-Hausdorff convergence in the family
$\ca M_b(n, \kappa,\lambda,d)$.
It is easy to construct  such a 
Gromov-Hausdorff convergence in the family
$\ca M(n, \kappa,\lambda,d)$ that does not occur in $\ca M_b(n, \kappa,\lambda,d)$.
%%%

For instance, let $N$ be a closed convex domain in $\R^n$ with 
nonempty interior and with smooth boundary
except finitely many singular points $x_1,\ldots,x_k$ at  
$\pa N$
such that for each $1\le i\le k$,
a metric ball around $x_i$ in $N$ 
%$(B^N(x_i,r_i),x_i)$ 
is isometric to a metric ball 
around the vertex in the Euclidean cone 
over an $(n-1)$-dimensional  smooth disk with curvature $\ge 1$ and with convex boundary.
Then by a smoothing procedure, we can construct a family of 
smooth convex domains $M_\e$ in $\R^n$
contained in $N$
converging to $N$ in the Hausdorff distance
in $\R^n$. Note that 
$M_\e\in\ca M(n, 0,0,0,d)$
for some $d$ but 
$M_\e\notin\ca M_b(n, 0,\lambda,d)$
for any $\lambda>0$.
\end{rem}

\pmed
\setcounter{equation}{0}

%%%%%%%%%%%%%%%%%%%%%%%%%%%%%%%%%%%%%%
\section{Infinitesimal Alexandrov structure} \label{sec:int-ext}
Towards 
the proof of Theorem \ref{thm:alex},
we first define the notion of infinitesimally Alexandrov  as follows.
Let $N$ be a geodesic space and $N_0$
a closed subset of $N$.

\begin{defn}\label{defn:inf-alex} 
(1)\,
We say that $N$ is {\it infinitesimally Alexandrov} if 
for each $x\in N$, we have the following:
\begin{itemize}
\item  A geodesic space $\Sigma_x(N)$, called the {\it space of directions at $x$},
is defined and is a compact Alexandrov space with curvature $\ge 1\,;$
 \item There exists a unique limit
\beq\label{eq:def-tangcone}
       T_x(N):=\lim_{\e\to 0}\biggl(\frac{1}{\e} N, x\biggr),
\eeq
and it is isometric to the Euclidean cone 
$K(\Sigma_x(N))$ 
over $\Sigma_x(N)$. 

\end{itemize}
The space $T_x(N)$ is called the
  {\it tangent cone} of $N$ at $x$.

The nonnegative integer defined as 
\[
      {\rm rank}(N):= \sup_{x,y \in N}|\dim T_x(N)-\dim T_y(N)|
\]
is called the {\it rank} of $N$.
\n 

(2)\, We say that  $N_0$ is {\it infinitesimally sub-Alexandrov} if for each $x\in N_0$, 
the following holds: 
\begin{itemize}
\item  A geodesic space $\Sigma_x(N_0)$, called the {\it space of directions at $x$},
is defined  as a closed subset of $\Sigma_x(N)$ and the intrinsic metric 
$\Sigma_x(N_0)^{\rm int}$ of $\Sigma_x(N_0)$ is a compact Alexandrov space with curvature $\ge 1\,;$
\item Under the convergence \eqref{eq:def-tangcone}$, (N_0,x)$ converges to $(K(\Sigma_x(N_0)),o_x)$,
denoted by $T_x(N_0)$ and called the tangent cone of $N_0$ at $x$.
\end{itemize}
The rank of $N_0$ is defined similarly.
\end{defn}

\psmall
A lot of infinitesimally Alexandrov
spaces appear as closed subsets of Alexandrov spaces. 
Let a geodesic space $N$ be embedded in an Alexandrov 
space $Y$ as a closed subset, and let $N_0$ be a closed subset 
of $N$.
For each $x\in N$, the  space of directions 
\beqq \label{SigmaSigma}
   \Sigma_x(N)\subset \Sigma_x(Y)
\eeqq
is defined as a closed subset of $\Sigma_x(Y)$ as in Section \ref{ssec:Alex}.
If $x\in N_0$,  $\Sigma_x(N_0)$
is defined in the same way, where 
we consider the intrinsic metric 
$\Sigma_x(N_0)^{\rm int}$ of $\Sigma_x(N_0)$.

\begin{ex} \label{ex:infin-sub}
(1)\,  Let $N$ be a submanifold with boundary of a Riemannian manifold $Y$.
If $N$ is closed in $Y$, then it
 is infinitesimally Alexandrov, and 
$\pa N$ is infinitesimally sub-Alexandrov.

\par\n
(2)\, Any closed convex domain $N$ of a Riemannian manifold $Y$ is infinitesimally Alexandrov, and $\pa N$ is infinitesimally sub-Alexandrov (see \cite{Buy}).
\par\n
(3)\, Let $I:=[-2,2]$, and choose a function $f:I\to\R_+$ such that $f$ is smooth on $\mathring{I}$ and $f^{-1}(0)=[-1,1]\cup\pa I$.
Consider the closed subsets $N$ and $N_0$ of $\R^3$ defined as
\begin{align*}
N&=\{ (x,y,z)\in\R^3\,|\, \sqrt{y^2+z^2}\le f(x), x\in I\},\\
N_0&=\{ (x,y,z)\in\R^3\,|\, \sqrt{y^2+z^2}=f(x), x\in I\}.
\end{align*}
Here we assume that $N$ is smooth at the points
$\pa I\times \{ (0,0)\}$.
Then $N$ is infinitesimally Alexandrov with ${\rm rank} N=2$, and $\pa N$ is infinitesimally sub-Alexandrov
with ${\rm rank} N_0=1$.
\end{ex} 

In what follows,  
let $N$ and $N_0$ be as in Section \ref{sec:non-inradius}. Using the extension $Y$ of $N$,
we define the spaces of directions $\Sigma_x(N)$ \,($x\in N$)   
and $\Sigma_x(N_0)$  \,($x\in N_0$)
as in Section \ref{ssec:Alex}.
We are going to show that $N$ is infinitesimally Alexandrov
and the boundary $N_0$ of $N$ is infinitesimally sub-Alexandrov
and that the isometry classes of 
$\Sigma_x(N)$  and $\Sigma_x(N_0)$ actually do not depend on the choice of the extension $Y$.

We often use the identification $N=X^{\rm int}$.

\begin{slem} \label{slem:distance-raio}
Fix $C>1$.
For arbitrary distinct points $x_0, x_1, x_2\in N$
with  
$C^{-1}\le |x_0,x_1|_N/|x_{1},x_2|_N
\le C$, 
we have
\[
|\tilde\angle^{N} x_0x_1x_2 -  \tilde\angle^{Y} x_0x_1x_2| <\tau_C(t),
\]
where  $t:=\max\{ |x_{0}, x_1|_N, |x_1,x_2|_N\}$ and 
$\tilde\angle^{N} x_0x_1x_2$ and $\tilde\angle^{Y} x_0x_1x_2$ denote
the comparison angles  of geodesic triangles with vertices $x_0, x_1, x_2$ at $x_1$
with respect to the intrinsic  and the extrinsic distances $d^{X^{\rm int}}=d^N$ and $d^{X}=d^Y|_X$ respectively.
\end{slem}
\begin{proof}
By Lemma \ref{lem:deviation},
we have 
\[
\left| \frac{|x_{i-1},x_i|_N}{|x_{i-1},x_i|_Y} -1\right|
     < O(t^2).
\]
Then the conclusion is immediate from the law of cosines.
\end{proof}

\begin{defn}\label{defn:admissible}
For $x\in N_0$, let $\ca A_x(N_0)$ be the set of all curves  $\sigma$ starting from $x$ that can be written as 
 $\sigma=\eta_0\circ\tilde\gamma$ via  
$C_0$-minimal geodesics $\tilde\gamma$ 
starting from $p$ with $\eta(p)=x$.
Let $\ca A_x({\rm int} N)$ denote the set of all 
$N$-minimal geodesics $\gamma$ starting from $x$ 
such that 
\beq \label{eq:liminft|N0|}
   \liminf_{t\to +0} \frac{|\gamma(t),N_0|_N}{t}>0,
\eeq
if such $\gamma$ exists.
Set 
\beqq
 \ca A_x:=\ca A_x(N_0)\cup \ca A_x({\rm int} N).
\eeqq  
We call an element of $\ca A_x$ {\it admissible}.
\end{defn}

Note that $\ca A_x$ does not depend on $Y$.

\begin{slem} \label{slem:N-Yminimal}
If $\gamma\in  A_x({\rm int} N)$, then 
$\gamma$ is $Y$-minimal 
at small neighborhood of $t=0$.
\end{slem}
\begin{proof} 
Suppose the conclusion does not hold.
Take any sequence $t_i\to 0$. If a $Y$-minimal geodesic 
$\gamma_i$ from $x$ to $\gamma(t_i)$
 is  included in $N$, then  $\gamma|_{[0,t_i]}$
is $Y$-minimal.
Therefore we may assume that  there is $s_i\in (0,t_i)$
satisfying
$\gamma_i(s_i)\in N_0$ and $\gamma_i((s_i,t_i])\subset {\rm int}N$. 
Consider the convergence
\[
  \biggl(\frac{1}{s_i} Y, \gamma_i(s_i)\biggr) \to (T_x(Y), v),
\]
where  $v:=\lim_{i\to\infty}\dot\gamma_i(0)\in T_x(N_0)$.

On the other hand, by Lemma \ref{lem:deviation}, under the convergence
\[
     \biggl(\frac{1}{t_i} Y, \gamma(t_i)\biggr) \to (T_x(Y), w),
\]
both $\gamma_i$ and $\gamma|_{[0,t_i]}$  converge to the  geodesic  from $o_x$ to  
$w$ together with $v_i\to w$,  which implies
$w\in\Sigma_x(N_0)$.
However, if
$\alpha$ denotes the value of \eqref{eq:liminft|N0|},
then we have 
\[
           |w, T_x(N_0)|\ge \alpha >0.
\]
This is a contradiction.
\end{proof} 

In Sublemma \ref{slem:N-Yminimal},
the condition \eqref{eq:liminft|N0|} is essential.
As the following example shows,
the condition $\gamma\setminus \{ x\} \subset
{\rm int}N$ without \eqref{eq:liminft|N0|}  does not imply the $Y$-minimality of $\gamma$.

\begin{ex}   \label{ex:counter-Nminimal}
This is similar to Example \ref{ex:cusp}. 
Let $I:=[0,1]$ and choose a smooth function $f:I\to\R_+$ 
such that 
\begin{itemize}
\item $f^{-1}(0)=\pa I$, $f'=f''=0$ at $\pa I\,;$
\item $f(t)=t^3$ on $[0, 1/2]$.
\end{itemize}

Let $B_1$, $B_2$ be two copies of $\{(x,y)|\, x\in
I, 0\leq y\leq f(x)\}$. 
For any $\e>0$, let $A_\epsilon$ be the intersection of $\pa B(I\times\{(0, 0)\}, \e)\subset\mathbb{R}^3$ and the half space $z\leq0$ of $\mathbb{R}^3$.  
We glue $A_\epsilon$, $B_1$ and $B_2$ together by gluing $I\times\{ (-\epsilon,0)\}\subset A_\epsilon$ with $I\times\{0\}\subset B_1$,
and gluing  $I\times\{ (\epsilon,0)\}\subset A_\epsilon$ with $I\times\{0\}\subset B_2$ respectively. If $M_\epsilon$ denotes the result of this gluing, then $M_\e$ is an element of 
$\ca M_b(2,0,\lambda,d)$ for certain 
$\lambda>0$ and $d>0$.
Here we choose the original warping  function 
$\phi:[0,t_0]\to \R_+$ in Section \ref{ssec:gluing} in such a way 
that 
\[
    \phi''  = -K(\lambda,\epsilon_0,t_0)\phi>0,
   \quad t_0=1/(2\lambda),
    \quad \phi(t_0)=1/2
\]
in addition to \eqref{eq:phi}.
%%%
Let $N$ be the gluing of $B_1$ and $B_2$ along $I\times\{0\}$. 
Then $M_\epsilon$ converges to $N$ as $\epsilon\to0$.
Let $\gamma:[0,1]\to I\times \{ 0\}\subset N$ be the canonical map, which is $N$-minimal.
%%%

We show that
$\gamma$ is not minimal in the limit $Y$
of the extension $\tilde M_i$.
Fix any $a\in (0,1/2)$, and denote by 
$\Sigma\subset N_0$ the curve defined as 
the graph of $f$ on $[0,a]$.
Let $\ell $ be the lengths of 
 $\Sigma$.
%%%%
Take small $a$ with 
\[
   \ell \ll \pi/\lambda.
\]
Note that 
\beq\label{eq:a<ell<a}
     a+a^3/3  <   \ell=\int_0^a\sqrt{1+3t^2}\,dt
                 <a+a^3/2,
\eeq
for small enough $a$.
Let $D^2(1/\lambda)$ be the Euclidean disk of radius
$1/\lambda$ with the canonical metric
$g=dt^2+h(t)^2 d_{\mathbb S^1(1/\lambda)}^2$,
where $h(t)=1-t\lambda$.
%%%
From the convexity of $\phi$,  we have 
\beq\label{eq:phileh}
 \phi\le h  \quad \text{on $[0,t_0]$}.
\eeq
Let $\tilde\Sigma$ be an  arc of length $\ell$ of 
$\pa D^2(1/\lambda)$.
By \eqref{eq:phileh},
the canonical embedding 
$\iota:[0,t_0]\times_\phi\Sigma\to D^2(1/\lambda)$ is expanding.
Let  $\tilde\sigma$ be the Euclidean geodesic
between the endpoints of $\tilde\Sigma$.
Then we have 
\beq\label{eq:L<L-L3}
     L(\tilde\sigma) = \frac{2}{\lambda}\sin \lambda \ell/2
        < \ell-\lambda^2 \ell^3.
\eeq
\vspace{0.2cm}
%%%%%%%%%%%%%%%%%%%%
\begin{center}
\begin{tikzpicture}
[scale = 0.4]
%\filldraw[fill=lightgray] 
\filldraw[fill=gray, opacity=.1] 
 (0,0) to [out=3, in=200] (12,2.3)
to [out=-90, in=90] (12, -2.3)
to [out=160, in=-3] (0,0);

\draw[thick] (0,0)--(12,0);
\draw [-, thick] (0,0) to [out=3, in=200] (12,2.3);
\draw [-, thick] (0,0) to [out=-3, in=160] (12,-2.3);
\fill (12,0) circle (0pt) node [right] {\small{$\gamma$}};
\fill (12,2.3)  circle (0pt) node [right] {\small{$N_0$}};
\fill  (12,-2.3) circle (0pt) node [right] {\small{$N_0$}};
\fill  (0,0) circle (2pt) node [left] {\small{$x$}};
\fill  (9.5,0) circle (2.5pt) node [below] {\tiny{$\gamma(a)$}};
\fill (9.5,1.45) circle (2.5pt);
\draw [thick, dotted] (9.5,0)-- (9.5,1.45) ;
\fill (9.5,1.45) circle (0pt) node [above] {\small{$y$}};
\fill (5,0.25) circle (0pt) node [above] {\small{$\Sigma$}};
\draw [-, thick] (0,0) to [out=35, in=165] (9.5,1.45);
\fill (5,1.7) circle (0pt) node [above] {\small{$\iota^{-1}\circ\tilde\sigma$}};

%%%%%%%%%%%%%%%%%%%%%%%%%%%
\draw[thick] (20,0) circle[radius=4];
\draw [thick] (16.3,-1.5)-- (23.7,-1.5) ;
\fill (20,-4.2) circle (0pt) node [above] {\small{$\tilde\Sigma$}};
\fill (20,-1.7) circle (0pt) node [above] {\small{$\tilde\sigma$}};
\fill (20,1) circle (0pt) node [above] {\small{$D^2(\lambda)$}};

\end{tikzpicture}
\end{center}
\vspace{-1cm}  
\begin{figure}[htbp]
  \centering
  \caption{}
  \label{fig:example-nonminimal2}  
\end{figure}  

Let $x:=\gamma(0)=(0,0)\in N_0$ and $y:=(a,f(a))\in N_0$.

 It follows from \eqref{eq:a<ell<a},  \eqref{eq:phileh} and \eqref{eq:L<L-L3}  that 
\begin{align*}
 d^Y&(x,\gamma(a))\le L(\iota^{-1}\circ \tilde\sigma)+d^N(y,\gamma(a)) \\
&\le L(\tilde\sigma)+a^3 
        < \ell-\lambda^2 \ell^3 + a^3 \\
   &< a+a^3/2-\lambda^2(a+a^3/3)^3+a^3 \\
&=a-a^3(\lambda^2-3/2) +O(a^4)
     <a=L(\gamma|_{[0,a]}),
\end{align*}
for any small $a>0$ if $\lambda>\sqrt{3/2}$.
Namely $\gamma$ is not $Y$-minimal
in any neighborhood of $t=0$.
\end{ex}   
\pmed

The following sublemma shows that 
any admissible curve has a definite direction.

\begin{slem} \label{slem:unique-tangent}
For each $x\in N_0$ and any $\sigma\in \ca A_x$, 
there is a unique limit, \, 
\beq\label{eq:unique-limit(sigma)}
\dot\sigma(0):=
\lim_{t\to 0} \dot\gamma^Y_{x,\sigma(t)}(0) \in \Sigma_x(N)
\subset\Sigma_x(Y).
\eeq
\end{slem}

\begin{proof} 
This is trivial for $\sigma\in\ca A_x({\rm int}\,N)$
by Sublemma \ref{slem:N-Yminimal}.
Let $\sigma=\eta_0\circ\tilde\gamma\in\ca A_x(N_0)$.
Since the convergence from \eqref{eq:deta}  to \eqref{eq:eta-infty}
does not depend on the choice of $t_i$, 
we have  $\dot\sigma(0):=d\eta_0(\dot{\tilde\gamma}(0))$.
Thus the sublemma follows.
\end{proof}

In Definition \ref{defn:inf-alex},  
the angle in $\Sigma_x(N)$ is defined in the extrinsic way. However, we can express the angle between admissible curves in terms of the original 
intrinsic metric of $N$ as follows.

\begin{lem} \label{lem:define-Sigma(N)}
For  $\sigma_1, \sigma_2\in \ca A_x$,
the angle 
$\angle_x(\sigma_1,\sigma_2)$ 
 in $\Sigma_x(N)$  can be expressed as
\beq \label{eq:int-angle-def}
   \angle_x(\sigma_1,\sigma_2) :=
       \lim_{t\to 0}  
      \tilde\angle^{N}\sigma_1(t) x \sigma_2(t).
\eeq 
In particular, $\angle_x(\sigma_1,\sigma_2)$
is uniquely determined by $N$.
\end{lem}
\begin{proof}
The limit of the right hand side of \eqref{eq:int-angle-def} certainly exists, since 
by Sublemma \ref{slem:distance-raio},  it 
coincides with
\beq \label{eq:angle-int=ext}
   \lim_{t\to 0}  
      \tilde\angle^{Y}\sigma_1(t) x \sigma_2(t)=\angle_x^{Y}
         (\dot\sigma_1(0),  \dot\sigma_2(0)).
\eeq
\end{proof}

\begin{lem} \label{lem:dense-sigma(0)}
The set of directions of admissible curves is dense in 
$\Sigma_x(N)$. More explicitly, 
$\{ \dot\sigma(0)\in \Sigma_x(N)\,|\,\sigma\in\ca A_x\}$
is dense in $\Sigma_x(N)$.
Similarly,
$\{ \dot\sigma(0)\in \Sigma_x(N_0)\,|\,\sigma\in\ca A_x(N_0)\}$
is dense in $\Sigma_x(N_0)$.
\end{lem} 
\begin{proof} Take  a lift $p\in C_0$ of $x$.
The latter conclusion follows from the denseness of
$\Sigma_p'(C_0)$ (see Section \ref{ssec:Alex})
and the surjectivity of the $1$-Lipschitz map
$d\eta_0:\Sigma_p(C_0)\to\Sigma_x(N_0)$.

For any $v\in\Sigma_x(N)\setminus \Sigma_x(N_0)$ and $\e>0$, choose a $Y$-minimal geodesic
$\gamma$ such that $\angle(\dot\gamma(0),v)<\e$
and $\angle(\dot\gamma(0),\xi_x^+)>\pi/2$.
Then clearly we have $\gamma\in\ca A_x({\rm int} N)$.
\end{proof}

To define $\Sigma_x(N)$ and $\Sigma_x(N_0)$, we
used the metric of $Y$. However from Lemmas \ref{lem:define-Sigma(N)} and \ref{lem:dense-sigma(0)},
we have the following immediately.

\begin{prop}\label{prop:uniqueSigma} The isometry classes of both
$\Sigma_x(N)$ and $\Sigma_x(N_0)$ do not depend on the 
choice of the extension $Y$ of $N$.
\end{prop}

\psmall  

\begin{prob} \label{prob:angle}
For any  $N$-geodesic  or $N_0$-geodesic $\gamma$ starting from $x\in N_0$, determine whether the unique existence for 
$\dot\gamma(0)$ in
Sublemma \ref{slem:unique-tangent} holds true.
\end{prob}

From now on, 
we make the
identification 
\[
\Sigma_x(N)=\Sigma_x(X), \quad \Sigma_x(N_0)=\Sigma_x(X_0).
\] 
\psmall   
\begin{lem} \label{lem:Sigma-Alex}
For each $x\in X_0$, we have the following$:$
\benu
\item  $\Sigma_x(X)$ is  convex in $\Sigma_x(Y)$, and  an Alexandrov space 
with curvature $\ge 1$.
%, and so is $\Sigma_x(N)$.
\item   If either $x\in  X_0^2$ or 
${\rm rad}(\xi_x^+)=\pi/2$, then 
$\Sigma_x(X)=\Sigma_x(X_0)$.
% and it is an  
% Alexandrov space with curvature $\ge 1$.
%, and so is $\Sigma_x(N_0)\subset\Sigma_x(N)$.
\eenu
\end{lem}

\begin{rem} \label{rem:Sigma=Alex}
The case of $x\in{\rm int}\, X_0^1$ in Lemma \ref{lem:Sigma-Alex}
is also true. The proof of this case is more involved, and given in 
Corollary \ref{cor:Sigma0=Alex} of Section \ref{ssec:iso-inv}. This is 
related with an open question: Is the boundary of an Alexandrov space is again 
an Alexandrov space ?
\end{rem}

\begin{proof}[Proof of Lemma \ref{lem:Sigma-Alex}]
(1)\, Let $\bm{\xi}_x$ be the set of all perpendicular directions at $x$.
Since $\Sigma_x(X)=\Sigma_x(Y)\setminus\mathring{B}(\bm{\xi}_x,\pi/2)$ is 
convex, the conclusion certainly follows.

(2)\,
By Lemma \ref{lem:double-sus} and Proposition \ref{prop:perp+horiz},
we have $\Sigma_x(X_0)=\Sigma_x(X)$
if $x\in  X_0^2$ or 
${\rm rad}(\xi_x^+)=\pi/2$.
The conclusion follows  from (1) immediately.
\end{proof}

For the proof of Theorem \ref{thm:alex},
we introduce the following family $\ca B_x$ a bit different from $\ca A_x$,
which is more effective in the proof of Theorem \ref{thm:alex}.
Let us consider 
\[
 \ca B_x:=\{ \sigma=\pi\circ\gamma^Y_{x,y}\,|\, y\in N\}.
\]

\begin{slem}\label{slem:directionBx}
Any element of $\ca B_x$ has a direction  at $x$
in the sense of Sublemma \ref{slem:unique-tangent}.
\end{slem}
\begin{proof}
Let $\sigma\in\ca B_x$ be expressed as $\sigma=\pi\circ\gamma$, where $\gamma=\gamma^Y_{x,y}$.
Set $\hat\gamma:=\gamma\setminus \{x\}$.
If a neighborhood of $x$ in $\hat\gamma$ is contained in 
either $X$ or $Y\setminus X$, then
the conclusion is obvious from Lemma \ref{lem:not-perp}.
Thus we may assume that 
$v:=\dot\gamma(0)\in\Sigma_x(X_0)$.
 Take $t_i\to 0$ with $\gamma(t_i)\in X_0$.
For any sequence $s_j\to 0+$ and for any $i$, take $j=j(i)$ with $s_j<t_i$.
By Lemma \ref{lem:deviation}(1), we have
 $|\gamma(s_j), \sigma(s_j)|_Y< Ct_i s_j$, which implies
$\lim_{j\to\infty}\angle(v, \dot\gamma^Y_{x,\sigma(s_j)}(0))=0$.
This completes the proof.
 \end{proof}

\begin{proof}[Proof of Theorem \ref{thm:alex}(1)] 
In view of Lemma \ref{lem:Sigma-Alex},  it suffices to show the convergence 
\beq\label{eq:conv=NtoKN}
     \lim_{\e\to 0} \biggl(\frac{1}{\e} N, x\biggl)= (K(\Sigma_x(N)), o_x).
\eeq
The proof of \eqref{eq:conv=NtoKN} is similar to
that of Proposition \ref{prop:tang-cone}.
However we have to proceed in terms of 
the intrinsic metric of $N$ rather than the extrinsic metric
induced from $Y$.

For any $R>0$ and $\delta>0$, 
from the compactness of  
$\Sigma_x(N)$ together with Proposition \ref{prop:tang-cone}, 
we can take directions 
$v_{1}, \ldots, v_L\in\Sigma_x(N)$
 satisfying 
\begin{enumerate}
\item $\{ v_{1},\ldots, v_L\}$ is $\delta$-dense in $\Sigma_x(N)\,;$
\item there are curves  $\sigma_i:[0,R]\to N/\e$\,\,$(1\le i\le L)$ in $\ca B_x$ joining 
$x$ to points $x_i\in N$ with $\sigma_i=\pi\circ\gamma^Y_{x,x_i}$
such that 
$|x,x_i|_{Y/\e}=R$ and $\dot{\sigma}_i(0)= v_i$.
\end{enumerate} 
%%%%%%%
\vspace{0.3cm}
%%%%%%%%%%%%%%%%%%%%
\begin{center}
\begin{tikzpicture}
[scale = 0.5]
\filldraw[fill=lightgray,opacity=.1] 
%\filldraw[fill=gray, opacity=.1] 
[-,  thick] (0,0) circle [x  radius=5,y  radius=1.25,rotate=0];
\draw [dotted, thick] (-5,0) to [out=-85, in=-95] (5,0);
\fill (5.2,0.2) circle (0pt) node [right] {$N_0$};
\fill (0,0.3) circle (2pt) node [above] {{\small $x$}};
\fill (4,-2.2) circle (0pt) node [right] {$N$};

\draw [-, thin] (0,0.3) to [out=20, in=190] (2.5, 1.1);
\draw [-, thin] (0,0.3) to [out=160, in=-10] (-2.5,1.1);
\draw [-, thin] (0,0.3) to [out=180, in=15] (-5,0);
\draw [-, thin] (0,0.3) to [out=185, in=27] (-4,-0.7);
\draw [-, thin] (0,0.3) to [out=200, in=35] (-2.5,-1.1);
\draw [-, thin] (0,0.3) to [out=225, in=75] (-0.9, -1.3);
\draw [-, thin] (0,0.3) to [out=-45, in=105] (0.9, -1.3);
\draw [-, thin] (0,0.3) to [out=-20, in=145] (2.5, -1.1);
\draw [-, thin] (0,0.3) to [out=-5, in=153] (4,-0.7);
\draw [-, thin] (0,0.3) to [out=0, in=165] (5,0);
%%%%%%
\draw [dotted, thick] (0,0.3) --(-3.8,-1.5);
\draw [dotted, thick] (0,0.3) --(-2.9,-1.8);
\draw [dotted, thick] (0,0.3) --(-1.9,-2.1);
\fill (-1.7,-1.9) circle (0pt) node [below] {{\tiny $\sigma_{L-2}$}};
%%%%
\draw [dotted, thick] (0,0.3) --(0,-2.4);
\fill (0,-2.2) circle (0pt) node [below] {{\tiny $\sigma_{L-1}$}};
\draw [dotted, thick] (0,0.3) --(1.9,-2.1);
\fill (1.9,-1.9) circle (0pt) node [below] {{\tiny $\sigma_{L}$}};
\draw [dotted, thick] (0,0.3) --(2.9,-1.8);
\draw [dotted, thick] (0,0.3) --(3.8,-1.5);

\fill (-3.8,-0.7) circle (0pt) node [right] {{\small $\sigma_2$}};
\fill (-2.1,-0.95) circle (0pt) node [right] {{\small $\sigma_3$}};
\fill (-5, -0.2) circle (0pt) node [right] {{\small $\sigma_1$}};
%\fill (-2.5, 1.2) circle (0pt) node [below] {{\small $\sigma_K$}};

\end{tikzpicture}
\end{center}
\vspace{-1cm}  
\begin{figure}[htbp]
  \centering
  \caption{}
  \label{fig:space-direc}  
\end{figure}  

By Lemma \ref{lem:deviation}  and Sublemma \ref{slem:directionBx},
we have 
\beq \label{eq:deviat(v-sigma)}
 \angle(\dot\sigma_{x,y_i}(0), v_i)<\tau_{R,\delta}(\e)
\eeq
for all $y_i\in\sigma_i$ and $1\le i\le L$.
%%%%%%%
%Let $\sigma_i=\pi\circ\gamma^Y_{x,x_i}$.
For any $y\in B^{N/\e}(x, R)$,  let $\sigma_{x,y}$ be an element of $\ca B_x$ joining $x$ to $y$  defined as
\[
      \sigma_{x,y}:=\pi\circ \gamma^Y_{x,y}.
\]
Set $K_x(N):=K(\Sigma_x(N))$.
We define $\varphi_{\e}:B^{N/\e}(x, R)\to B^{K_x(N)}(o_x, R)$ by
\[
          \varphi_{\e}(y):= |x,y|_{N/\e}\,\dot\sigma_{x,y}(0),
\]
where $\dot\sigma_{x,y}(0)\in\Sigma_x(N)$.
%%%%%%%%
We show that $\varphi_{\e}$ provides a
$(\tau_R(\e,\delta)+\tau_{R,\delta}(\e))$-approximation.

We first show that for arbitrary $y_i\in\sigma_i$ and $y_j\in\sigma_j$\,$(1\le i, j\le L)$
\beq \label{eq:|yiyj|}
  | |y_i,y_j|_{N/\e}-|\varphi_{\e}(y_i),\varphi_{\e}(y_j)|_{K_x(N)}|<\tau_{R}(\e,\delta)+\tau_{R,\delta}(\e).
\eeq
In fact, letting $w_i:=\uparrow_x^{y_i}\in\Sigma_x(Y)$, by Lemma \ref{lem:not-perp},
we have  
\beq\label{eq:angle(v,w)}
\angle(v_i,w_i)<\tau_{R,\delta}(\e)
\eeq
for all $y_i\in\sigma_i\setminus \{ x\}$. 
From Lemma \ref{lem:deviation}, we have 
\beq\label{eq:deviation}
   \frac{|\,\,, \,\,|_{N/\e}}{|\,\,, \,\,|_{Y/\e}}< 1+\tau_R(\e)
\eeq  
on $B^{N/\e}(x,R)$.
Set 
$t_{i,\e} :=|x,y_i|_{Y/\e}$ and $s_{i,\e} :=|x,y_i|_{N/\e}$.
Now we use the symbol $a\sim b$ if $|a-b|<
\tau_R(\e,\delta)+\tau_{R,\delta}(\e)$. By \eqref{eq:angle(v,w)}, \eqref{eq:deviation} and 
Lemma \ref{lem:define-Sigma(N)},
we certainly have
\begin{align*}
|y_i,y_j|_{N/\e} &\sim  |y_i,y_j|_{Y/\e}\,\,(\because\eqref{eq:deviation} )
  \\
&\sim \bigl((t_{i,\e})^2 + (t_{j,\e})^2 - 2t_{i,\e} t_{j,\e} \cos |w_i,w_j|_{\Sigma_x(Y)}\bigr)^{1/2} \,\,\\
&\sim 
\bigl((s_{i,\e})^2 + (s_{j,\e})^2 - 2s_{i,\e} s_{j,\e} \cos
|w_i,w_j|_ {\Sigma_x(Y)}\bigr)^{1/2} \,\,(\because\eqref{eq:angle(v,w)},\eqref{eq:deviation})\\
&= 
\bigl((s_{i,\e})^2 + (s_{j,\e})^2 - 2s_{i,\e} s_{j,\e} \cos |w_i,w_j|_ {\Sigma_x(N)}\bigr)^{1/2} \,\,(\because\, \text{Lemma \ref{lem:Sigma-Alex}})\\
&= |\varphi_{\e}(y_i), \varphi_{\e}(y_j)|_{K_x(N)}.
\end{align*}

Next, for any $y\in B^{N/\e}(x,R)$,
let $\sigma_{x,y}=\pi\circ\gamma^Y_{x,y}$.
Take $1\le i\le L$ with 
$\angle(\dot\sigma_{x,y}(0), \dot\sigma_i(0))<\delta+
\tau_R(\e)$.
Let $y_i$ be the point of $\sigma_i$ such that 
$|x,y_i|_{N/\e}=|x,y|_{N/\e}$.
%%%
From \eqref{eq:deviat(v-sigma)}, we have  
\[
\angle(\dot\gamma^Y_{x,y}(0),\dot\gamma^Y_{x,y_i}(0))\le
\angle(\dot\sigma_{x,y}(0), \dot\sigma_i(0))+\tau_R(\e)+
\tau_{R,\delta}(\e).
\]  
Since $Y$ is an Alexandrov space,
we immediately have 
\beq
\begin{aligned} \label{eq:|y,y_i|}
|y,y_i|_{N/\e}&\le |y,y_i|_{Y/\e} \,(1+\tau_R(\e)) \\
&\le R(\delta+\tau_{R,\delta}(\e))(1+\tau_R(\e)) \\
       & = \tau_R(\e,\delta) + \tau_{R,\delta}(\e).
\end{aligned}
\eeq
In view of \eqref{eq:deviat(v-sigma)}, we have 
\beq
\begin{aligned} \label{eq:|varphi(y),varphi(yi)|}
 &|\varphi_{\e}(y), \varphi_{\e}(y_i)|_{K_x(N)} \\
    &=||x,y|_{N/\e}\,\dot\sigma_{x,y}\,(0),|x,y_i|_{N/\e}\,\dot\sigma_{x,y_i}(0)|_{K_x(N)}\\
&\le ||x,y|_{N/\e}\,\dot\sigma_{x,y}\,(0),|x,y_i|_{N/\e}\,\dot\sigma_{i}(0)|_{K_x(N)}+\tau_{R,\delta}(\e)  \\
  &<\tau_R(\e, \delta)+\tau_{R,\delta}(\e).
\end{aligned}
\eeq
For any other point $y'\in B^{N/\e}(x,R)$, we take $j$ and $y_j$ in the same way as the above
$i$ and $y_i$ for $y$.
It follows from \eqref {eq:|y,y_i|}  and \eqref{eq:|varphi(y),varphi(yi)|} that 
 \begin{align*}
  &| |y,y'|_{N/\e} - |y_i,y_j|_{N/\e} | < \tau_R(\e,\delta)    
      +\tau_{R,\delta}(\e),\\
  &  | |\varphi_{\e}(y), \varphi_{\e}(y')|_{K_x(N)} - |\varphi_{\e}(y_i), \varphi_{\e}(y_j)|_{K_x(N)} | < \tau_R(\e,\delta) +\tau_{R,\delta}(\e).
\end{align*}
From \eqref{eq:|yiyj|}, we conclude that 
\[
    | |\varphi_{\e}(y), \varphi_{\e}(y')|_{K_x(N)}-|y,y'|_{N/\e}|< \tau_R(\e,\delta)+\tau_{R,\delta}(\e).
\]

Finally for any $v\in B^{K_x(N)}(o_x,R)$, take 
$1\le i\le L$ such that $\angle(v,v_i)<\delta$.
Let $y\in\sigma_i$ be such that $|x,y|_{N/\e}=|v|$.
Then obviously we have $|v,\varphi_{\e}(y)|<R\delta$.
Thus $\varphi_\e$ is a $(\tau_R(\e,\delta)+\tau_{R,\delta}(\e))$-approximation.
This completes the proof of Theorem \ref{thm:alex}(1).
\end{proof}

\begin{rem} \label{rem:Tx(X0)}
From the above proof of 
Theorem \ref{thm:alex}(1), one can easily verify that
under the convergence \eqref{eq:conv=NtoKN},
$(N_0,x)$ also converges to $K(\Sigma_x(N_0))$.
\end{rem} 
\pmed\n
{\bf The case of intrinsic metric of $N_0$.}
\,

Concerning Theorem \ref{thm:alex} (2),  if we consider 
the intrinsic metric, denoted by $N_0^{\rm int}$, 
of $N_0$ induced from $N$, then it 
 is not infinitesimally Alexandrov. More precisely, the space of directions
$\Sigma_x(N_0^{\rm int})$ could be  more complicated
because of the complexity of 
the local structure of $N_0^{\rm int}$ near 
the points of $\ca S$.
See Example \ref{ex:notAlex}.

To make it clear,  we
consider the space of
directions,  denoted by
$\Sigma_x(N_0^{\rm int})$,
at any point $x\in N_0$, 
by introducing the upper angle.
For two curves $\sigma_1, \sigma_2\in \ca A_x(N_0)$,  let us consider the upper angle
\[
 \angle_x^{N_0^{\rm int}}(\sigma_1,\sigma_2):= \limsup_{t_1,t_2\to 0} \tilde\angle^{N_0^{\rm int}}\sigma_1(t_1)x\sigma_2(t_2).
 \]
Let $\Sigma_x(N_0^{\rm int})$ denote the completion of the set of all those admissible curves with respect to the 
upper angle. 
%%%%%%

Note that the intrinsic metirc and the extrinsic metric define the same topology on $N_0$.
Note also that 
$\eta_0:C_0\to N_0$ defines
$1$-Lipschitz map $\eta_0^{\rm int}:C_0\to N_0^{\rm int}$,
and that the restriction 
$\eta_0^{\rm int}:C_0\setminus \tilde{\ca S}\to 
N_0^{\rm int}\setminus {\ca S}$ is  an
local isometry (Lemma \ref{lem:eta'}).
% , where 
%$\bar{N_0^2}$ denotes the closure of  $N_0^2$ 
%in $N_0^{\rm int}$. 
This implies the following lemma.

\begin{lem} \label{lem:eta0-int}  For any $p\in\tilde{\ca S}$, $d\eta_0^{\rm int}:\Sigma_p(C_0)\setminus\Sigma_p(\tilde{\ca S})\to
\Sigma_x(N_0^{\rm int})\setminus\Sigma_x({\ca S})
$ is locally isometric.
\end{lem}

Although  
$\Sigma_p(C_0)=\Sigma_p(\tilde{\ca S})$ may happen
sometime as in Figure 1 in Section \ref{sec:intro},
Lemma \ref{lem:eta0-int} is useful to determine $\Sigma_x(N_0^{\rm int})$ in the following 
Example \ref{ex:notAlex}, where 
$\Sigma_x(N_0^{\rm int})$ is not necessarily an 
Alexandrov space with curvature
$\ge 1$ in case $x\in \ca S$.
Moreover if one replace $N$ by $N_0^{\rm int}$ in \eqref{eq:int-angle-def}, then 
 the limit  \eqref{eq:int-angle-def}
for some curves $\sigma_1, \sigma_2\in \ca A_x$
does not exist. 
Therefore, one can not expect that
$(N_0^{\rm int}/\e,x)$ converges to the 
Euclidean cone $K(\Sigma_x(N_0^{\rm int}))$ as $\e\to 0$.
%%%%%%%%%
\psmall 

\begin{ex} \label{ex:notAlex}
For a monotone decreasing sequence $\{ a_n\}$
converging to $0$ as $n\to\infty$,
let us consider a smooth function $g:\R^2_+\to\R_+$ satisfying 
\begin{itemize}
\item $g^{-1}(0)=\{ (0,a_n)\}\cup (\R\times\{ 0\})\,;$
\item $|\nabla\nabla g| \le C$
and  $\nabla g=0$ on $g^{-1}(0)$.
\end{itemize}
Let $N:=\{ (x,y,z)\in\R^2_+\times\R\,|\, |z|\le g(x,y)\}$, and 
$N_0:=\{ (x,y,z)\in\R^2_+\times\R\,|\, |z|=g(x,y)\}$.
%%%%%%
Let $\hat g:\mathbb R^2\to\R_+$ be the symmetric extension of $g$ defined as 
$\hat g(x,y)=g(x,-y)$ for $y\le 0$, and set 
$C_0:=\{ (x,y,z)\in\R^3\,|\, z=\hat g(x,y)\}$.
%%%%%%%
Making use of Example \ref{ex:basic-disk}  as in Example \ref{ex:general-cusp},
one can construct a sequence $M_i$ contained in 
$\ca M_b(3,0,\lambda,\infty)$ for some 
$\lambda$ such that 
$(M_i,\pa M_i)$ (resp. the intrinsic metric $(\pa M_i)^{\rm int}$)
converges to 
$(N,N_0)$ (resp. to $C_0$). 
Set ${\bf x}:=(0,0,0)\in N_0$.
Note that 
$N_0^2=\{ (0,a_n,0)\}$, $\ca S^1=\{\bf x\}$ and 
$\ca C=\R\times \{ (0,0)\}\setminus \{ {\bf x}\}$.
Note that $\Sigma_p(C_0)=\mathbb S^1$ and 
$\Sigma_x(N)=\Sigma_x(N_0)=\mathbb S^1/f_*=[0,\pi]$, where $p:=\eta_0^{-1}({\bf x})$, and the isometric involution $f_*$ on $\Sigma_p(C_0)$
is given by the reflection $f_*(x,y)=(x,-y)$.

Consider the curves on $C_0$ starting 
from ${\bf x}$ defined as 
\[
   \tilde\sigma_+(t)=(0, t, g(0,t)), \,\quad
 \tilde\sigma_-(t)=(0,-t, g(0,-t))
\]
defined on $[0,\infty)$, and set
$\sigma_\pm(t):=\eta_0\circ\tilde\sigma_\pm(t)$.
Now the projection 
$\eta_0^{\rm int}:C_0\to N_0^{\rm int}$ is defined by 
the identification $(0,-a_n,0)=(0,a_n,0)$ for all $n$.
In what follows, we verify that the structure of the space of directions 
$\Sigma_{\bf x}(N_0^{\rm int})$ depends on 
the sequence $\{ a_n\}$.

For instance, suppose $a_n=1/n!$.
Then as $n\to\infty$,
$\{ a_n\}$ becomes more and more discrete in the sense that   
$1-a_{n+1}/a_n =n/(n+1)\to 1$ as $n\to\infty$.
Let $b_n:=(a_n+a_{n+1})/2$, and set 
${\bf y}_n^\pm:=\sigma_{\pm}(b_n)$.
Obviously we have 
$\lim_{n\to\infty}|{\bf y}_n^\pm,{\bf x}|_{N_0^{\rm int}}/b_n
=1$. Note that 
\begin{align*}
&1\ge\lim_{n\to\infty}|{\bf y}_n^+, {\bf y}_n^-|_{N_0^{\rm int}}/2b_n  \\
&=\lim_{n\to\infty} (|{\bf y}_n^+,\sigma_+(a_{n+1})|+
          |\sigma_-(a_{n+1}),{\bf y}_n^-|)/2b_n \\
	&=\lim_{n\to\infty} 2(b_n-a_{n+1})/2b_n=1,
\end{align*}  
which implies  
\[
 \lim_{n\to\infty} \tilde\angle^{N_0^{\rm int}}{\bf y}_n^+{\bf x} {\bf y}_n^-=\pi,
\]
and hence  
$\angle^{N_0^{\rm int}}(\dot\sigma_+(0),\dot\sigma_-(0))=\pi$. Therefore by Lemma \ref{lem:eta0-int},
$\Sigma_{\bf x}(N_0^{\rm int})$ is 
isometric to $\mathbb S^1$.

Next, suppose 
$a_n=1/n$. Then as $n\to\infty$,
$\{ a_n\}$ becomes more and more dense in the sense that for any $m$, under the $1/m$-rescaling,
$\{ a_n\}_{n>m}$ is $1/m$-dense in  
$[0,1]$. Under this observation, a simple calculation yields that 
$$
\angle^{N_0^{\rm int}}(\dot\sigma_+(0),\dot\sigma_-(0))=0,
$$
and therefore by Lemma \ref{lem:eta0-int},
$\Sigma_{\bf x}(N_0^{\rm int})$ is the one point union of two copies of $S^1_{\pi}$.
\end{ex}

\pmed
\setcounter{equation}{0}

%%%%%%%%%%%%%%%%%%%%%%%%%
\section{Differential of gluing maps} \label{sec:Isometric}

The main purpose of this section is to 
make clear the relations
between $\Sigma_p(C_0)$ and 
$\Sigma_x(X_0)$ and between
$\Sigma_p(C)$ and 
$\Sigma_x(Y)$ (Theorem \ref{thm:X1-f}).
From here on, we fix
$p\in C_0$ and $x\in X_0$ with $\eta_0(p)=x$.

\subsection{Isometric involutions}\label{ssec:iso-inv}
\psmall\n
In this subsection, for any $x\in X_0^1$, 
we shall construct an isometric involution
$f_*$ on $\Sigma_p(C_0)$ for the proof of 
Theorem \ref{thm:X1-f}.

To define $f_*$, 
we investigate  the $1$-Lipschitz map $d\eta_0:\Sigma_p(C_0)\to \Sigma_x(X_0)$.
Recall that $\xi_x^+$ and $\tilde\xi_p^+$ are
the perpendicular directions
at $x$ and $p$ respectively.  

We begin with 

\begin{lem}\label{lem:sharp-deta} We have 
$1\le \# d\eta_0^{-1}(v) \le 2$ for all $v\in\Sigma_x(X_0)$.
\end{lem}
\begin{proof} 
Suppose there are distinct three elements 
$\tilde v_j$\, $(j=1,2,3)$\,
in $d\eta_0^{-1}(v)$, and consider the geodesics 
$\tilde\gamma_j:[0,\pi/2]\to\Sigma_p(C)$ from $\tilde v_j$ to $\tilde\xi_p^{+}$.
Then $\gamma_j=d\eta(\tilde\gamma_j)$ are distinct three minimal geodesics
from $v$ to $\xi_x^+$ in $\Sigma_x(Y)$
that are perpendicular to $\Sigma_x(X_0)$
at $v$. 
Since 
arbitrary two of $\gamma_1,\gamma_2,\gamma_3$
form a geodesic near $v$, this contradicts the non-branching 
property of geodesics in Alexandrov spaces.
\end{proof}

\begin{defn} \label{defn:Sigma^12}
 For $k=1,2$, let $\Sigma_x(X_0)^k$ denote the set of all directions $v\in\Sigma_x(X_0)$
such that 
\[
        \# d\eta_0^{-1}(v) = k,
\]       
or equivalently, the number of minimal geodesics in $\Sigma_x(Y)$ joining $\xi_x^+$ 
and $v$ is equal to $k$.
\end{defn}

\pmed
Note the following 

\begin{lem}[Rigidity lemma(cf.\cite{Shm})]\label{lem:rigidity}
For every $x\in X_0^1$, let us assume ${\rm rad} (\xi_x^+)=\pi/2$.
Then for any $v, w\in\Sigma_x(X_0)$ with $|v,w|<\pi$ and minimal geodesics
$\xi_x^+ v$,  $vw$ in $\Sigma_x(Y)$, there is a minimal geodesic 
$\xi_x^+ w$ such that the geodesic triangle $\triangle \xi_x^+ v w$ spans a 
totally geodesic surface of constant curvature $1$.
\end{lem}

Let $f:C_0\to C_0$ be the involution defined in Section \ref{sec:gluing}.
We define the involution $f_*:\Sigma_p(C_0)\to\Sigma_p(C_0)$ 
in a way similar to \eqref{eq:defn-f} as follows:
\begin{align} \label{eq:defn-df}
 \text{$f_*(\tilde v):= \tilde w$ \, if\, $\{\tilde v, \tilde w\}=d\eta_0^{-1}(d\eta_0(\tilde v))$}, 
\end{align}

Let $\Sigma_x(X_0)^{\rm int}$ denote the intrinsic metric 
on $\Sigma_x(X_0)$ induced from $\Sigma_x(X)$.
 Note that 
\begin{itemize}
\item $\Sigma_x(X_0)^{\rm int}=\Sigma_x(X_0)$ if ${\rm rad}(\xi_x^+)=\pi/2$ or $x\in X_0^2\,;$ 
\item $\Sigma_x(X_0)^{\rm int}$ is different from $\Sigma_x(X_0^{\rm int})$ 
discussed at the end of 
Section \ref{sec:int-ext}.
\end{itemize}

The following result plays an  important role
in the present paper.

\begin{thm}\label{thm:X1-f}
For each $x\in X_0^1$,
take $p\in C_0$ with $\eta_0(p)=x$. Then 
$f_*:\Sigma_p(C_0)\to\Sigma_p(C_0)$ is an isometry satisfying  the following:
\benu
 \item $\Sigma_x(X_0)^{\rm int}$ is isometric to the quotient space $\Sigma_p(C_0)/f_*\,:$
 \item If ${\rm rad}(\xi_x^+)=\pi/2$, then  $\Sigma_x(Y)$ and $\Sigma_x(X)$ are  isometric to the quotient geodesic  spaces 
 $\Sigma_p(C)/f_*$ and 
 $\Sigma_p(C_0)/f_*$ respectively$\,;$
\item If ${\rm rad}(\xi_x^+)>\pi/2$, then $f_*$ is the identity and 
$\Sigma_x(Y)$ is isometric to the gluing $\Sigma_p(C)\cup_{d\eta_0} \Sigma_x(X)$, where the identification is made by the isometry $d\eta_0:\Sigma_p(C_0) \to \Sigma_x(X_0)^{\rm int}.$ 
\eenu
\end{thm}

 Let $\ca F_x:=d\eta_0(\tilde{\ca F}_p)$, where 
$\tilde{\ca F}_p\subset \Sigma_p(C_0)$ denotes the
fixed point set of $f_*$. By \cite{PtPt:extremal},
$\ca F_x$ is an extremal subset of $\Sigma_x(X_0)$.

%%%%%%%%%%%%%%%%%%%%%%%
\pmed

For points $x\in X_0^2$, we already know the following result,
which is immediate from Lemma \ref{lem:double-sus} and Proposition \ref{prop:isometry'}.

\begin{prop} \label{prop:inf-str-X02}
For any $x\in X_0^2$, let $\{ p_1, p_2\}=\eta_0^{-1}(x)$. Then we have 
\benu
 \item $d\eta_0: \Sigma_{p_i}(C_0) \to \Sigma_x(X_0)$ is an isometry for each $i=1,2\,;$
 \item $\Sigma_x(Y)$ is isometric to the gluing                                                                 
  $\Sigma_{p_1}(C)\cup_{f_*} \Sigma_{p_2}(C)$, where the isometry 
$f_*:\Sigma_{p_1}(C_0)\to \Sigma_{p_2}(C_0)$ is given by 
$$
f_*=((d\eta_0)|_{\Sigma_{p_2}(C_0)})^{-1}\circ (d\eta_0)|_{\Sigma_{p_1}(C_0)}.
$$
 \eenu
\end{prop} 

As an immediate consequence of Theorem \ref{thm:X1-f}
 together with Lemma \ref{lem:Sigma-Alex}, we have the following.

\begin{cor} \label{cor:Sigma0=Alex}
For every $x\in X_0$, $\Sigma_x(X_0)^{\rm int}$ is an Alexandrov space with curvature $\ge 1$.
\end{cor}

\begin{proof}[Proof of Theorem \ref{thm:alex}(2)] 
This is now immediate from Lemma \ref{lem:Sigma-Alex}, Remark \ref{rem:Tx(X0)}
and Corollary \ref{cor:Sigma0=Alex}.
\end{proof}

\begin{ex} \label{ex:summary} 
 In Example \ref{ex:cusp}, for any cusp $x$ of the boundary $X_0$,
$\Sigma_x(X)=\Sigma_x(X_0)$ is a point. For $p\in C_0$ with $\eta_0(p)=x$,
we have $\Sigma_p(C_0)=\mathbb S^0$, which consists of two points with distance $\pi$, 
 and $f_*:\Sigma_p(C_0)\to \Sigma_p(C_0)$ is the transposition.
In this case, $f_*$ is nontrivial although $x\in {\rm int}\,X_0^1$. 
\end{ex}

\pmed 
For the proof of Theorem \ref{thm:X1-f}, we need 
the following.

\begin{lem} \label{lem:limit-one}
For any $x\in X_0^1$, take $p\in C_0$ with $\eta_0(p)=x$.
For $\tilde v, \tilde v_i \in \Sigma_p(C_0)$ with $\tilde v_i\to \tilde v$,
let $v=d\eta_0(\tilde v)$ and $v_i=d\eta_0(\tilde v_i)$.
Then we have 
\[
      \lim_{i\to\infty} \frac{|\tilde v, \tilde v_i|}{|v, v_i|}=1.
\]
Namely  $d\eta_p:T_p(C_0)\to T_x(X_0)$ is length-preserving.
\end{lem}
\pmed
From here on, we implicitly use the metric 
$|\tilde v, \tilde v_i|=|\tilde v, \tilde v_i|_{\Sigma_p(C_0)}$,
$|v, v_i|=|v, v_i|_{\Sigma_x(Y)}$ because of simplicity.

\pmed\n
\begin{proof}[Proof of Lemma \ref{lem:limit-one}]
Let $\tilde \gamma:=\tilde\xi_p^+\tilde v$, $\tilde \gamma_i:=\tilde\xi_p^+\tilde v_i$,
and set $\gamma:=d\eta(\tilde\gamma)$, $\gamma_i:=d\eta(\tilde\gamma_i)$.
Let $\theta_i:=|\tilde v, \tilde v_i|$.
In view of Lemma \ref{lem:loc-isom}, we have 
\[
      \theta_i=\angle \tilde v \tilde \xi_p^+ \tilde v_i= \angle_{\xi_x^+}(\gamma,\gamma_i).
\]
Let $\sigma_i$ be a minimal geodesic from $v$ to $v_i$ in $\Sigma_x(Y)$.
Since $|\xi_x^+, v|=|\xi_x^+,v_i|=\pi/2$, the convexity of 
$\Sigma_x(Y)\setminus\mathring{B}(\xi^+,\pi/2)$ shows 
$\sigma_i\subset\Sigma_x(X)$, and hence 
$\angle(\uparrow_v^{\xi_x^+}, \dot\sigma_i(0))\ge  \pi/2$.
The first variation formula then
implies 
\begin{align} \label{eq:limit-under-angle}
   \lim_{i\to\infty}\angle(\Uparrow_v^{\xi_x^+}, \dot\sigma_i(0))= \pi/2.
\end{align}
%Here it should be remarked that if ${\rm rad}(\xi_x^+)=\pi/2$,

%Therefore, from now on we are concerned with the case ${\rm rad} (\xi_x^+)>\pi/2$.

Let $\zeta_i:=|v, v_i|\le \theta_i$, 
For each $\e>0$, take $u_i(\e)\in\gamma_i$ and $\tilde u_i(\e)\in\tilde \gamma_i$ 
with $|v_i, u_i(\e)|=|\tilde v_i, \tilde u_i(\e)|=\e\zeta_i$.
Now consider the convergence
\[
      \left( \frac{1}{\zeta_i} \Sigma_x(Y), v \right) \to (T_{v}( \Sigma_x(Y)), o_v),
\]
Let $\rho_i$ be a minimal geodesic from $v$ to $u_i(\e)$
in $\Sigma_x(Y)$, and 
let $\rho_\infty$ be the limit of $\rho_i$ under the above convergence.

We show
\beq \label{eq:alpha>0}
    \alpha:= \angle(\dot\rho_\infty(0), T_v(\Sigma_x(X))) >0.
\eeq
%%%%%%%%
If $\alpha=0$, then $\rho_\infty$ must be contained in the tangent cone $T_v(\Sigma_x(X))$.
This is a contradiction 
since 
$u_\infty(\e)=\lim_{i\to\infty}u_i(\e)\notin T_v(\Sigma_x(X))$. Thus we have
$\alpha>0$.

From the lower semicontinuity of angles, we have for  all large enough $i\ge i_\e$,
\[
\angle(\dot\rho_i(0), \Sigma_x(X_0)) >\alpha/2.
\]
Therefore we can define the lift  $\tilde\rho_i:=(d\eta)^{-1}(\rho_i)$
of $\rho_i$,
and conclude
\[
   |\tilde v, \tilde u_i(\e)|=|v, u_i(\e)|,
\]
%and hence $|\tilde v, \tilde u_i(\e)|=|v, u_i(\e)|$ 
for all $i\ge i_\e$. 
Letting $\e\to 0$ and $i\to\infty$ properly,  we obtain the conclusion of the lemma.
This completes the proof.
\end{proof}
\psmall
Recall that the gluing map $\eta:C\to Y$, which is $1$-Lipschitz,  induces the bijective local isometry
$\eta:C\setminus C_0\to Y\setminus X$.
For $p\in C_0$ with $\eta(p)=x$, we have the derivative
$d\eta_p:T_p(C)\to T_x(Y)$, which induces the  injective local isometry
$d\eta_p:T_p(C)\setminus T_p(C_0)\to T_x(Y)\setminus T_x(X)$ (Proposition \ref{prop:length'}).
By Lemma  \ref{lem:limit-one},  $d\eta_p:T_p(C_0)\to T_x(X_0)$ is a surjective length-preserving map.

\begin{lem} \label{lem:local-surj-varphi}
Fix any $\tilde v\in\Sigma_p(C_0)$ set $v=d\eta_0(\tilde v)$.
For any $\e>0$ and $w\in B(v,\e;\Sigma_x(X_0))$, there is an element 
$\tilde w\in\Sigma_p(C_0)$ satisfying 
$d\eta_0(\tilde w)=w$ and 
$|\tilde v,\tilde w|/|v,w|\le 1+\tau_v(\e)$.
\end{lem} 

Lemma \ref{lem:local-surj-varphi} implies that $\Sigma_x(X_0)^2$ is open 
in $\Sigma_x(X_0)$.

\begin{proof}[Proof of Lemma \ref{lem:local-surj-varphi}]
 Suppose ${\rm rad}(\xi_x^+)=\pi/2$.
Then by Rigidity Lemma  \ref{lem:rigidity},
we get the conclusion.
\par 
Next suppose ${\rm rad}(\xi_x^+)>\pi/2$, and take $u\in\Sigma_x(X)$ such that 
$|\xi_x^+,u|>\pi/2$.
From the curvature condition of $\Sigma_x(Y)$, 
we have $\angle\xi_x^+ v u>\pi/2$ for any $v\in
\Sigma_x(X_0)$. This implies 
$\Sigma_x(X_0)=\Sigma_x(X_0)^1$ 
(see the proof of Lemma \ref{lem:sharp-deta}).
Therefore $d\eta_0:\Sigma_p(C_0)\to \Sigma_x(X_0)$ is a length-preserving bijection,
and hence the conclusion certainly holds.
\end{proof}
 
\begin{proof}[Proof of Theorem \ref{thm:X1-f}]

\pmed\n
(1)\, We show that $f_*:\Sigma_p(C_0)\to \Sigma_p(C_0)$ is an isometry.
\psmall
\begin{ass} \label{ass:isom}
For any fixed $\tilde v\in \Sigma_p(C_0)$ and $\tilde v_i\in \Sigma_p(C_0)$
with $|\tilde v, \tilde v_i|\to 0$ as $i\to\infty$, we have
$$
\lim_{i\to\infty} \frac{|f_*(\tilde v), f_*(\tilde v_i)|}{|\tilde v, \tilde v_i|}=1.
$$
\end{ass}
\begin{proof}
We set  
\balii
   v:= d\eta_0(\tilde v)=d\eta_0(f_*(\tilde v)), \,\,\,
  v_i:= d\eta_0(\tilde v_i)=d\eta_0(f_*(\tilde v_i)).
\end{align*}  
First consider the case when $v\in\Sigma_x(X_0)^1$.
We may assume  $v_i\in\Sigma_x(X_0)^2$.
Since $f_*(\tilde v_i)\to \tilde v$, Lemma \ref{lem:limit-one}
yields the conclusion.

Next suppose $v\in\Sigma_x(X_0)^2$.
Applying Lemma \ref{lem:local-surj-varphi} to $f_*(\tilde v)$,
$v$ and $v_i$, we obtain  $f_*(\tilde v_i)\to f_*(\tilde v)$.
Then again Lemma \ref{lem:limit-one}
yields the conclusion.
\end{proof}

\pmed
It is now immediate to show that $f_*$ is an isometry.
For arbitrary $\tilde v, \tilde w$ in the same component of $\Sigma_p(C_0)$,
take a minimal geodesic $\tilde \gamma$ joining them. For any $\e>0$, 
applying Assertion \ref{ass:isom} to each point of $\tilde\gamma$, we have a 
finite sequence of points of $\tilde\gamma$,
$\tilde v=\tilde v_0<\tilde v_1<\cdots <\tilde v_N=\tilde w$ such that 
$|f_*(\tilde v_{i-1}), f_*(v_i)| < (1+\e)|\tilde v_{i-1}, \tilde v_i|$ for each $1\le i\le N$.
Summing up these and letting $\e\to 0$, we have $|\tilde v, \tilde w|\ge |f_*(\tilde v), f_*(\tilde w)|$.
Repeating this to $f_*(\tilde v)$, $f_*(\tilde w)$, we also have 
 $|f_*(\tilde v), f_*(\tilde w)| \ge |\tilde v, \tilde w|$, and conclude
$|f_*(\tilde v), f_*(\tilde w)| = |\tilde v, \tilde w|$.
\psmall
%%%%%%
Define the $1$-Lipschitz bijective map $[\eta_0]_*:\Sigma_p(C_0)/f_*\to \Sigma_x(X_0)$
by 
$$
 [\eta_0]_*([\tilde v])=(d\eta_0)(\tilde v), 
$$
where $[\tilde v]$ is the element of $\Sigma_p(C_0)/f_*$ represented by 
$\tilde v\in \Sigma_p(C_0)$.
%%%%%%
Lemma \ref{lem:limit-one}
shows that $[\eta_0]_*$ is a
length-preserving bijection, and therefore induces an isometry  $[\eta_0]_*:\Sigma_p(C_0)/f_*\to \Sigma_x(X_0)^{\rm int}$.
 
(2)\, Suppose ${\rm rad}(\xi_x^+)=
\pi/2$.
In this case, we have $\Sigma_x(X_0)=\Sigma_x(X)$, and hence
 $\Sigma_x(X_0)^{\rm int}=\Sigma_x(X_0)$.
It follows from  Proposition \ref{prop:length'}, the map
\[
     [d\eta]_*:\Sigma_p(C)/f_* \to \Sigma_x(Y)
\]
sending $[\xi]$ to $d\eta(\xi)$ is a length-preserving bijection, and hence is an isometry.

\psmall
(3)\,Suppose  
${\rm rad}(\xi_x^+)>\pi/2$.
In the proof of  Lemma \ref{lem:local-surj-varphi}, we already showed that 
$\Sigma_x(X_0)=\Sigma_x(X_0)^1$, showing that  $f_*$ is the identity.
Since $d\eta:\Sigma_p(C)\to\Sigma_x(Y)\setminus\mathring{\Sigma}_x(X)$
is a length-preserving bijection,
it induces an isometry between
$\Sigma_p(C)$ and the intrinsic metric of 
$\Sigma_x(Y)\setminus\mathring{\Sigma}_x(X)$.
Thus, $\Sigma_x(Y)$ is isometric to the gluing
$\Sigma_p(C)\cup_{d\eta_0} \Sigma_x(X)$
as required.

This completes the proof of Theorem 
\ref{thm:X1-f}.
\end{proof}

\subsection{Boundary points and cusps}\label{ssec:CB}  
\psmall\n
 In this subsection, we define the notion of boundary points of $X_0$ and discuss its 
properties.
We also define the notion of cusps, and 
provide a few examples concerning these
notions.

\pmed\n
{\bf Boundary of $X_0$}\,\,
\psmall\n 
We begin with the following 
basic facts on the boundaries of Alexandrov spaces, which will be used implicitly in the 
argument below. 

\begin{prop} \label{prop:summary-bdy}
\begin{enumerate}
\item Let $Y$ be an Alexadrov space with curvature bounded below. If $G$ is a compact group of isometries acting on $Y$, then 
$\pa(Y/G)\supset (\pa Y)/G \,;$
\item Let $\Sigma$ be an Alexandrov space
with curvature $\ge 1$. Then 
\begin{enumerate}
\item the boundary of the spherical suspension
 $\{ \xi_\pm\}*\Sigma$ coincides with
 $\{ \xi_\pm\}*\pa\Sigma\,;$
\item the boundary of the spherical  half suspension
$\{ \xi_+\}*\Sigma$ coincides with
 $\Sigma\cup (\{ \xi_+\}*\pa\Sigma)$.
\end{enumerate}
\item Let $Y_i$ be a sequence of $m$-dimensional Alexandrov spaces with
curvature $\ge \kappa$ converging to an
$m$-dimensional Alexandrov space $Y$.
Then there is a homeomorphism 
$(Y_i,\pa Y_i)\to (Y,\pa Y)$ that is 
also an $o_i$-approximation for any large enough $i$.
\end{enumerate}
\end{prop}
\begin{proof} (1) follows since each element of $G$ preserves $\pa Y$. (2) is elementary.
(3) is due to \cite{Pr:alexII}.
\end{proof}

\begin{defn}\label{defn:int-pa-X0}
We say that a point $x\in X_0$ is a  {\it  boundary point} of $X_0$
(resp. an  {\it  interior point}  of $X_0$)
if the Alexandrov space 
$\Sigma_x(X_0)^{\rm int}$ has nonempty boundary (resp. no boundary).
We denote by $\pa X_0$ and ${\rm int}\, X_0$
the set of all  boundary points of  $X_0$
and  the set of all interior points of  $X_0$
respectively.

We also set 
\[
    \pa_* X_0 := \eta_0(\pa C_0), \qquad {\rm int}_* X_0:=X_0\setminus \pa_* X_0.
\]
\end{defn}

Obviously, we have $\pa_* X_0\subset \pa X_0$.
In some cases, $\pa X_0$ is not empty even if 
$\pa_* X_0$ is empty
(see Example \ref{ex:summary}).
\begin{rem}
This is a confirmation of 
terminology concerning the interiors defined so far:
\begin{itemize}
\item \text{${\rm int} X_0^k$\, $(k=1,2)$ is the topological interior of $X_0^k$ in $X_0\,;$}
\item \text{${\rm int} X_0$ is the interior of $X_0$ defined 
 in Definition \ref{defn:int-pa-X0}}. 
\end{itemize}
\end{rem}

%\psmall\n

\begin{lem} \label{lem:criteria-0}
For any  $x\in X_0$, we have $x\in\pa Y$ if and only if one of the following holds$:$ 
\begin{enumerate}
\item $x\in \pa_* X_0\,;$
\item $x\in X_0^1$, ${\rm rad}(\xi_x^+)=\pi/2$ and $f_*$ is the identity.
\end{enumerate}
\end{lem}
\begin{proof} Suppose $(1)$ or $(2)$. Then 
Theorem \ref{thm:X1-f} immediately implies  $x\in\pa Y$. 
Next suppose $x\in \pa Y$.
Take any $p\in\eta_0^{-1}(x)$.
If $p\in C_0^2$, then  Lemma  \ref{lem:double-sus}
shows $x\in \pa_* X_0$. Let us assume 
$p\in C_0^1$. If ${\rm rad}(\xi_x^+)>\pi/2$, then Theorem \ref{thm:X1-f} implies
$p\in\pa C_0$ and hence $x\in \pa_* X_0$.
Suppose $f_*$ is not the identity.
If $p\in{\rm int}\,C_0$, then in the expression  
$\Sigma_x(Y)=(\{ \tilde\xi_p^+\}*\Sigma_p(C_0))/f_*$ due to
Theorem \ref{thm:X1-f}, $\Sigma_x(Y)$ would have
no boundary since $\dim \tilde{\ca F}_p\le\dim\Sigma_p(C_0)-1\le\dim\Sigma_x(Y)-2$. Thus we have 
$p\in\pa C_0$ and hence $x\in \pa_* X_0$.
This completes the proof.
\end{proof}

\begin{lem} \label{lem:criteria-int-pa}
For any $x\in X_0$,  we have the following.
\benu
 \item  For $x\in X_0^1$ with  
$p=\eta_0^{-1}(x)$,  we have $x\in\pa X_0$ if and only if 
$p\in \pa C_0$ or 
   $\dim \tilde{\ca F}_p = \dim \Sigma_p(C_0) -1\,;$
 \item For $x\in X_0^2$ with  $\{ p_1,p_2\}=\eta_0^{-1}(x)$,  we have $x\in\pa X_0$ if and only if
    $p_i\in \pa C_0$ \,$(i=1,2).$ 
\eenu
\end{lem}
\begin{proof} (1)\, Note that if $\dim\tilde{\ca F}_p=\dim \Sigma_p(C_0) -1$, then 
$f_*$ is a reflection.
Therefore we have 
$d\eta_0(\tilde{\ca F}_p)\subset \pa \Sigma_x(X_0)$.
Thus if $p\in\pa C_0$ or $\dim\tilde{\ca F}_p=\dim \Sigma_p(C_0) -1$, then 
$\eta_0(p)\in\pa X_0$.
Conversely, assume $x\in\pa X_0$. 
If $f_*$ is trivial, then clearly we have $p\in\pa C_0$.
Suppose $p\in {\rm int}\, C_0$ and $\dim \tilde{\ca F}_p \le \dim \Sigma_p(C_0) -2$.
Then we can take $v\in \pa\Sigma_x(X_0)\setminus \ca F_x$.
Choose $\tilde v\in\Sigma_p(C_0)$ with $d\eta_0(\tilde v)=v$.
Then $d\eta_0:\Sigma_p(C_0)\to\Sigma_x(X_0)$ isometrically maps a small neighborhood of $\tilde v$ in $\Sigma_p(C_0)$ to a neighborhood of $v$ in $\Sigma_x(X_0)$.
This is impossible.

(2)\,  If $x\in X_0^2$, then 
Lemma \ref{lem:double-sus} and Proposition \ref{prop:isometry'}
imply that $\Sigma_x(X_0)$ is isometric to $\Sigma_p(C_0)$. (2) is now immediate.
\end{proof} 
\pmed
From Lemma \ref{lem:criteria-int-pa} together with
Theorem \ref{thm:X1-f}, we have the following.

\begin{cor} \label{cor:big-bdy}
For $x\in X_0$, if ${\rm rad}(\xi^+_x)>\pi/2$, then we have 
\[
     x\in\pa Y \Longleftrightarrow x\in\pa_* X_0\Longleftrightarrow x\in\pa X_0.
\]
\end{cor}  
%%%%%%%%%%%%% Cusps   
\pmed\n   
{\bf Cusps}\,\,

\begin{defn} \label{defn:cusp}
A point $x\in {\rm int}\,X_0^1$ is called a {\it cusp} of $X$ if
$f_*:\Sigma_p(C_0)\to\Sigma_p(C_0)$ is nontrivial, where $\eta_0(p)=x$.
The set of all cusps of $X_0$ is denoted by ${\bf {\ca C}}$.
\end{defn}  

In Example \ref{ex:cusp}, we have already encountered with a typical
example of cusps (see Example \ref{ex:summary}).

\pmed
\n

%%%%%%%%%%%%%%%%%%%%%%%%%%%%
 In the following example, we construct some nonnegatively curved two-disk
with corner, which plays an important role to construct 
several Riemannian manifolds with boundary in 
$\ca M_b(n, \kappa, \lambda, d)$.
 
\begin{ex} \label{ex:basic-disk}
For $\e, \delta>0$, let 
$I_\delta=\{ (x, 0, 0)\in\mathbb R^3\,|\, 0\le x\le \delta\}$, and 
set $\e'=2\e/\pi$.
Let $D(\e, \delta)$ be  the intersection of the boundary of
the $\e'$-neighborhood of $I_\delta$ with the two half spaces $z\le 0$ and $x\le \delta$.  
Note that the intersection 
$I(\e,\delta):=D(\e,\delta)\cap \{ x=\delta \}$ is an arc of  length $2\e$. 
Note that $D(\e,\delta)$
is already used in Example \ref{ex:cusp}.
\end{ex}

We generalize Example \ref{ex:cusp} to the 
general dimension in the following example.

\begin{ex}\label{ex:general-cusp}
Let $D$ be a nonnegatively curved 
$n$-disk such that $\pa D$ has a neighborhood isometric to a product $\pa D\times [0,\delta)$ for some $\delta>0$.
Let $g:D\to \R_+$ be a smooth function such that 
\begin{enumerate}
\item $g^{-1}(0)=\pa D\,;$
\item $\nabla g =0$ on $\pa D$.
\end{enumerate}
For any $\e>0$, we set 
\begin{align*}
&\hat D:=\{ (u, t)\in D\times\R\,|\, |t|\le g(u)\}, \\
&L_\e:=\{ (u, t)\in D\times\R\,|\, |t|\le g(u)+\e\}.
\end{align*}
$L_\e$ is 
 an $(n+1)$-dimensional Riemannian manifold 
with corner  around $\pa D$,
where  $\pa L_\e$ is the union of 
$A_\e:=\pa D\times [-\e, \e]$
and $\{ (u,t)\in D\times\R\,| |t|=g(u)+\e\}$. 
To resolve the corner singularities of $L_\e$, 
let $D(\e, 10\e)$ be the surface
 in the $xyz$-space constructed in Example \ref{ex:basic-disk} such that  
$I(\e,10\e)$ is isometric to $[-\e,\e]$, 
and consider 
\[
P_\e=\pa D\times D(\e,10\e),\quad
Q_\e:=\pa D\times I(\e,10\e).
\]
Let $M_\e$ be the gluing of  $L_\e$ and $P_\e$ along
$A_\e$ and $Q_\e$.
Note that 
$M_\e\in \ca M_b(n+1, 0, \lambda,d)$ for some $\lambda, d$,
and $M_\e$ converges to $N:=\hat D$
as $\e\to 0$.

Let $\Gamma:=\{ (u,t)\,|\, |t|=g(u), u\in D\}$.
Let $C_0$ (resp.  $N_0$) be the limit of 
$\pa M_\e$ with respect to the the intrinsic metric 
(resp. the extrinsic metric) as $\e\to 0$ as usual.
Then $C_0=\Gamma$ (resp. $N_0=\Gamma$) equipped with the intrinsic metric (resp. the extrinsic metric).
We immediately have
$N_0=N_0^1$.
Consider any $p\in C_0\cap (\pa D\times \{ 0\})$,  and set $x:=\eta_0(p)\in N_0$.
Observe that $\Sigma_p(C_0)$ and 
$\Sigma_x(N_0)$ are isometric to the unit sphere $\mathbb S^{n-1}$ and the unit hemisphere $\mathbb S^{n-1}_+$ respectively.
Therefore $\Sigma_x(Y)=\{\xi_x^+\}* \Sigma_p(C_0)/f_*$,
where $f_*$ is the reflection of $\Sigma_p(C_0)$. 

As a summary, we have
\begin{itemize}
\item $\ca C=\pa N_0=\pa D\times \{ 0\}$, while $\pa_* N_0$ is empty$\,;$ 
\item $x\in {\rm int}\, Y\cap \pa N_0$ for any 
$x\in \pa D\times \{ 0\}$.
\end{itemize}
\end{ex}

\begin{rem} \label{rem:CtoS1} \upshape
In Example \ref{ex:general-cusp}, let us change 
the construction only the function $g$ such that instead of the 
above condition (1) we have 
\begin{enumerate}
\item[(1')] $g^{-1}(0)=\pa D\cup Q$,
\end{enumerate}
where 
$Q=\{ x\in D| d(x, \pa D)=1/k, k\in\mathbb N,
k\ge k_0\}$ for a large enough $k_0\in\mathbb N$. This is possible by retaking $D$ with 
a lot of symmetry if necessary.
In this case, we have 
\begin{itemize}
\item $\ca S^1=\pa N_0=\pa D\times \{ 0\}$, while $\pa_* N_0$ is empty. 
\end{itemize}
\end{rem}

A boundary point of $X_0$ is  defined 
infinitesimally. Sometimes
it has a  feature different from the usual notion of 
boundary. For instance, $\pa X_0$ can be a single
point even if $X_0$ is of general dimension
(see Example \ref{ex:non-closed}(2)).

As the following example shows, 
the case when
$x\in{\rm int}\,X_0^1\setminus\ca C$ and  
${\rm rad}(\xi_x^+)=\pi/2$
occurs even in the case of non-inradius collapse.
 Namely the converse to Lemma \ref{lem:single-interior}
does not hold.

\begin{ex}\label{ex:X01cosh}
(1)\,
For $\e>0$, let $\Gamma_\e:=\mathbb Z\times \e \mathbb Z
\subset\mathbb R^2$, which acts on $\R^2$ 
by translation. 
Choose a $\Z$-invariant smooth
function $g:\R\to \R_+$ such that 
\[
g^{-1}(0)= \Z,\quad |\nabla\nabla g|\le C.
\]
Let
\[
 W_\e:=\{ (s,t,u)\in \R^2\times\R_+| 0\le u\le g(s)+\e\},
\quad L_\e:=W_\e/\Gamma_\e.
\]
Let  $D^2(\e)$ denote a positively curved disk
with diameter $\le 2\e$ and 
with totally geodesic boundary circle of length
$\e$. We further assume that
$D^2(\e)$ has a product collar neighborhood 
near $\pa D^2(\e)$.
Since a boundary component of 
$L_\e$ is isometric to $S^1_1\times S^1_\e$,
we can make the following gluing along boundaries:
\[
     M_\e:=(S^1_1\times D^2(\e))\cup_{S^1_1\times S^1_\e} L_\e.
\]
Obviously,  $M_\e$
belongs to $\ca M_b(3,0,\lambda,d)$ for some
$\lambda,d$ independent of $\e$, and converges to 
\[
    N:=\{ (s,t)\in S^1_1\times \mathbb R_+\,|\, 0\le t\le \hat g(s)\}
\]
as $\e\to 0$,
where $\hat g$ is the function on $S^1_1$ induced by 
$g$. Note that $N_0=C_0=\{ (s,t)\in S^1_1\times \mathbb R_+\,|\, t=\hat g(s)\}$, and $N_0=N_0^1$.
Let $x_0$ be the unique point of $S^1_1$ with 
$\hat g(x_0)=0$, and set $x:=(x_0,0)\in N_0$.
Remark that ${\rm rad}(\xi_x^+)=\pi/2$
and $f_*$ is the identity on $\Sigma_p(C_0)$.
Note that $S^1_1\times \{ 0\}\subset \pa Y$.

(2)\, Let $S_\e$ be a nonnegatvely curved two-sphere
converging to an interval $I$ as $\e\to 0$, and consider the 
product  $P_\e:=M_\e\times S_\e$,  where $M_\e$ is
as in (1). Then $P_\e$ converges to $X:=N\times I$ with 
$X_0=N_0\times I$, where $N,N_0$ are as in (1).
For the point $x=(x_0,0)\in N_0$ in (1) and for an end point 
 $y$ of $I$, set ${\bf x}=(x,y)\in X_0$. Then ${\bf x} \in
{\rm int}\, X_0^1$ and  
$\Sigma_{\bf x}(Y)=\xi_{\bf x}^+ * \Sigma_{\bf x}(X_0)$. In particular, we have 
\[
     \pa\Sigma_{\bf x}(Y)=
         \Sigma_{\bf x}(X_0)\cup (\xi_{\bf x}^+ * \pa\Sigma_{\bf x}(X_0)).
\]
\end{ex}

Here we summarize notations defined in this section.

\begin{table}[h]
\begin{center}
\begin{tikzpicture}[auto]

\fill (-1, 1.6) circle (0pt) node [right] 
{$\Sigma_x(X_0)^k=\{ v\in\Sigma_x(X_0)\,|\,
\# d\eta_0^{-1}(v)=k\}$};

\fill (-1, 0.8) circle (0pt) node [right] {$\pa X_0=
\{ x\in X_0\,|\,\pa\Sigma_x(X_0)\neq \emptyset\},\quad 
 {\rm int} X_0=X_0\setminus \pa X_0$}; 

 \fill (-1, 0) circle (0pt) node [right] {$\pa_*X_0=\eta_0(\pa C_0)\subset\pa X_0,\quad 
 {\rm int}_*X_0=X_0\setminus \pa_*X_0\supset {\rm int}X_0$};
\fill (-1, -0.8) circle (0pt) node [right]{$\ca C: \text{the set of cusps}$};
\fill (-1, -1.6) circle (0pt) node [right]{$\tilde{\ca F}_p={\rm Fix}(f_*), \quad
  {\ca F}_x=d\eta_0({\rm Fix}(f_*))$};
\draw[thick] (-1.2,2.1) rectangle (9.2,-2.2);
\end{tikzpicture}
\caption{\small{Boundary, cusps and fixed point sets}}
\end{center}
\end{table}

\pmed
\setcounter{equation}{0}

%%%%%%%%%%%%%%%%%%%%%%%%%%
%\section{Geometry of boundary singular points} \label{sec:local-str-S1}
%\pmed
%%%%%%%%%%%%%%%%%%%%%%
\section{Infinitesimal structure at $\ca S^1\cup\ca C$}\label{sec:inf-SC}

In this subsection, we prove Theorem \ref{thm:S1extremal} by establishing
a splitting theorem in Alexandrov
spaces with nonnegative curvature.

%%%%%%%%%%%%%%%%%%%%%%%%%

\pmed
\n
{\bf Splitting theorem.}\,\,
\par\n
We now provide the following general splitting
theorem for Alexandrov spaces with nonnegative curvature, which is an extension of \cite[Theorem 17.3]{Ya:four}.
 See \cite{MY2:dim3bdy} for a related discussion on  three-dimensional Alexandrov spaces.
%%%
See also \cite{Worn} for a splitting theorem
via a boundary stratum in the compact case.

\begin{thm}\label{thm:splitting}
Let $X$ be an Alexandrov space with nonnegative curvature. Suppose that $X$ contains two disjoint closed connected extremal
subset  $A$ and $B$ of $X$ contained in $\pa X$ satisfying 
\[
\dim A=\dim B=\dim\pa X.
\]
Then $X$ is isometric to a product 
$A\times I$ for an interval $I$. 
\end{thm}
\begin{proof}
Let $C$ denote the closure of $\pa X\setminus
(A\cup B)$.
%%%%%%%%%
If $C$ is empty, then 
\cite[Theorem 17.3]{Ya:four} shows that 
$X$ is isometric to $A\times  I$ for an interval $I$. 

Next suppose that $C$ is nonempty.
By \cite{PtPt:extremal}, $C$ is also an extremal subset of 
$X$.
Since $\dim C=\dim\pa X$, it follows from \cite{M}, the partial double  
$D_C(X)$ 
of $X$ along $C$, which is defined as
the gluing 
\[
      D_C(X)=X\cup_{C} X,
\]
is an Alexandrov space with nonnegative 
curvature. 
Set 
$$
D_C(A):= A\cup_{C\cap A} A,\qquad
D_C(B):=B\cup_{C\cap B} B.
$$
Note that 
$D_C(A)$ and $D_C(B)$
are distinct components of 
$\pa D_C(X)$.
It follows from \cite{Ya:four} again that 
$D_C(X)$ is isometric to $D_C(A)\times I$ for an interval $I$. This isometry induces an isometry
$\varphi:D_C(A)\to D_C(B)$.
For a copy $A_0$ of $A$ in $D_C(A)$, choose a 
copy $B_0$ of $B$ in $D_C(B)$ such 
that $\varphi(a_0)\in \mathring{B}_0$ for a point 
$a_0\in \mathring{A}_0$. 
Set $$
A':=\{ a\in \mathring{A}_0\,|\,\varphi(a)\in \mathring{B}_0\}.
$$
Clearly $A'$ is open. To show that $A'$ is closed in $\mathring{A}_0$,
let $a_i\in A'$ converge to a point $a\in\mathring{A}_0$,
and suppose that $b:=\varphi(a)\in \pa B_0\cap  C$.
Let $r$ denote the isometry of $D_C(X)$ defined by the reflection of the double
$D_C(X)$ about $C$.
It turns out that $b$ is a nearest point of 
$D_C(B)$ from the distinct points $a$ and $r(a)$ of $D_C(A)$. This is a contradiction
since $\varphi(a)=\varphi(r(a))$.
Thus $A'$ is closed in $\mathring{A}_0$,
and therefore $A'=\mathring{A}_0$.
This implies that 
$\varphi(A_0)=B_0$
and  $X$ is isometric to 
$A\times I$. 
This completes the proof.
\end{proof}
\pmed

%%%%%%%%%%%%%%
%%%%%%%%%%% Theorem  %%%%%%%%%%%%%%%%%%%%%%%%%%%%%

Let $Y, X, X_0$ be the limit spaces as before
with $X_0\subset X\subset Y$.
For any sequence $y_i\in X_0$ and $\e_i\to 0$,
let us consider the rescaling limit 
\begin{align} \label{eq:inradius-conv}
     \left(\frac{1}{\e_i} Y, y_i\right) \to (Y_\infty, y_\infty).
\end{align}

\par\n
{\bf Notations.}\,
From here on, let us denote by  $X_\infty$, $(X_0)_\infty$, $Y_\infty$, 
$(\pa Y)_\infty$  and $C_\infty$, $(C_0)_\infty$
 the limits of $X$, $X_0$, $Y$, $\pa Y$
and $C$, $C_0$  respectively,
with respect to a blow-up rescaling limit
like \eqref{eq:inradius-conv}
under consideration.
Obviously, $Y_\infty$ is a complete noncompact nonnegatively curved Alexandrov space with boundary $(\pa Y)_\infty$.
$1$-Lipschitz maps $\eta_\infty:C_\infty\to Y_\infty$ and $(\eta_0)_\infty:(C_0)_\infty\to (X_0)_\infty$ are defined as the limits of 
$\eta$ and $\eta_0$ respectively.
A perpendicular $\gamma_{y_\infty}^+$ and 
a perpendicular direction $\xi_{y_\infty}^+$ at 
$y_\infty$ are defined similarly as well as 
$(X_0)_\infty^k$\, $(k=1,2)$.  More explicitly,
\[
   (X_0)_\infty^k:=\{ x\in (X_0)_\infty\,|\,
        \# (\eta_0)_\infty^{-1}(x)=k\}.
\]
Note that $(X_0)_\infty^k$ is not necessarily the limit
of $X_0^k$.
For instance, $(X_0)^2_\infty$ can be nonempty while
$X_0^2$ is empty. This happens precisely at cusps.

\begin{lem} \label{lem:inrad-collapse}
For any sequence $y_i\in X_0$ and $\e_i\to 0$,
let us consider the rescaling limit 
\eqref{eq:inradius-conv}.
Then the limit $(X_\infty, y_\infty)$ of $(X,y_i)$ under \eqref{eq:inradius-conv}
is convex in $Y_\infty$. 

In particular, if $\dim\Sigma_{y_\infty}(X_\infty)=\dim\Sigma_{y_\infty}((X_0)_\infty)$
$($or equivalently, if ${\rm rad}(\xi_{y_\infty}^+)=\pi/2$$)$,
then we have $X_\infty= (X_0)_\infty$.
\end{lem}
\begin{proof} 
Take a lift $\tilde y_i\in C_0$ of $y_i$,
and let $(C_\infty,\tilde y_\infty)$ and $((C_0)_\infty,\tilde y_\infty)$
be the limits of $(C, \tilde y_i)$ and $(C_0, \tilde y_i)$ under the
same rescaling constants. It follows from the convexity of $(C_0)_\infty$ in 
$C_\infty$ that  $X_\infty$ is convex in $Y_\infty$.
It is easily seen that $(X_0)_\infty$ coincides with the topological boundary of $X_\infty$ 
in $Y_\infty$.
Therefore if $X_\infty\setminus (X_0)_\infty$ is nonempty, 
then we easily have $\dim (X_0)_\infty< \dim X_\infty$.
Thus the second conclusion is an immediate consequence.
 \end{proof}

\pmed

%%%%%%%%%%%%%%%

\begin{proof} [Proof of Theorem \ref{thm:S1extremal}]
 A rough idea of the proof is as follows.
By contradiction, we will show that there are sequences of disjoint  almost parallel  domains $U_i$ and $U_i'$ in $X_0$ (see Definition \ref{defn:DC} for the detail)
 converging to $x$ such that 
each $U_i$ also contains two disjoint almost 
parallel domains $W_{ij}$ and $W_{ij}'$.
This sounds strange and causes a contradiction.

Suppose that there is a direction 
$v\in\Sigma_x(\ca S^1\cup\ca C)\setminus\ca F_x$.
Let $\{ \tilde v, f_*(\tilde v)\}:=d\eta_0^{-1}(v)\subset
\Sigma_p(C_0)$, 
%where we may assume that
%$\tilde v\in\Sigma_c(C_0)\setminus\pa\Sigma_p(C_0)$ by slightly changing  
%$v$ if necessary.
and set $\delta:=\angle(\tilde v, f_*(\tilde v))$.
%\[
%\delta:=\max 
%\{ \angle(\tilde v, f_*(\tilde v)),
%\angle(\tilde v, \pa\Sigma_p(C_0))\}>0.
%\]
Take $p_i\in \tilde{\ca S^1}\cup\tilde{\ca C}$ converging to $p$ such that 
$\uparrow_p^{p_i}\to \tilde v$, and set 
$r_i:=|p, p_i|$.
Choose  $p_i'\in C_0$ such that 
$|p, p_i'|=r_i$ and 
$\uparrow_p^{p_i'}\to f_*(\tilde v)$.
Consider the metric balls 
\[
\tilde U_i:=\mathring{B}^{C_0}(p_i, \delta r_i/10), \quad 
\tilde U_i':=\mathring{B}^{C_0}(p_i', \delta r_i/10).
\]   
Then  
$U_i:=\eta_0(\tilde U_i)$ and $U_i':=\eta_0(\tilde U_i')$  converge to the ball $B(v, \delta/10)$  in $T_x(X_0)$
under the convergence
\beq \label{eq:conv=(1/ri,x)}
    \lim_{i\to\infty} \biggl( \frac{1}{r_i} X, x\biggr) = (T_x(X), o_x).
\eeq
More explicitly, we have 
\beq\label{eq:dH(U,U')}
\text{$d_H^{X/r_i}(U_i, U_i') < o_i$,}
\eeq
where $d_H^{X/r_i}$ is the Hausdorff-distance in  $X/r_i$.
% with the extrinsic metric induced from $X$ 
%and $\lim_{i\to\infty} o(i)=0$.

\pmed\n
Case a)\, \,$x\in\ca C$.
\pmed

In this case, $U_i$ and $U_i'$ are disjoint for large $i$.
We shall verify that $U_i$ and $U_i'$
are almost parallel in the following sense.
 
\begin{slem}\label{slem:perpatw}
Let $\check U_i:=B^{C_0}(p_i, \delta r_i/20)$.  
For any $w_i\in \eta_0(\check U_i)$, let
$w_i'$ be a nearest point of  $U_i'$
from $w_i$. Then the angle $\theta_{w_i}$ between the perpendicular $\gamma_{w_i}^+$ and any geodesic $w_iw_i'$ 
 satisfies 
\[
    \lim_{i\to \infty} \theta_{w_i} = \pi.
\]
\end{slem}
\begin{proof}
Suppose there is $w_i\in  \eta_0(\check U_i)$ satisfying 
\beq \label{eq:anglewi-pi/2}
\theta_{w_i} \le \pi-c
\eeq
for a positive constant $c$ independent of $i$.
Set $\delta_i:=|w_i,w_i'|$, and
consider
$$ 
\hat Y_i:= \{ y\in Y\,|\, |y,X|\le\delta_i\}
    \subset Y.  
$$
Note that $\lim_{i\to\infty}\delta_i/r_i=0$.
Put 
$$ 
Q_i:= \eta(\{\delta_i\}\times\tilde U_i),\quad
Q_i':=\eta(\{\delta_i\}\times\tilde U_i')
\subset \pa \hat Y_i.  
$$
%(see Section \ref{sec:non-inradius}).
%%
Passing to a subsequence, we may assume that 
$(\frac{1}{\delta_i}\hat Y_i, w_i)$
converges to a pointed space $(\hat Y_\infty,w_\infty)$.
Note that $\hat Y_\infty$ is a complete noncompact 
Alexandrov space with nonnegative curvature.
Let $Q_\infty$ and $Q_\infty'$ be the 
limits of $Q_i$ and $Q_i'$ under this convergence.
From $d(Q_i, Q_i')\ge 2\delta_i$, 
$Q_\infty$ and $Q_\infty'$ are disjoint.
%%%%
Although $\hat Y_i$ is not an Alexandrov space, it is easy to verify that 
$Q_\infty$ and $Q_\infty'$ are extremal subsets of $\hat Y_\infty$, which are contained in $\pa \hat Y_\infty$ with $\dim Q_\infty=\dim Q_\infty'=\dim \pa \hat Y_\infty$.
Theorem \ref{thm:splitting} 
shows that 
$\hat Y_\infty$ is isometric to $Q_\infty\times I$ for an interval $I$.
Let $\hat w_i\in Q_i$ and $\hat w_i'\in Q_i'$
be the points corresponding to 
$w_i$ and $w_i'$ respectively.
Note that the union of the geodesics $w_iw_i'$ and
$w_i'\hat w_i'$ provides a shortest curve from
$w_i$ to $Q_i'$. Therefore from the splitting
$\hat Y_\infty=Q_\infty\times I$, 
the geodesics $w_\infty w_\infty'$ and
$w_\infty'\hat w_\infty'$ is a subarc of shortest
geodesic joining $\hat w_\infty'$ to $Q_\infty$.
It turns out that the union of geodesics
$\hat w_i w_i$, $w_iw_i'$ and $w_i' \hat w_i'$
converges to a minimal geodesic 
between $Q_\infty$ and $Q_\infty'$.
This is a contradiction to \eqref{eq:anglewi-pi/2}, and completes the proof of Sublemma  \ref{slem:perpatw}.
\end{proof}

We fix large $i$ so that $p_i\in\tilde{\ca C}$, and take distinct $\tilde v_i, \tilde v_i'\in\Sigma_{p_i}(C_0)$ such that 
$d\eta_0(\tilde v_i)=d\eta_0(\tilde v_i')$.
Choose sequences $q_{ij}, q_{ij}'\in C_0$ converging to $p_i$ as 
$j\to\infty$ such that 
\beqq
\text{$\tilde v_i=\lim_{j\to\infty}\uparrow_{p_i}^{q_{ij}}$, \quad
$\tilde v_i'=\lim_{j\to\infty}\uparrow_{p_i}^{q_{ij}'}$, \quad $|p_i, q_{ij}|=|p_i, q_{ij}'|$.}
\eeqq
 Take large enough $j$ with $s_{ij}:=|p_i, q_{ij}|<\delta r_i/100$.
%%%
Consider the convergence
\beq
\lim_{j\to\infty} \biggl( \frac{1}{s_{ij}} X, x_i\biggr) = (T_{x_i}(X), o_{x_i}),
\eeq
where $x_i:=\eta_0(p_i)$.
%%%
Set $\delta_i:=\angle(\tilde v_i,\tilde v_i')$, and 
\beq \label{eq:ballsWW}
W_{ij}:=\eta_0(B(q_{ij},\delta_is_{ij}/10)), \quad
W_{ij}':=\eta_0(B(q_{ij}',\delta_is_{ij}/10)),
\eeq

$W_{ij}$ and $W_{ij}'$ are disjoint and contained in $U_i$,
and both converge to $B(v_i,\delta_i/10)$
under the above convergence,
where $v_i:=d\eta_0(\tilde v_i)$.
%%%%
Therefore in a similar way, we conclude that $W_{ij}$ and 
$W_{ij}'$ are almost parallel in the same sense 
as Sublemma \ref{slem:perpatw}.

Take larger $j=j(i)$ satisfying
\beq \label{eq:largerJ}
   d_H^{X/s_{ij}}(W_{ij},W_{ij}')\ll |U_i,U_i'|.
\eeq
Let  $y_{ij}'$ be a nearest point of 
$W_{ij}'$ from $y_{ij}:=\eta_0(q_{ij})$.
Since $y_{ij}$ is not contained in $U_i'$,
we can take a nearest point $z_{ij}$ of 
$U_i'$ from $y_{ij}$. 
By Sublemma \ref{slem:perpatw}, both
$\uparrow_{y_{ij}}^{y_{ij}'}$ and 
$\uparrow_{y_{ij}}^{z_{ij}}$ are almost orthogonal to $X_0$,
%%%
which implies  $\angle z_{ij}y_{ij}y_{ij}' <o_i$.
It follows from \eqref{eq:largerJ} that
\beq\label{wangle(yy'z)}
       \wangle y_{ij}y_{ij}'z_{ij}>\pi-o_i.
\eeq
%%% 
Note that the union $\gamma_{ij}$ of the geodesic
$y_{ij}y_{ij'}$ and the perpendicular 
$\gamma^+_{y_{ij}'}$ is shortest near the point $y_{ij}'$.
%%%
\eqref{wangle(yy'z)}implies that 
\[
    \angle(\dot\gamma_{y_{ij}',z_{ij}}(0), \dot\gamma_{y_{ij}'}^+(0))<o_i.
\]
%%%%
 In view of \eqref{eq:dH(U,U')},
it is now easy verify  that the perpendicular 
$\gamma^+_{y_{ij}'}$ meets $U_i'$ at a point
near $z_{ij}$.
This is a contradiction. 
%\end{proof}

\pmed
%%%%%%%%%%%%%%%%%%%%
\begin{center}
\begin{tikzpicture}
[scale = 0.5]

\draw[-, thick](0,0)circle[x radius=6,y radius=2];
\fill (-6.5, 0) circle (0pt) node [left] {$U_i$};
\draw[-](3,0.5)circle[x radius=1.1,y radius=0.5];
\fill (2, 0.5) circle (0pt) node [left] {{\tiny $W_{ij}$}};
\fill (3,0.5) circle (2pt) ;
\fill (3,0.5) circle (0pt) node [right] {{\tiny $y_{ij}$}};

\draw[-](3,-0.6)circle[x radius=1.1,y radius=0.5];
\fill (2, -0.6) circle (0pt) node [left] {{\tiny $W_{ij}'$}};
\fill (3,-0.6) circle (2pt) ;
\fill (3,-0.6) circle (0pt) node [right] {{\tiny $y_{ij}'$}};
\fill (2.77,-0.6) circle (2.5pt) ;

\draw[-, thick](0,-3.5)circle[x radius=6,y radius=1.2];
\fill (-6.5, -3.5) circle (0pt) node [left] {$U_i'$};
\fill (2.8,-3.5) circle (2pt) ;
\fill (2.8,-3.5) circle (0pt) node [right] {{\tiny $z_{ij}$}};

\draw [dotted,very thick] (3,0.5)  to [out=255, in=95] (2.8,-3.5);
\draw [dotted,very thick] (3,0.5)  to [out=275, in=85] (3,-0.6);
\end{tikzpicture}
\end{center}
\vspace{-0.5cm}  
\begin{figure}[htbp]
  \centering
  \caption{}
  \label{fig:W-W'}  
\end{figure}  

\pmed
Next we consider the case $x\in\ca S^1$.
In this case, $U_i$ and $U_i'$ may have 
nonempty intersection. Note that 
$U_i\cap U_i'\subset X_0^2$.

\pmed\n
Case b)\, $x\in \ca S^1$ and $p_i\in \tilde{\ca C}$.
\pmed\n

First we apply the argument after Sublemma 
\ref{slem:perpatw} in  Case a)  to the cusp $p_i$,
and choose sequences $q_{ij}, q_{ij}'\in C_0$ converging to $p_i$ as 
$j\to\infty$ such that 
the domains $W_{ij}$ and $W_{ij}'$ around 
$y_{ij}=\eta_0(q_{ij})$ and $\eta_0(q_{ij}')$ respectively, 
defined as in \eqref{eq:ballsWW}
are almost parallel.
Note that  both $W_{ij}$ and $W_{ij}'$ are contained
in ${\rm int} X_0^1$ and hence do not meet 
$U_i'$.
Take a nearest point $z_{ij}$ of
$U_i'$  (resp. $y_{ij}'$ of
$W_{ij}'$) from $y_{ij}$.

Let $\sigma_{ij}$ and  $\gamma_{ij}$ be $Y$-geodesic from $y_{ij}$ to $y_{ij}'$ and from $y_{ij}$ to $z_{ij}$
respectively.
First consider the convergence
\[
  \left(\frac{1}{|y_{ij}, y_{ij}'|}Y, y_{ij}\right) \to (Y_\infty, y_\infty).
\]
From here on,  we omit the subindex $i$ for simplicity.
By the argument in the proof of Sublemma,
\ref{slem:perpatw}, the limit $X_\infty$ of
$X$ under the above convergence is isometric to a 
product $Q_\infty\times I$, where 
$(X_0)_\infty=Q_\infty\times \pa I$.
It follows that the limit $\gamma_\infty$ of $\gamma_{ij}$
must be contained in $(X_0)_\infty$. This implies 
\beq \label{eq:jangle(gamma,sigma)}
       \lim_{j\to\infty} \angle(\dot\gamma_{ij}(0),\dot\sigma_{ij}(0))=\pi/2.
\eeq
Next consider the convergence
\[
  \left(\frac{1}{|y_{ij}, z_{ij}|}Y, y_{ij}\right) \to (\hat Y_\infty,  y_\infty).
\]
Let $\hat X_\infty$ and $(\hat X_0)_\infty$ be the 
limits of $X$ and $X_0$ under the above 
convergence respectively.
Since  $\hat X_\infty$  is convex, from 
$z_\infty\in (\hat X_0)_\infty$, we conclude 
$\gamma_\infty\subset (\hat X_0)_\infty$.
However the perpendicular $\gamma^+_{z_\infty}$
at $z_\infty$ makes an angle $\pi/2$ with any direction
in $\Sigma_{z_\infty}((\hat X_0)_\infty)$.
Since $\gamma_\infty$ and $\gamma^+_{z_\infty}$
form a geodesic, this is a contradiction.

\pmed
\n
Case c)\, $x\in \ca S^1$ and $p_i\in \tilde{\ca S}^1$.
\pmed\n

%Set $r_i:=|x,x_i|$.
For each $i$, take a sequence $q_{ij}\in C_0^2$ with 
$\lim_{j\to\infty}q_{ij}=p_i$. Take large enough $j$ with $|p_i, q_{ij}|<\delta r_i/20$.
%Suppoose first $q_{ij}\in\tilde{\ca S^2}$.
Choose small neighborhoods $\tilde W_{ij}:=B(q_{ij},s_{ij})$ and 
$\tilde W_{ij}':=B(f(q_{ij}),s_{ij})$ in $C_0$ with $s_{ij}\ll  |q_{ij}, f(q_{ij})|$.
Then $\eta_0(\tilde W_{ij})$ and $\eta_0(\tilde W_{ij}')$
are tangent at $y_{ij}:=\eta_0(q_{ij})$ (see also Sublemma \ref{slem:eta-ball}).
%%%%
Note that $y_{ij}$ is not contained in $U_i'$,  since otherwise, we would have 
the contradiction $\# \eta_0^{-1}(y_{ij})\ge 3$.
Now take a nearest point $z_{ij}$  of 
$U_i'$ from $y_{ij}$.

Let $\gamma_{ij}$ be an $Y$-minimal geodesic joining
$y_{ij}$ to $z_{ij}$, and consider the convergence
\[
  \left(\frac{1}{|y_{ij}, z_{ij}|}Y, y_{ij}\right) \to (Y_\infty, y_\infty),
\]
where we omit the subindex $i$ for simplicity.
Let $\gamma_\infty:[0,1]\to Y_\infty$ and $\gamma_{z_\infty}^+$ be 
the limits of $\gamma_{ij}$ and $\gamma_{z_{ij}}^+$ respectively.
From $y_{ij}\in X_0^2$, we have 
$\dot\gamma_\infty(0)\in\Sigma_{y_\infty}((X_0)_\infty)$ while $z_\infty\in (X_0)_\infty$.
Since $(X_0)_\infty\subset\pa X_\infty$ 
and since $X_\infty$ is convex in $Y_\infty$ 
by Lemma \ref{lem:inrad-collapse},
we have $\gamma_\infty\subset\pa X_\infty$.
 However, from construction we have
$\gamma_\infty(1-\e)\in X_\infty\setminus\pa X_\infty$ for small $\e>0$.
This is a contradiction.
This completes the proof of  Theorem \ref{thm:S1extremal}.
\end{proof}

\psmall 

Recall that $\ca S^1$ is closed in $X_0^1$ (see Lemma  \ref{lem:closed-S1}).
As the following example shows, this is not the case for $\ca C$.

%%%%%%%%%%%
\psmall
\begin{ex} \label{ex:smallS3}
In Example \ref{ex:general-cusp}, let us change 
the construction only the function $f$ such that instead of the  condition (1) in Example \ref{ex:general-cusp}, we assume 
\begin{enumerate}
\item[(1')] $f^{-1}(0)=\pa D\cup Q$,
\end{enumerate}
where 
$Q$ in an infinite sequence $q_i$ in 
$\mathring {D}$ converging 
to a point $x\in \pa D$ such that 
$v_0:=\lim_{i\to\infty}\uparrow_x^{q_i}$
is perpendicular to $\Sigma_x(\pa D)$.
Define $M_\e$ in the same way.
In this case, $\ca S^1=\{ x\}$ and 
$\ca C=\pa D\setminus \{ x\}$ is not closed in $N_0$. 
\end{ex}

\pmed\n
{\bf Extremal subsets.} \, 
\psmall   
Concerning Theorem \ref{thm:S1extremal},
we define the notion of extremal subsets in
our limit spaces, and check if $\ca S^1\cup\ca C$ can be an extremal subset of $X$.

\begin{defn}\label{defn:extremal-subset}
Following \cite{PtPt:extremal}, we say that a closed subset $E$ of a geodesic space $X$ is  {\it extremal in $X$} 
 if 
for any $x\in X\setminus E$ 
the distance function 
$d_x^{X}:E\to \mathbb R$  takes a local minimum at 
$y\in E$, then we have  
\[
\limsup_{z\to y} \wangle^{X_0} xyz\le\pi/2.
\]
This is equivalent to the usual notion of 
extremal subsets defined in Section \ref
{ssec:Alex} when $X$ is an Alexandrov space.
\end{defn}  

  For our limit spaces, $X_0$ is extremal in $X$.
In Example \ref{ex:smallS3}, 
$\ca S^1\cup\ca C=\pa D$ is an extremal subset of $N_0$.

Now we give examples showing that 
$f_*$ is not always a reflection even for a cusp, and 
$\ca S^1\cup\ca C$ is not necessarily an extremal subset of $X_0$.

%%%%%%%%%%%%%%%%%%%%

\begin{defn}\label{defn:C(k)-S(k)}
For each integer $0\le k\le \dim X_0-2$,
we denote by $\ca S^1(k)$ (resp. $\ca C(k)$) the set of 
all points $x$ of $\ca S^1$ (resp.  of $\ca C$) such that 
the isometry $f_*$ on $\Sigma_{\tilde x}(C_0)$ has the fixed 
point set of dimension $k$, where $\eta_0(\tilde x)=x$.
\end{defn}

%%%%%%%%%%%%%%%%%%%

We now  construct an example 
with $\dim \ca C(k)=k+1$ (compare to Theorem \ref{thm:dim(metric-sing)}).

\begin{ex}\label{ex:non-closed} (1)\,
Let $T^n:=\R^n/\Z^n$ be the flat $n$-torus
with the base point ${\bf {\bf [0]}}$.
For any integer $0\le k\le n-2$, consider the 
decomposition $T^n=T^{k+1}\times T^{n-k-1}$.
We choose a smooth  function $g:T^n\to\R_+$ satisfying
\benu
 \item[(a)] $g^{-1}(0)=T^{k+1}\times {\bf [0]}\,;$
 \item[(b)] $g$ is invariant under the symmetry about $T^{k+1}\times [0]$. Namely,
$
         g([{\bf x_1, -x_2}]) = ([{\bf x_1,x_2}]) 
$ for $[{\bf x_1,x_2}]\in T^{k+1}\times T^{n-k-1}$.
\eenu 
Set
\begin{gather*}
    L_i:= \{ ([{\bf x}], t)\in\mathbb T^n\times\R\,|\, |t|\le g([{\bf x}])+1/i\}.
\end{gather*}
Let  $s_i$ be the isometric 
involution 
defined by  
\begin{align*}
&s_i([{\bf x_1, x_2}], t, u)=([{\bf x_1, -x_2}],-t, -u).
\end{align*}
Consider  
\[
           M_i:= (L_i\times S^1_{1/i})/s_i,
\]
which is contained in 
$\ca M_b(n+2,0, \lambda,d)$ for some $\lambda, d$.
As $i\to\infty$, $M_i$ converges to 
\[
          N:=\{ ({\bf x},t)\in\mathbb T^n\times \R\,|\, |t|\le g([{\bf x}])\}/s_\infty,
\]
where 
$s_\infty$ is the isometric involution defined as 
$$
s_\infty([{\bf x_1, x_2}],t)=([{\bf x_1,-x_2}],-t).
$$
From construction,  we obtain  
\begin{align*}
   &N_0=\{ ([{\bf x}],t)\in\mathbb T^n\times\R\,|\, |t|=g([{\bf x}])\}/s_\infty,\\
& C_0=\{ ([{\bf x}],t)\in\mathbb R^{n+1}\,|\, t=g([{\bf x}])\},
\end{align*}
where the map $\eta_0:C_0\to N_0$ is given by 
$\eta_0(x,g(x))=[(x,g(x))]$, which is bijective.
Thus we have $N_0=N_0^1$.
For any $u:=([{\bf x_1,0}], 0)\in N_0$,
note that  
$\Sigma_{\tilde u}(C_0)=\mathbb S^{n-1}$ ($\tilde u=\eta_0^{-1}(u)$) and 
$\Sigma_{u}(N_0)=\mathbb S^{n-1}/\hat s_\infty$ 
with $\ca F_u={\rm Fix}(\hat s_\infty)=\mathbb S^{k}$,
where 
$\hat s_\infty$ denotes the restriction of 
$s_\infty$.
Thus we have $u\in \ca C(k)$, and
$\ca C(k)$ is isometric to $T^{k+1}$.
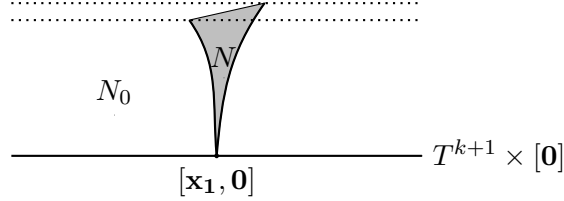
\pmed
 %%%%%%%%%%%%%%%%%%%%
\begin{center}
\begin{tikzpicture}
[scale = 0.45]

\filldraw[fill=lightgray, opacity=.1] 
(0,0)  [out=95, in=-55] to (-0.8,4)
 (-0.8,4) -- (1.4,4.5)
 [out=235, in=85] to (0,0);

\draw [-, thick] (-6,0) to  (6,0);
\draw [-, thick] (0,0)  [out=95, in=-55] to
(-0.8,4) ;
\draw [-, thick] (0,0)  [out=85, in=235] to
(1.4,4.5) ;
\draw [dotted, thick] (-6,4) to  (6,4);
\draw [dotted, thick] (-6,4.5) to  (6,4.5);

\fill (0, 0) circle (2pt) node [below] {\small${\bf [x_1,0]}$};
\fill (6, 0) circle (0pt) node [right] {\small $T^{k+1}\times [{\bf 0}]$};

\fill (0.2, 2.3) circle (0pt) node [above] {\small$N$};
\fill (-3, 1.2) circle (0pt) node [above] {\small$N_0$};
\end{tikzpicture}
\end{center}
\vspace{-0.5cm}  
\begin{figure}[htbp]
  \centering
  \caption{}
  \label{fig:cusp}  
\end{figure}  

%%%%%%%%%%%%%%%%%%%%%%%%%%%%%
\psmall 
Replacing the condition (1) in Example \ref{ex:non-closed},
by a condition similar  to $(1)'$ in Example \ref{rem:CtoS1},
one can construct an example 
with $\dim \ca S^1(k)=k+1$. 

(2)\, We slightly change the function $g$ in (1)
such that 
\begin{enumerate}
\item[(a)']  \,\, $g^{-1}(0)={\bf [0]}\in T^n$,
\end{enumerate}
instead of $(a)$ in (1). Similarly, 
 we have $N_0=N_0^1$,
and $\ca C$ consists of the single point 
$x:=([{\bf 0}], 0)\in N_0$.
However we still  have  
$\ca F_x=\mathbb S^{k}$
as in $(1)$.
\end{ex}  
\psmall
Example \ref{ex:non-closed}(2) shows that
for $x\in \ca S^1\cup\ca C$, 
$\Sigma_x(\ca S^1\cup\ca C)=\ca F_x$ 
 does not hold in general, and  
$\ca S^1\cup\ca C$ is not an extremal subset of
$N$.

\pmed

%%%%%%%%%%%%%%%%%%%%%%%%
\section{Geometry of almost parallel domains} \label{sec:local-str-S1}
\pmed

 In this section, we prove 
Theorem \ref{thm:S1-nontrivial-f*}.

%%%%%%%%%%%%%%%%%%%%%
\subsection{Reduction and preliminaries}\label{ssec:bdypt}

First we give a reduction of Theorem \ref{thm:S1-nontrivial-f*} to 
Theorem \ref{thm:never-happen}.

\pmed\n
{\bf Reduction.}\,
For any $p\in C_0\cap \pa C$, from \eqref{eq:defC},
we have the following:
\begin{align*}%\label{eq:Sigma-C0}
\pa\Sigma_p(C)=\Sigma_p(C_0)\cup(\{\tilde\xi_p^+\}*\pa\Sigma_p(C_0)).
\end{align*}

\begin{lem}\label{lem:tirivial=paY}
For any 
$x\in\pa Y\cap ({\rm int}\,X_0^1\setminus\ca C)$, suppose  
${\rm rad}(\xi_x^+)=\pi/2$. 
Then we have 
\beq \label{eq:paSigma-never}
\pa\Sigma_x(Y)=\Sigma_x(X_0)\cup
 (\{\xi_x^+\}*\pa\Sigma_x(X_0)).
\eeq
\end{lem}
\begin{proof}
Since $f_*={\rm id}$, the conclusion follows from 
Theorem \ref{thm:X1-f}.
\end{proof}   

\begin{lem}\label{lem:conv-above}
\eqref{eq:paSigma-never} is equivalent to 
\[
 \text{$x\in X_0$, ${\rm rad}(\xi_x^+)=\pi/2$ and $f_*={\rm id}$.}
\]
\end{lem}
\begin{proof} \eqref{eq:paSigma-never}  certainly implies 
$x\in X_0$ and ${\rm rad}(\xi_x^+)=\pi/2$. 
If $f_*$ is not the identity, $\Sigma_x(X_0)$ cannot be contained in 
$\pa \Sigma_x(Y)$ since $\Sigma_x(X_0)^2$ is open.
The converse is an easy consequence from Theorem \ref{thm:X1-f}.
\end{proof}

In what follows,  we shall prove that the converse to 
Lemma \ref{lem:tirivial=paY} is also 
true.

\begin{thm}\label{thm:never-happen}
For a point  $x\in X_0^1\cap\pa Y$, assume \eqref{eq:paSigma-never}.
Then there is an $r>0$ such that 
$\mathring{B}^{X_0}(x,r)\subset {\rm int}\,X_0^1\setminus \ca C$.

In particular, we have $x\in {\rm int}\, X_0^1\setminus \ca C$.
\end{thm} 

For an example concerning Theorem \ref
{thm:never-happen}, see Example \ref{ex:X01cosh}(2).

\begin{rem}\label{rem:indirect}\upshape
Here is a remark on the proof of Theorem \ref{thm:never-happen}.
Suppose that there is a sequence $y_i\in X_0^2$
converging to $x$. 
We may assume that 
$\uparrow_x^{y_i}$ converges to a direction 
$v\in\Sigma_x(X_0)$.
Let $\gamma_{y_i}^{\pm}$ be the 
two perpendiculars at $y_i$. Joining $x$ to
$\gamma_{y_i}^\pm(t)$ by $Y$-minimal geodesics
for all $t\in[0,t_0]$
and letting $i\to\infty$,
we have  minimal geodesics $\sigma_{\pm}$
joining
$v$ to $\xi_x^+$.
If $\sigma_+\neq\sigma_-$, then 
we have a contradiction to the 
assumption \eqref{eq:paSigma-never},
and we could conclude the proof.
However, we do not know if $\sigma_+\neq\sigma_-$.
This suggests that the proof must be  much more involved.
In what follows, we go somewhat indirectly to reach 
the completion of the proof.
\end{rem}
 
\psmall

Theorem \ref{thm:never-happen} means that 
the infinitesimal data \eqref{eq:paSigma-never} completely determines 
the local information of the space $X_0$. This is a rare case in our spaces whose local geometry may be 
complicated in general.

The proof of Theorem \ref{thm:never-happen} is deferred to  Subsection \ref{ssec:proof}.

\begin{proof}[Proof of Theorem \ref{thm:S1-nontrivial-f*}
assuming Theorem \ref{thm:never-happen}]
For $x\in \ca S^1$, suppose that $f_*$ is the identity.
From  Theorem \ref{thm:X1-f}, it turns out that 
 $\Sigma_x(Y)$ is isometric to $\Sigma_p(C)$.
Thus we have \eqref{eq:paSigma-never}. 
Since $x\in\ca S^1$, this is a contradiction to 
Theorem \ref{thm:never-happen}.
\end{proof}

%%%%
From now, towards the proof of Theorem \ref{thm:never-happen},
we prepare some materials.
\pmed\n
{\bf Infinitesimal structure at $X_0^1\cap\pa Y$.}

\pmed

The following lemma will be needed
several times later on.

\begin{lem} \label{lem:frequent-later}
For every $x\in X_0^1\cap\pa Y$ with 
${\rm rad}(\xi_x^+)=\pi/2$, 
consider a sequence $y_i\in X_0\setminus \pa Y$ converging to $x$.
Let $w_i$ be a nearest point of $\pa Y$ from $y_i$.
We assume that both $\uparrow_x^{y_i}$ and $\uparrow_x^{w_i}$ converge to the same direction, say $v\in\Sigma_x(X_0)$, and consider the rescaling limit 
\begin{align} \label{eq:conv-yw-Yinfty}
   \left(\frac{1}{|w_i, y_i|}Y, w_i\right)  \to (Y_\infty, w_\infty),
\end{align}
Then if $w_\infty\in (X_0)_\infty$, then 
a minimal geodesic joining $v$ and $\xi_x^+$ is contained in 
$\pa\Sigma_x(Y)$. 

In particular, we have  $v\in\Sigma_x(\pa_*X_0)$, and hence
$\pa\Sigma_x(X_0)$ is nonempty.
\end{lem}
\begin{proof}
Let  $y_{\infty}$ be the limit of $y_i$  under \eqref{eq:conv-yw-Yinfty}. 
From $w_\infty\in (X_0)_\infty$, we can 
define a perpendicular  $\gamma_{w_\infty}^+$ at $w_\infty$ to $(X_0)_\infty$.
 By the hypothesis on $w_i$, we have 
$\angle_{w_i}(\uparrow_{w_i}^{y_i}, \xi)\le \pi/2$ for all $\xi\in \Sigma_{w_i}(Y)$.
It follows that $\angle_{w_\infty}(\uparrow_{w_\infty}^{y_\infty},  \gamma_{w_\infty}^+ )\le \pi/2$.
The convexity of $X_\infty$ in $Y_\infty$ 
(Lemma \ref{lem:inrad-collapse}) then implies that  
$\angle(\uparrow_{w_\infty}^{y_\infty}, \gamma_{w_\infty}^+)=\pi/2$, which yields
$\gamma_{w_\infty}^+\subset (\pa Y)_\infty$. 
Choose $\hat u_i\in\pa Y$ converging  to 
$\gamma_{w_\infty}^+(1)$ under \eqref{eq:conv-yw-Yinfty} and set 
$u_i:=\pi(\hat u_i)\in X_0$.
Then the perpendicular $\gamma_{u_i}^+$ is contained in $\pa Y$ and converging to 
$\gamma_{w_\infty}^+$.
 Since $|u_i, w_i|/|y_i, w_i|\to 0$, we certainly have  
$\uparrow_x^{u_i} \to v$.
 It follows from $u_i\in\pa_* X_0$ that $v\in\Sigma_x(\pa_* X_0)\subset\pa\Sigma_x(X_0)$
(see Theorem \ref{thm:X1-f} and Proposition \ref{prop:summary-bdy}).

Now let us consider the convergence
\begin{align} \label{eq:conv-xu-Yinfty}
   \left(\frac{1}{|x, u_i|}Y,u_i\right)  \to (T_x(Y), v).
\end{align}
Let $\gamma_v^+$ be the limit of the perpendicular $\gamma_{u_i}^+$ under 
\eqref{eq:conv-xu-Yinfty},  which is contained in 
$\pa T_x(Y)$.
Since the geodesic rays from $o_x$ in the directions to all the points of 
$\gamma_v^+$ must be contained in $\pa T_x(Y)=K(\Sigma_x(\pa Y))$,
the corresponding minimal geodesic $\xi_x^+ v$ must be contained in $\pa\Sigma_x(Y)$.
\end{proof}

\pmed\n
{\bf Causing by inradius collapse.}
\psmall
In the proof of Theorem \ref{thm:never-happen}, we often use rescaling like 
\eqref{eq:inradius-conv}.  In such a new convergence,  we sometimes encounter with inradius collapsing, whose geometry is much simpler.

\begin{defn} \label{defn:inradius-conv}
We say that the convergence \eqref{eq:inradius-conv} is {\it caused by an inradius collapse} if the latter half of the conclusion of Lemma \ref{lem:inrad-collapse} holds,  that is,
$X_\infty=(X_0)_\infty$.
\end{defn}

\psmall We use the following lemma 
in the next subsection.
Recall our original situation that a sequence
$M_i$ in $\ca M(n,\kappa,\nu,\lambda,d)$
converges to a geodesic space $N=X^{\rm int}$,
and consider the convergence \eqref{eq:inradius-conv}.

\begin{lem}\label{lem:caused-by}
Assume that  the convergence \eqref{eq:inradius-conv} is caused by an inradius collapse. Let $q_{mi}\in\pa M_{m}$ converges to
$y_i\in X_0$ under the convergence
$\tilde M_m\to Y$ as $m\to\infty$.
Then there is a subsequence $\{ m_i\}_{i=1}^\infty$ of $\{ m\}$ such that 
$\bigl(M_{m_i}/\e_i, q_{m_i i}\bigr)$ inradius collapses
to $(X_\infty,y_\infty)$ as  $i\to\infty$,
in the sense that 
for any $R>0$, the inradius of 
the open $R$-ball in $M_{m_i}/\e_i$ around 
$q_{m_ii}$ converges to $0$ as $i\to\infty$.
\end{lem}
\begin{proof}
Choose an $o_m$-approximation 
$\varphi_m:\tilde M_m\to Y$ that restricts to
$o_m$-approximations 
$\varphi_m|_{M_m}:M_m^{\rm ext}\to X$ and  
$\varphi_m|_{\pa M_m}:(\pa M_m)^{\rm ext}\to X_0$ with $\lim_{\to\infty} o_m=0$,
where the superscript ''ext'' denotes the extrinsic metric induced from $\tilde M_m$. Then 
%by Proposition \ref{prop:extendAS},
$\varphi_m|_{M_m}:M_m^{\rm ext}/\e_i\to X/\e_i$
is an $o_m/\e_i$-approximation, 
which restricts to an  $o_m/\e_i$-approximation
$\varphi_m|_{\pa M_m}:(\pa M_m)^{\rm ext}/\e_i\to X_0/\e_i$.
Since $X_\infty=(X_0)_\infty$
and ${\rm inrad}(M_m)={\rm inrad}(M_m^{\rm ext})$,
the conclusion follows immediately if $m=m_i\gg i$ with $\lim_{i\to\infty}o_{m_i}/\e_i=0$.
\end{proof}

The following lemma provides a sufficient condition
for being caused by an inradius collapse,
 and will be useful since the hypothesis
of the lemma is satisfied under the presence 
of almost parallels defined in the next subsection. See also  Sublemma \ref{slem:perpatw}.

\begin{lem} \label{lem:almost-parallel}
Let $y_i,z_i\in X_0$ and  $\e_i>0$ be  sequences
satisfying 
\begin{enumerate}
\item $\lim_{i\to\infty}\angle(\xi_{y_i}^+,\uparrow_{y_i}^{z_i})=\pi$, \quad $\lim_{i\to\infty}\angle(\xi_{z_i}^+,\uparrow_{z_i}^{y_i})=\pi \,;$
\item 
    $\lim_{i\to\infty} \e_i=0$,\quad
     $\lim_{i\to\infty} |y_i,z_i|_Y/\e_i =0$.
\end{enumerate}
Then  the convergence \eqref{eq:inradius-conv}
is caused by an inradius collapse.
\end{lem}
\begin{proof}
If $y_\infty\in (X_0)_\infty^2$, then $\Sigma_{y_\infty}(Y_\infty)$ is the spherical suspension over
$\Sigma_{y_\infty}((X_0)_\infty)$, and the conclusion follows from Lemma \ref{lem:inrad-collapse}.
Suppose  $y_\infty\in (X_0)_\infty^1$.
It suffices to show  
$\alpha:={\rm rad}(\xi_{y_\infty}^+)=\pi/2$.
%%%%%
Suppose $\alpha>\pi/2$.
Take  $\alpha>\beta>\pi/2$, a point $a_\infty\in{\rm int}\,X_\infty$ and $c>0$
such that 
$$
\tilde\angle \gamma_{y_\infty}^+(c) y_\infty a_\infty >  \beta.
$$
Choose $a_i\in {\rm int} X$ converging to $a_\infty$.
Since both $\gamma^+_{y_{i}}$ and 
$\gamma^+_{z_{i}}$ converge to 
$\gamma_{y_\infty}^+$,
 we have for large $i$
\beq
\begin{cases}
\begin{aligned} \label{eq:0wangle-gamma(yza)}
 &\angle \gamma^+_{y_{i}}(c) y_{i} a_i \ge\tilde\angle \gamma^+_{y_{i}}(c) y_{i} a_i > \beta, \\
& \angle \gamma^+_{z_{i}}(c) z_{i} a_i\ge
  \tilde\angle \gamma^+_{z_{i}}(c) z_{i} a_i> \beta.
\end{aligned}
\end{cases}
\eeq
From the condition (1) and  \eqref{eq:0wangle-gamma(yza)},
setting $\zeta:=\beta-\pi/2>0$,  we obtain  
\begin{align*}
&\tilde\angle a_i y_{i} z_{i}\le 
\angle a_iy_{i} z_{i}\le \pi-\beta+o_i<\pi/2-\zeta/2, \quad
\\
&\tilde\angle a_i z_{i} y_{i}\le \angle a_iz_{i} y_{i}\le \pi- \beta+o_i<\pi/2-\zeta/2,
\end{align*}
for large enough $i$.
Since $\lim_{i\to\infty}\tilde\angle y_{i} a_i z_{i}=0$, this is a contradiction.
This completes the proof.
\end{proof}

%%%%%%%%%%%%%%%%%%%%%%%%%%%%
\psmall 
 \subsection{Proof of Theorem \ref{thm:never-happen}}   \label{ssec:proof}
  %%%%%%%%%%%%%%%

Since the  proof of Theorem \ref{thm:never-happen}
is rather long, let us begin with 
\pmed\n
{\bf Strategy for the proof of Theorem \ref{thm:never-happen}.}
Under the assumption \eqref{eq:paSigma-never},
suppose $x\in\ca S^1$. We may assume there is a sequence $y_i\in\ca S^2$ converging to $x$.
As a remarkable feature of elements of $\ca S^2$,
there are two disjoint domains of $X_0$ 
in any neighborhood of $y_i$ that are very close and almost parallel to each other.
The same occurs when $y_i$ are cusps.
Actually, we work in a more controlled framework
of almost parallels. This yields the following 
Definition \ref{defn:DC}.
Here we roughly denote by $\ca D$ the set of 
points of $X_0$ having almost parallels, and 
by $\ca E$ the complement of $\ca D$ in $X_0$.

\begin{itemize}
\item
In the first main step, we show that for small enough $r$, $\ca E$ occupies the large part of
$B^{X_0}(x,r)$ except a very thin region around 
$B^{X_0}(x,r)\cap\pa_* X_0\,;$
\item
In the second main step, we actually show that 
$B^{X_0}(x,r)\subset \ca E$, which yields the 
conclusion.
\end{itemize}
\psmall   
In this subsection, we always assume
$x\in X_0^1$.

Taking the hypothesis of Lemma \ref{lem:almost-parallel} into account, we give 
the definition of almost parallels below.

\begin{defn} \label{defn:DC}
We fix a positive continuous function $\theta=\theta(t)$ with $\lim_{t\to 0} \theta(t)=0$.
For $x\in X_0^1$ and $r$, let $\ca D(x,r,\theta)$  be the set of points
$y\in B^{X_0}(x, r)$ such that there are a constant $s>0$
and a point $z\in X_0$ satisfying the following:
\begin{enumerate}
 \item  For some lifts $\tilde y, \tilde z$ of $y, z$,  let $U(y,s):=\eta_0(\mathring{B}^{C_0}(\tilde y, s))$,  $V(z,s):=\eta_0(\mathring{B}^{C_0}(\tilde z, s))$.
 Then we have 
\[
     U(y, s)\cap  V(z, s) = \emptyset,\qquad
 U(y,s), V(z,s) \subset {\rm int}\,X_0\,;  \]
 \item The distance function $d_y^Y$ from $y$ in $Y$ restricted to 
  $V(z,s)$ has a minimum at the point $z\,;$ 
 \item $d_y^Y(z)/s \le \theta(|x,y|)$.
\end{enumerate}
Then we say that the point $y$ {\it has $(x,r,\theta)$-almost parallels} $U(y, s),V(z, s)$.
We set $\ca E(x,r,\theta):=B^{X_0}(x, r)\setminus \ca D(x,r,\theta)$.

We also need the symbol  $\ca D(x,r,\e)$ for a constant $\e>0$ to denote the set of all
points $y\in B^{X_0}(x,r)$ satisfying the above (1), (2) and
\begin{enumerate}
 \item[(3')] $d_y^Y(z)/s \le \e$.
\end{enumerate}
We set $\ca E(x,r,\e):= B^{X_0}(x,r)\setminus \ca D(x,r,\e)$.
\end{defn}
\pmed
 %%%%%%%%%%%%%%%%%%%%
\begin{center}
\begin{tikzpicture}
[scale = 0.45]
\filldraw[fill=lightgray, opacity=.1] 
(-5,0.5) to [out=-15, in=195] (5,0.5)
to  (5,-1.5)  to [out=165, in=15] (-5,-1.5);

\draw [-, very thick] (-5,0.5) to [out=-15, in=195] (5,0.5);
\fill (0, -0.25) circle (2pt) node [above] {\small {$y$}};
\draw [-, very thick] (-5,-1.5) to [out=15, in=165] (5,-1.5);
\fill (0, -0.75) circle (2pt) node [below] {\small {$z$}};
\draw [-, thick] (0,-0.25) to (0,-0.75);
\fill (4, 0.5) circle (0pt) node [above] {\small {$U(y,s)$}};
\fill (4, -1.5) circle (0pt) node [below] {\small {$V(z,s)$}};

\fill (-5,-0.5) circle (0pt) node [left] {\small {$X$}};

\end{tikzpicture}
\end{center}
\vspace{-0.7cm}  
\begin{figure}[htbp]
  \centering
  \caption{}
  \label{fig:parallels}  
\end{figure}  
%%%%%%%%%%%%%%%%%%%%%%%%%%%%%
\psmall

%Note that under the assumption \eqref{eq:paSigma-never}, $x$ is contained in 
%$\ca E(x,r,\theta)$ and $\ca E(x,\e)$.

In the next two lemmas, we present basic properties of almost parallel domains.
\psmall
\begin{lem} \label{lem:calDC}
Let $y, z\in X_0$ be as in Definition \ref{defn:DC}. Then we have
\begin{enumerate}
\item $z\in {\rm int}\, Y$ and
${\rm rad}(\xi_z^+)=\pi\,;$
\item $\angle(\uparrow_z^y,v)=\pi/2$ for all $v\in\Sigma_z(X_0)$.
\end{enumerate}
\end{lem}
\begin{proof}
In view of Corollary \ref{cor:big-bdy},
the lemma is immediate since $\gamma_z^+$ and $zy$
form a minimal geodesic in $Y$.
\end{proof}

\psmall  

\begin{lem}\label{lem:almostperp+} 
For any $y\in\ca D(x,r,\theta)$ with 
$(x,r,\theta)$-almost parallels 
$U(y,s), V(z,s)$ as in Definition \ref{defn:DC}, the following hold: 
\begin{enumerate}
\item $\lim_{y\to x}\bigl(\sup_{v\in\Sigma_{y}(X_0)}  |\angle_{y}(\uparrow_{y}^{z}, v) -\pi/2|\bigr)=0\,;$
\item  Let $t_y$ be the supremum
of $t<s$ such that 
$$
  U(y,t)\subset  ({\rm int} X_0^1\setminus \ca C)    \cap {\rm int} X_0.
$$
Define $t_z$ similarly.  Then we have
\[
        \lim_{y\to x} |y,z|/\min\{ t_y,t_z\}=0. 
\]
\item $\lim_{y\to x} |y,z|/|y,\pa Y|=0$.
\end{enumerate}
\end{lem}
\begin{proof} (1)\, Although the basic idea is identical with that of Sublemma \ref{slem:perpatw}, we give the proof for readers'
convenience.
Suppose (1) does not hold. Then  we have a sequence $y_i\in \ca D(x, r,\theta)$ 
having $(x,r,\theta)$-almost parallels $U_i:=U(y_i,s_i), V_i:=V(z_i,s_i)$
 with $y_i\to x$ such that 
 there exists $v_i\in\Sigma_{y_i}(X_0)$ satisfying 
\beq\label{eq:angleyzv}
\text{ $|\angle_{y_i}(\uparrow_{y_i}^{z_i}, v_i) -\pi/2|\ge \alpha>0$ }
\eeq
for a uniform constant $\alpha$.
Let $\delta_i:=|y_i,z_i|$, and set 
$$
\hat X_i:=
      X\cup_{\eta_0} ([0,\delta_i]\times_\phi C_0)
\subset Y.
$$
Let us consider the convergence 
\begin{align} \label{eq:conv-1/delta-Yyi}
   \left(\frac{1}{\delta_i}Y, y_i\right)  \to (Y_\infty, y_\infty),
\end{align}
where we may assume that  $X$, 
$\hat X_i$ converge to nonnegatively  curved  noncompact Alexandrov spaces $X_\infty$, $\hat X_\infty$ under the above convergence. 
Let $\hat U_i$ and $\hat V_i$ be the neighborhoods
of $\pa \hat X_i$ corresponding to 
$U_i$ and $V_i$.
Let  $U_\infty$, $V_\infty$, $\hat U_\infty$, $\hat V_\infty$   be the 
limits of  $U_i$, $V_i$, $\hat U_i$, $\hat V_i$ 
under this convergence, respectively.
Let $\gamma_\infty$ be the limit of a geodesic $y_iz_i$
 in $Y$. By Lemma \ref{lem:inrad-collapse}, we have $\gamma_\infty\subset X_\infty$.
%%%
From $d(\hat U_i, \hat V_i)\ge 2\delta_i$, 
$\hat U_\infty$ and $\hat V_\infty$ are disjoint.
In view of $\lim_{i\to\infty} s_i/\delta_i=\infty$, 
 \cite{PtPt:extremal} implies that  $\hat U_\infty$ and $\hat V_\infty$ are extremal subsets of $\hat X_\infty$, which is contained in the Alexandrov bounary $\pa \hat X_\infty$.
%%%%
By Theorem \ref{thm:splitting},
$\hat X_\infty$ is isometric to $\hat U_\infty\times \hat I$ for an interval $\hat I$. 
This shows that $X_\infty$ is also isometric to $U_\infty\times  I$ for an interval $I\subset \hat I$. 
Therefore,   $\gamma_\infty$ must be  
 perpendicular to
$U_\infty\times \pa I$ at the limits
$y_\infty,z_\infty$ of $y_i,z_i$ respectively.
%%%%%%%%
This implies $\angle(\uparrow_{y_i}^{z_i},\dot\gamma_{y_i}^+(0))>\pi-o_i$.
Since $\angle(v_i,\xi_{y_i}^+)=\pi/2$, it is straightforward
to get $|\angle(\uparrow_{y_i}^{z_i},v_i)-\pi/2|<o_i$
(see \cite[Corollary 5.7]{BGP}). 
This is a contradiction.
\par

(2)\, Suppose there is a sequence
$y_i\in \ca D(x,r,\theta)$ converging to $x$
such that $t_{y_i}\le C|y_i,z_i|$ for a constant $C$.
Then there is a point 
$u_i\in U(y_i, 2t_{y_i})\cap (\ca S^1\cup \ca C
\cup X_0^2)$.
Let $u_\infty\in U_\infty$ be a limit of $u_i$
under the convergence \eqref{eq:conv-1/delta-Yyi}.
 From the above argument in (1),  for any  
$u_\infty\in U_\infty$, 
\[
  \angle(\xi_{u_\infty}^+,\dot\gamma_{u_\infty,v_\infty}(0))=\pi
\]
for some $v_\infty\in V_\infty$.
However this is impossible for the limit of $u_i$.
In the same way, we have a contradiction if $t_{z_i}\le C|y_i,z_i|$.

\par

(3)\,  Suppose that $|y_i, z_i|>c|y_i,\pa Y|>0$ for a sequence $y_i\in X_0$
converging to $x$ with a uniform constant $c$, where $z_i$ is chosen as in (1).  
Then under the convergence \eqref{eq:conv-1/delta-Yyi},  $(\pa Y)_\infty$ is nonempty.
Let $w_\infty$ be a nearest point of 
$\pa Y_\infty$ from $y_\infty$.
%%%
By the splitting $\hat X_\infty=\hat U_\infty\times \hat I$ in (1),   
there is a geodesic $\gamma_\infty$
containing a perpendicular at a point $u_\infty$
of $(X_0)_\infty$ with $w_\infty\in\gamma_\infty$.
Note that $\gamma_\infty\subset\pa Y_\infty$.
Then we can take $u_i\in X_0\cap \pa Y$ converging to $u_\infty$ under \eqref{eq:conv-1/delta-Yyi}. 
Since $u_i\in \pa X_0$, this is a contradiction
to $U(y_i,s_i)\subset {\rm int} X_0$.
\end{proof} 

%%%%%%%%%%%%%%%%%%%%%
As stated in the strategy for the proof of Theorem \ref{thm:never-happen}, the next lemma
shows that there are almost parallel domains
in any neighborhood of a point of 
$\ca S^2\cup\ca C$.

\pmed
\begin{lem} \label{lem:S2nbdW}  
For any $x\in X_0^1$ and $r>0$, we have 
\[
  B^{X_0}(x,r)\cap (\ca S^2\cup\ca C)
\subset \pa\ca D(x,r,\theta).
\]
Namely, for arbitrary $w\in B^{X_0}(x,r)\cap (\ca S^2\cup\ca C)$, $\theta=\theta(t)$ and $\e>0$, there exists a point 
$y\in\ca D(x,r,\theta)\cap B^{X_0}(w,\e)$.
\end{lem} 
\begin{proof} Let  $w\in B^{X_0}(x,r)\cap \ca S^2$, and $\{ q_1,q_2\}:=\eta_0^{-1}(w)$. 
Choose $0<\delta\ll\min \{ \e, |q_1,q_2|\}$ in such 
a way that  $U_i:=\eta_0(\tilde U_i)$\, $(i=1,2)$ 
are as in Sublemmas \ref{slem:eta-ball} and \ref{slem:eta-inter}, where $\tilde U_i:= \mathring{B}^{C_0}(q_i, \delta)$.
From $w\in\ca S^2$, we can take a point $w_0\in B^{X_0}(w,\delta/2)\cap (U_1\cup U_2\setminus X_0^2)$. 
Let $y_0$ be a nearest point of $U_1\cap U_2\subset X_0^2$ from $w_0$ with respect to the intrinsic metric of $X_0$. Notice that
$|w_0,y_0|=|w_0, X_0^2|$ by
Sublemma \ref{slem:eta-inter}.
Note also that 
\beq \label{eq:U12=tangent}
\Sigma_{y_0}(U_1)=\Sigma_{y_0}(U_2).
\eeq
Setting $s:=|w_0,y_0|^{X_0}$, let us 
take a shortest
path  $\gamma:[0,s]\to X_0$  from $y_0$ to $w_0$.
Assuming  $w_0\in U_1$, we have $\gamma\subset U_1$ and $\gamma$ is $U_1$-minimal.
Choose $s_i\in (0,s]$ converging to $0$, and set $y_i:=\gamma(s_i)$.
Take a nearest point $z_i$ of $U_2$ from $y_i$ with respect to $d^{Y}$. Since $\lim_{i\to\infty} \angle^Y y_iy_0z_i=0$ by \eqref{eq:U12=tangent}, we have 
$\lim_{i\to\infty} |y_i,z_i|^Y/s_i =0$.
Thus if $w\in {\rm int}_* X_0$, replacing $\delta$ by a constant 
$\ll \min \{ |q_i, \pa C_0|\,|\,i=1,2\}$, we can conclude that 
the balls $U(y_i, s_i/2)$, $U(z_i, s_i/2)$ and $z_i$  satisfy
the conditions in Definition \ref{defn:DC} for large $i$
to conclude $y_i\in\ca D(x, r,\theta)$.

Next suppose $w\in \pa_* X_0$.
Note that 
$\pa_* X_0\cap U_i=\eta_0(\pa C_0\cap \tilde U_i)$.
By Sublemma \ref{slem:eta-inter}, there is $\nu>0$
such that
%Fix  $\nu>0$ such that 
$B^{X_0}(w,\delta/2)\cap U_{i,\nu}\setminus X_0^2$ is nonempty, where 
$U_{i,\nu}:=\{ u\in U_i|\, |\pa_* X_0, u|\ge \nu\}$.
Choose $w_0\in B^{X_0}(w,\delta/2)\cap U_{1,\nu}\setminus X_0^2$, 
and let $y_0$ be a nearest point of 
$U_1\cap U_2$ from $w_0$ with respect to 
$d^{X_0}$. 
Here we assume $y_0\in \pa_* X_0$. The other case is
addressed similarly to the case $y_0\in {\rm int}_* X_0$.
Assuming $w_0\in U_1$,
fix a $U_1$-minimal geodesic $\gamma$ from $y_0$ to $w_0$. By Lemma \ref{lem:eta-2}, 
$\gamma$ has a definite direction everywhere, and we have
\[
\alpha:=\angle(\dot\gamma(0),\pa_* X_0)
   =\angle(\dot{\tilde\gamma}(0),\pa C_0) >0,
\]
where $\tilde\gamma$ is the lift of $\gamma$.
Note that the last inequality holds since 
$\pa C_0$ is an extremal subset of the Alexandrov space $C_0$.
For  $y_i=\gamma(s_i)$, 
and let $z_i \in U_2$ be a $d^Y$-nearest point of $U_2$
from $y_i$ as above.
Then in view of \eqref{eq:U12=tangent}
and $|y_i,z_i|_Y/s_i\ll \alpha$ for large $i$,
we see that both $U(y_i, \hat s_i)$ and $U(z_i, \hat s_i)$
are disjoint neighborhoods of ${\rm int}_* X_0$, 
where $\hat s_i=c(\alpha)s_i$ for small enough $c(\alpha)>0$.
Thus we have $y_i\in \ca D(x,r,\theta)$ for large $i$.

Finally consider the case $w\in\ca C\cap B^{X_0}(x,r)$.
Let $q:=\eta_0^{-1}(w)$ and take  $\tilde v\in{\rm int}\,\Sigma_q(C_0)$ such that 
$f_*(\tilde v)\neq\tilde v$.
Choose  geodesics $\tilde\gamma, \tilde\sigma:[0,s_0]\to C_0$
almost tangent to $\tilde v, f_*(\tilde v)$ respectively.
Set $\gamma:=\eta_0(\tilde\gamma)$
and $\sigma:=\eta_0(\tilde \sigma)$.
For a sequence $s_i\to 0$, let 
$y_i:=\gamma(s_i)$, 
$U_i=B^{X_0}(y_i,s_i/2)$
and $V_i=B^{X_0}(\sigma(s_i), s_i/2)$.
Let $z_i$ be a nearest point of $V_i$ from 
$y_i$. 
In a way similar to the above, we conclude
that $y_i\in  \ca D(x,r,\theta)$ for large $i$.
This completes the proof of Lemma \ref{lem:S2nbdW}.
\end{proof}
\pmed
Our next purpose is to show that 
$\ca C(x,r,\theta)$ occupies a large part of 
$B^{X_0}(x,r)$ (Lemma \ref{lem:CsubC}). To achieve this,  we define  cone-like regions in $X_0$ around $x$.
%%%%
 In what follows, we always assume 
\[   \text{
     $x\in X_0^1\cap\pa Y$\,\, and \, \eqref{eq:paSigma-never},}
\]
or equivalently,
\begin{align}\label{eq:Sigma<Sigma}
\Sigma_x(X_0)\subset \pa\Sigma_x(Y).
\end{align}

For $a>0$, let
\begin{align*}
     & \Sigma_p(C_0)_a:=\{ \xi\in\Sigma_p(C_0)\,|\,
        \angle(\xi, \pa\Sigma_p(C_0))>a\},\\
 & \Sigma_x(X_0)_a:=\{ \xi\in\Sigma_x(X_0)\,|\,
        \angle(\xi, \pa\Sigma_x(X_0))>a\}.
\end{align*}
By Theorem \ref{thm:X1-f}, we have
\beq \label{eq: cone-Cone}
d\eta_0(\pa\Sigma_p(C_0))=\pa \Sigma_x(X_0),\quad
d\eta_0(\Sigma_p(C_0)_a)=\Sigma_x(X_0)_a.
\eeq	
%%%%%
For $r>0$, 
we denote by $C(p;r,a)$ 
the set of points
$q\in \mathring{B}^{C_0}(p,r)\setminus\{ p\}$ satisfying  
\beq\label{defn:Cone}
   \angle(\Uparrow_p^q, \pa\Sigma_p(C_0))>a.
\eeq
%\begin{cases}
%\begin{aligned}
%&\angle(\Uparrow_p^q, \pa\Sigma_p(C_0)\cup
%\tilde{\ca F}_p)>a   \quad &\text{if $\dim \tilde{\ca F}_p=\dim \Sigma_p(C_0)-1$} \\
%&\angle(\Uparrow_p^q, \pa\Sigma_p(C_0))>a  \qquad & \text{if $\dim \tilde{\ca F}_p\le \dim \Sigma_p(C_0)-2$}.
%\end{aligned}
%\end{cases}
%\eeq

\begin{defn}\label{defn:cone}
We set 
\beq \label{eq:Conex}
     {\rm Cone}(x; r, a):=\eta_0(C(p;r,a)).
\eeq
If $\alpha=\alpha(t)$ is a positive increasing function defined on 
$(0,r]$ with $\lim_{t\to 0} \alpha(t)=0$, we define  
$C(p;r,\alpha)$ as the set of points
$q\in \mathring{B}^{C_0}(p,r)\setminus\{ p\}$ 
satisfying \eqref{defn:Cone} for
$a=\alpha(|p,q|)$, and define
$$
 {\rm Cone}(x; r, \alpha):=\eta_0(C(p;r,\alpha)).
$$
%Here for $x\in{\rm int} X_0$, we define 
%$$
%{\rm Cone}(x; r, a)={\rm Cone}(x; r,\alpha):= \mathring{B}^{X_0}(x,r)\setminus\{ x\}.
%$$
\end{defn}

\pmed
 %%%%%%%%%%%%%%%%%%%%
\begin{center}
\begin{tikzpicture}
[scale = 0.45]

 %opacity=.1

\filldraw[fill=lightgray,opacity=.1] 
%[fill=gray, opacity=.1] 
 (0,0) to (5.7, 2.5)
 to [out=-63, in=63] (5.7,-2.5)
to (0,0);

\filldraw[shift={(11,0)}][fill=lightgray,opacity=.1] 
%[fill=gray, opacity=.1] 
 (0,0) [out=35, in=195]     to (5.7, 2.5)
 to [out=-63, in=63] (5.7,-2.5)
[out=165, in=-35] to (0,0);

\draw [-, thick] (0,0) to  (5,3.5);
\draw [-, thick] (0,0) to  (5,-3.5);
\draw [-, thick] (5,-3.5)  [out=50, in=-50] to
(5,3.5) ;
\draw [-, thick] (0,0) to  (5.7,2.5);
\draw [-, thick] (0,0) to  (5.7,-2.5);
\fill (0, 0) circle (2pt) node [left] {\small {$x$}};
\fill (1, -2.5) circle (0pt) node [below] 
{\small {${\rm Cone}(x; r, a)$}};
\draw [-, thin] (2.3,1.0)  [out=100, in=-40] to
(2.05,1.45) ;
\fill (2.05,1.4)  circle (0pt) node [right] {\small {$a$}};

%%%%%%%%%%%%%%%%%%%%%%%%%%%%%%%%%%%
\draw[shift={(11,0)}] [-, thick] (0,0) to  (5,3.5);
\draw[shift={(11,0)}] [-, thick] (0,0) to  (5,-3.5);
\draw[shift={(11,0)}] [-, thick] (5,-3.5)  [out=50, in=-50] to
(5,3.5) ;
\draw[shift={(11,0)}] [-, thick] (0,0) to
[out=35, in=195]  (5.7,2.5);
\draw[shift={(11,0)}] [-, thick] (0,0) to  
[out=-35, in=165](5.7,-2.5);
\fill[shift={(11,0)}] (0, 0) circle (2pt) node [left] {\small {$x$}};
\fill[shift={(11,0)}] (1, -2.5) circle (0pt) node [below] 
{\small {${\rm Cone}(x; r, \alpha)$}};

\end{tikzpicture}
\end{center}
\vspace{-0.5cm}  
\begin{figure}[htbp]
  \centering
  \caption{}
  \label{fig:cones}  
\end{figure}  
%%%%%%%%%%%%%%%%%%%%%%%%%%%%%
\psmall
\begin{lem}\label{lem:cone-conn}  
${\rm Cone}(x; r, a)$ has the following properties$:$
\begin{enumerate}
\item ${\rm Cone}(x; r, a)$ coincides with the set of all
$y\in \mathring{B}^{X_0}(x,r)$ such that 
%$\angle(\tilde\Uparrow_x^y, \pa\Sigma_x(X_0))> \alpha(|x,y|)$.
\beqq\label{eq:Cone}
\angle(d\eta_0(\Uparrow_p^{\tilde y}), \pa\Sigma_x(X_0))>a
\eeqq
%where 
%$\tilde\Uparrow_x^y:=d\eta_0(\Uparrow_p^{\tilde y})$
 for some $\tilde y\in\eta_0^{-1}(y)$.
\item  Both ${\rm Cone}(x; r, a)$ and the interior
${\rm \mathring{C}one}(x; r, a)$ are connected$\,;$
\item For any $b>a$, if $r$ is small enough,
the interior ${\rm \mathring{C}one}(x; r, a)$
contains ${\rm Cone}(x; r, b)$.
\end{enumerate}
The statements corresponding to (1) and (2)
also hold for 
${\rm Cone}(x; r, \alpha)$.
\end{lem}
\begin{proof}
We show the conclusion for  ${\rm Cone}(x; r, a)$.
The case of  ${\rm Cone}(x; r, \alpha)$ can be similarly 
discussed, and hence omitted.

\n (1) immediately follows from \eqref{eq: cone-Cone}.

\n
(2) \, For arbitrary  $y_1,y_2\in {\rm Cone}(x; r, a)$,
take  $\tilde y_k$ with $\eta_0(\tilde y_k)=y_k$  $(k=1,2)$ satisfying the conclusion  of (1).
Choose 
any $C_0$-minimal geodesic $\tilde\gamma_k$ joining $p$ to $\tilde y_k$.
%
%, then 
%\[
%\angle(\dot{\tilde\gamma}_1(0), \dot{\tilde\gamma}_2(0))=
%\angle((\eta\circ\tilde\gamma_1)'(0),
%(\eta\circ\tilde\gamma_2)'(0)).
%\]   
Note that $\tilde\gamma_k\subset \tilde C(p;r,a)$. 
For small $t>0$,
let us consider a $C_0$-minimal geodesic 
$\tilde\sigma_t(s)$\,$(0\le s\le 1)$
joining $\tilde\gamma_1(t)$ and $\tilde\gamma_2(t)$.
Note that as $t\to 0$, the set 
$$
\tilde\Gamma_t:=\cup_{s\in[0,1]}\uparrow_x^{\tilde\sigma_t(s)}
$$ converges to 
a minimal geodesic $\tilde\xi(s)$ in $\Sigma_p(C_0)$
joining $\dot{\tilde\gamma}_1(0)$ to $\dot{\tilde\gamma}_2(0)$.
It follows from the concavity of 
$d(\pa\Sigma_p(C_0),\,\cdot\,)$
on 
$\Sigma_p(C_0)$ that  
$\tilde\xi\subset \Sigma_p(C_0)_a$
 (see \cite{Pr:alexII}).
This implies $\tilde\sigma_t\subset \tilde C(p;r,a)$
for small enough $t$, and therefore
the connectedness of ${\rm Cone}(x; r, a)$.
The connectedness of ${\rm \mathring{C}one}(x; r, a)$
is also similarly discussed.

\n
(3) can be shown by contradiction together with a limit argument.
\end{proof}

%
%In the proof of Lemma \ref{lem:almostperp+}, we have
%already proved the following.
%
%\begin{prop} \label{prop:limit-split}
%Let $y_i$ be a sequence in $\ca D(x,r,\theta)$ with $y_i\to x$, and 
%let $y_i$ have $(x,r,\theta)$-almost parallels
%$U(y_i,s_i), V(z_i,s_i)\subset {\rm int}\,X_0$.
%Consider the convergence
%\begin{align} \label{eq:rescaling-split}
%\left(\frac{1}{|y_i, z_i|}Y, y_i\right)  \to (Y_\infty, y_\infty).
%\end{align}
%Then the limit $(X_\infty,y_\infty)$ of $(X, y_i)$ under the above convergence
%is isometric to a product $L\times [0,1]$,
%where $L$ is the limit of $U(y_i, s_i)$ under the 
%above convergence.
%\end{prop}

\pmed\n
{\bf The first main step.} The following is one of the two main steps in the proof of Theorem \ref{thm:never-happen}.

\begin{lem} \label{lem:CsubC}
There are $r>0$ and positive functions 
$\alpha=\alpha(t)$, $\beta=\beta(t)$ and $\theta=\theta(t)$ defined on $(0,r]$ with $\alpha(t)<\beta(t)$ and 
$\lim_{t\to 0}\beta(t)=\lim_{t\to 0}\theta(t)=0$ 
such that 
\begin{gather*}
     {\rm Cone}(x; r, \alpha)\subset \ca E(x,r,\theta)\cap
   ({\rm int}\,X_0^1\setminus \ca C), \\
 {\rm Cone}(x; r, \beta) \subset {\rm \mathring{C}one}(x; r, \alpha).
\end{gather*}
\end{lem}

%As an immediate consequence of Lemma \ref{lem:CsubC},
%Theorem \ref{thm:never-happen} follows in the case of $x\in{\rm int} X_0$.
\pmed
%%%%%%%%%  %%%%%%
\n 
{\bf Rough idea of the proof Lemma \ref{lem:CsubC}.}\,
%
%A primitive idea of the proof of Lemma \ref{lem:CsubC} is as follows.
It is proved by contradiction. Here we only observe the situation that 
a sequence $y_i\in {\rm Cone}(x; r, c)\cap  X_0^2$
(see \eqref{eq:three-cond} for the general case) converges to $x$ with a positive
constant $c$. Let $\tilde y_i^k\in C_0^2$ \, $(k=1,2)$
be the lifts of $x$. 
Passing to a subsequence, we may assume 
\[
      \uparrow_x^{y_i}\to v,\quad
      \uparrow_p^{\tilde y_i^k}\to \tilde v^k.
\]
Obviously, we have $d\eta_0(\tilde v^k)=v$.
If  
$|\tilde y_i^1,\tilde y_i^2|/|x,y_i|\ge c>0$
for a constant $c$ independent of $i$,
we would have $\tilde v^1\neq\tilde v^2$ 
causing a contradiction to $f_*={\rm id}$.
In the general case, we consider a nearest point 
$w_i\in \pa Y$ from $y_i$. 
Making use of Lemma 
\ref{lem:frequent-later}, we obtain 
$v\in \pa \Sigma_x(X_0)$, a contradiction
to $y_i\in {\rm Cone}(x; r, c)$.

\begin{proof}[Proof of Lemma \ref{lem:CsubC}]
We shall show that for any fixed constant
$b>c>0$, there exist $r>0$ and $\e>0$ satisfying
\beq \label{eq:ConesubsetE}
\begin{cases}
\begin{aligned}
 &{\rm Cone}(x; r, c)\subset \ca E(x,r,\e)\cap
   ({\rm int}\,X_0^1\setminus \ca C),\\
& {\rm Cone}(x; r, b) \subset {\rm \mathring{C}one}(x; r, c).
\end{aligned}
\end{cases}
\eeq
Then we obtain the conclusion of the lemma as follows.
Choose decreasing sequences
$b_i>c_i$ converging to $0$. Applying
\eqref{eq:ConesubsetE} for $b_i>c_i$, we  choose  decreasing sequences
$r_i$ and $\e_i$ converging to $0$
such that 
\begin{align*}
& {\rm Cone}(x; r_i, c_i)\subset
\ca E(x,r_i,\e_i)\cap
({\rm int} \,X_0^1\setminus \ca C),\\
&{\rm Cone}(x; r_i, b_i) \subset {\rm \mathring{C}one}(x; r_i, c_i).
\end{align*}
Let $r=r_1$, and define step functions
$\alpha_*$, $\theta_*$ defined on $(0,r]$ by
$\alpha_*(t)=c_i$ and $\theta_*(t)=\e_i$
on $(r_{i+1}, r_i]$. It is now immediate to obtain 
continuous functions $\alpha(t)$, $\beta(t)$ and $\theta(t)$ on $(0,r]$ satisfying the conclusion of the lemma.

By Lemma \ref{lem:cone-conn}, we certaily have the second inclusion in
\eqref{eq:ConesubsetE}.
We show the first inclusion in \eqref{eq:ConesubsetE}
by contradiction.
Suppose it does not hold. Then for some $c>0$, we have 
sequences $y_i\in X_0$ and $\e_i>0$
satisfying
\beq \label{eq:three-cond}
\begin{cases}
\begin{aligned}
   &y_i\to x,\\
    & y_i\in {\rm Cone}(x;r_0, c),\\
& y_i\in \ca D(x, r_0,\e_i)\cup X_0^2\cup \ca S^1\cup \ca C, \quad \e_i\to 0,
\end{aligned}  \end{cases}
\eeq
where $r_0>0$ is a constant.
We may assume $\uparrow_x^{y_i}$ converges to a direction $v\in {\rm int}\,\Sigma_x(X_0)$.

First we consider 

\pmed\n
{\bf Case A).}\, $y_i\in \ca D(x,r_0,\e_i)$.
\pmed

Take $s_i>0$, $z_i$  and   $U(y_i,s_i)$,  $U(z_i, s_i)$ 
 as in Definition \ref{defn:DC} (1),(2),(3') such that 
the restriction of $d_{y_i}^Y$ to $U(z_i, s_i)$ has a positive minimum at $z_i\in V_i$.
Let $w_i$ be a nearest point of $\pa Y$ from $z_i$.

\begin{slem}\label{slem:angleyxw=0}
\beqq
\lim_{i\to\infty}\angle z_i x w_i =0.
\eeqq
\end{slem}
\begin{proof}
Suppose there is a subsequence $\{ j\}\subset \{i\}$ such that 
$\angle z_j x w_j \ge\theta>0$ for
a uniform constant $\theta$.
Consider the rescaling limit
\[
  \left(\frac{1}{|x,z_j|} Y, z_j\right)  \to (T_x(Y), z_\infty),
\]
where we may assume that $w_j$ converges 
to an element $w_\infty\in T_x(Y)$.
From the assumption, we have 
$w_\infty\neq z_\infty$.
On the other hand, it follows from the choice of $w_j$ that 
$w_\infty$ is a nearest point of 
$T_x(\pa Y)=\pa T_x(Y)$ from $z_\infty$.
However, from \eqref{eq:Sigma<Sigma}, we
get  $z_\infty=y_\infty\in T_x(X_0)\subset \pa T_x(Y)$. This is a contradiction.
\end{proof}

Since $\lim_{i\to\infty}\angle y_i x z_i=0$,
Sublemma \ref{slem:angleyxw=0} implies 
$\lim_{i\to\infty}\angle y_i x w_i=0$.
%%%%%%%

Now consider the convergence 
\begin{align} \label{eq:rescaling-zw}
   \left(\frac{1}{|z_i, w_i|}Y, z_i\right)  \to (Y_\infty, z_\infty).
\end{align}
By Lemma \ref{lem:almostperp+}, we have 
$\lim_{i\to\infty} \frac{|z_i, y_i|}{|z_i,w_i|}=0$.
  It follows from Lemma 
\ref{lem:almost-parallel} that the convergence 
\eqref{eq:rescaling-zw} is caused by an inradius collapse.
Thus we have $w_\infty\in (X_0)_\infty$.
Since $\lim_{i\to\infty}\angle(\uparrow_x^{w_i},v)=0$,
from Lemma \ref{lem:frequent-later}  we 
have    
%$\xi_x^+ v\subset \pa\Sigma_x(Y)$ and 
$v\in \pa\Sigma_x(X_0)$. This is a contradiction
to \eqref{eq:three-cond}.

\psmall
Next we consider 

\pmed\n
{\bf Case B).}\, $y_i\in \ca S^2\cup\ca C$.
\pmed

By Lemma \ref{lem:S2nbdW},
there exists a point $y_i'\in \ca D(x,r,\theta)$
in any small neighborhood of $y_i$.
Thus this case can be reduced to Case A),
and causes a contradiction, too.
Thus we only have to consider the following case.

\pmed\n 
{\bf Case C).}\, $y_i\in \ca S^1\cup {\rm int}\,X_0^2$.
\pmed
If $y_i\in\ca S^1$, then any neighborhood
of $y_i$ contains a point $y_i'\in X_0^2$.
If $y_i'\in\ca S^2$, we have a contradiction 
by Case B).
If $y_i'\in{\rm int}\,X_0^2$, we can reduce to 
the case $y_i\in {\rm int}\,X_0^2$.

Therefore in what follows, we consider
the case $y_i\in {\rm int}\,X_0^2$.

\begin{slem} \label{slem:ConesubsetX02}
There exists $r>0$ satisfying 
\[
  {\rm \mathring{C}one}(x; r, c)
    \subset {\rm int} X_0^2.
\]
In particular,  ${\rm \mathring{C}one}(x; r, c)$
is open in $X$.
\end{slem}
\begin{proof} 
Fix any $0<c_1<c_0<c$ and small $r>0$ such that
\[
{\rm Cone}(x; r, c)\subset
{\rm \mathring{C}one}(x; r, c_0) \subset 
 {\rm Cone}(x; r, c_0)\subset
{\rm \mathring{C}one}(x; r, c_1).
\]
In view of Case B), we may assume that 
${\rm Cone}(x; r, c_1)$
does not meet $\ca S^2$.
Take large $i$ such that 
\[
y_i\in {\rm int}X_0^2 \cap {\rm \mathring{C}one}(x;r,c_0).
\] 
Let $Q_i$ denote the intersection of 
${\rm \mathring{C}one}(x; r, c_0)$ and 
the component of ${\rm int}X_0^2$ containing $y_i$.
It suffices to show 
\[
{\rm \mathring{C}one}(x; r, c_0)=Q_i.
\]
Suppose this does not hold, and  take a point 
$z\in {\rm \mathring{C}one}(x; r, c_0)\setminus Q_i$. 
Replacing $z$ if necessary, we may assume
$z\in {\rm \mathring{C}one}(x; r, c_0)\setminus \bar Q_i$. 
Let $w$ be a point of $\bar{Q_i}$ nearest from 
$z$. 
Take $\e>0$ such that 
$B^{X_0}(w,\e)\subset  {\rm \mathring{C}one}(x; r, c_1)$.
One can choose two points $z',w'$ so close to $w$  that 
\[
  z'\in {\rm \mathring{C}one}(x; r, c_1)\setminus \bar Q_i,  \quad 
w'\in Q_i, \quad
\gamma:=\gamma^{X_0}_{z',w'}\subset B^{X_0}(w,\e).
\]
Let $u$ be a point of $\gamma\setminus\{ z',w'\}$ with $u\in \pa Q_i$.
Since $u\in \pa({\rm int}X_0^2)$, 
Lemma \ref{lem:non-extend} shows 
$u\in\ca S^2$. This is a contradiction.

The latter immediately follows from Lemma \ref{lem:loc-inradX2S1}.
\end{proof}

By Sublemma \ref{slem:ConesubsetX02},
we can consider 
${\rm \mathring{C}one}(x; r, c)$
as a local Alexandrov space.
Let $\gamma$ be an $X_0$-minimal 
geodesic joining $x$ to $y_i$ for a fixed large $i$, which is also an admissible curve by
Sublemma \ref{slem:ConesubsetX02}
and satisfies 
\[
\gamma_i\setminus \{ x\}\subset 
{\rm \mathring{C}one}(x; r, c), \quad
\dot\gamma(0)\in {\rm int} \Sigma_x(X_0).
\]
Since $\gamma_i\setminus \{ x\}\subset {\rm int} X_0^2$,  we have 
two lifts
$\tilde\gamma_1$ and $\tilde\gamma_2$
of $\gamma$ with $\tilde\gamma_k(0)=p$
\,$(k=1,2)$.
This implies that there are two mimimal geodesics in $\Sigma_x(Y)$ joining 
$\xi_x^+$ to $\dot\gamma(0)$.
This contradicts \eqref{eq:Sigma<Sigma},
and completes the proof of Lemma \ref{lem:CsubC}.
\end{proof}

\pmed

Let $\alpha=\alpha(t)$, $\beta=\beta(t)$ and $\theta=\theta(t)$ be as in Lemma \ref{lem:CsubC}.
\pmed
We denote by $\ca E_0(x, r,\theta)$ the connected component of the interior of $\ca E(x,r,\theta)$ that contains ${\rm \mathring{C}one}(x; r, \alpha)$.
 
%%%%%%%%%%%%%%%%%%%%
\begin{lem}  \label{lem:CsubsetX01}
We have 
$\ca E_0(x, r,\theta)\subset {\rm int}\,X_0^1
\setminus \ca C$. 
\end{lem}
\begin{proof}
By Lemma \ref{lem:S2nbdW}, 
$\ca E_0(x, r,\theta)$ does not meet $\ca S^2$ nor $\ca C$.
We show that 
$\ca E_0'(x, r,\theta):=\ca E_0(x, r,\theta)\setminus \ca S^1$ 
is connected.

We may assume that  
$\ca E_0(x, r,\theta)$ meets $\ca S^1$.
 For a point $z\in \ca E_0(x, r,\theta)\cap\ca S^1$,
choose an open neighborhood $U$ of $z$ in  
$\ca E_0(x, r,\theta)$.
Then  $V:=U\cap{\rm int}\,X_0^2$ is nonempty.
If $U\cap{\rm int}\,X_0^1$ is nonempty,
using Lemma \ref{lem:non-extend}, we would find an element of $\ca S^2$ in $U$, which is a contradiction.
Thus we have 
$$
    U=U\cap ({\rm int}\,X_0^2\cup\ca S^1).
$$ 
 Namely $U$ is a part of local inradius collapse (Lemma \ref{lem:loc-inradX2S1}).
Using \cite{Per} and \cite{Kap}, 
%we choose a convex neighborhood $U$ properly, 
we may assume that $U$ is convex.
Note that $U$ is an (incomplete) Alexandrov space with boundary
 $U\cap \pa X_0$.
It follows from \cite[Lemma 4.28]{YZ:inrdius} 
%together with the fact $X_0^2\cap\pa X_0\subset \pa_*X_0$ (see Lemma \ref{lem:eta'})
%Sublemma \ref{slem:eta-ball}) 
that 
$U\setminus \ca S^1
=U\cap {\rm int}\,X_0^2$
is also convex.

For any $w\in \ca E_0(x, r,\theta)\cap {\rm int} X_0^1$,
obviously we can take a connected neighborhood $V$ of $w$ in 
$\ca E_0(x, r,\theta)\cap {\rm int} X_0^1$.
%Lemma \ref{lem:eta'} also shows $V\setminus \pa X_0\subset \pa_*X_0$,
%and hence $V\setminus \pa X_0$ is connected.
Therefore the connectedness of $\ca E_0(x, r,\theta)$ yields that $\ca E_0'(x, r,\theta)=\ca E_0(x, r,\theta)\setminus \ca S^1$ 
is connected.

Now   
$\ca E_0'(x, r,\theta)$ is the union of  the two 
open subsets
$\ca E_0'(x, r,\theta)\cap{\rm int}\,X_0^1$ 
and $\ca E_0'(x, r,\theta)\cap{\rm int}\,X_0^2$.
If $\ca E_0'(x, r,\theta)\subset {\rm int}\,X_0^2$
(this happens if $\ca E_0(x, r,\theta)$ meets $\ca S^1$
as above), we have a contradiction  by Lemma \ref{lem:CsubC}.
%in a way similar to 
%the discussion after Definition \ref{defn:Theta-cone}.
Thus we have $\ca E_0(x, r,\theta)\subset {\rm int}\,X_0^1
\setminus \ca C$. 
This completes the proof.
\end{proof}
%%%%%%%%%%%%%%%%%%%

Let $\pa \ca E_0(x, r,\theta)$ be the topological boundary of
$\ca E_0(x, r,\theta)$ in $X_0$.
%%%%%
\pmed

Here we summarize some notations concerning almost parallels defined in this section so far.

\begin{table}[h]
\begin{center}
\begin{tikzpicture}[auto]
\fill (-1, 0) circle (0pt) node [right] {$\ca D(x,r,\theta) : \text{the set of points having almost parallels}$};
\fill (-1, -0.8) circle (0pt) node [right]{$\ca E(x,r,\theta)=B^{X_0}(x,r)\setminus\ca D(x,r,\theta)
$};
\fill (-1, -1.6) circle (0pt) node [right]{$\ca E_0(x,r,\theta) : \text{the component of $\mathring{\ca E}(x,r,\theta)
\supset {\rm \mathring{C}one}(x;r,\alpha)$}$};

\draw[thick] (-1.2,0.5) rectangle (9.2,-2.1);
\end{tikzpicture}
\caption{\small{Sets concerning almost parallels}}
\end{center}
\end{table}

\pmed\n
{\bf Second main step.}\,
\par
Here we prove $\ca E(x, r,\theta)=
B^{X_0}(x,r)$ yielding Theorem \ref{thm:never-happen}.

 In what follows, we need the following
notations.

\begin{defn}\label{defn:Theta-cone}
For $u,v\in X_0^1$, $\Lambda\subset \Sigma_u(X_0)$   and 
$a,\e>0$, let 
\begin{align*}
  &\Theta_u(\Lambda, a,\e) := \{ y\in X_0\setminus \{ u\}\,|\,
|u,y|<a, \,\angle( 
d\eta_0(\Uparrow_u^{\tilde y}),\Lambda)<\e\\
&\hspace{7cm} \text{for some $\tilde y\in \eta_0^{-1}(y)$}\}, \\
&\Theta(u,v):=\Theta_{u}(\Uparrow_{u}^{v},|u,v|,\pi/4).
\end{align*}
\end{defn}  

The following sublemma is easily verified.

\begin{slem}\label{slem:Theta(a,c)}
For any $0<\e_1<\e$, if $0<a_1<a$ is small enough, we have 
\[
\Theta_u(\Lambda;a_1,\e_1)\subset \mathring{\Theta}_u(\lambda;a,\e).
\]     
\end{slem}
\pmed
The following is crucial in
the second main step towards the proof of 
$\ca E(x,r,\theta)=B^{X_0}(x,r)$.

\begin{prop}\label{prop:paC-subset-paX}
There exists an  $r>0$ such that 
\[
     \pa \ca E_0(x,r,\theta) \cap \mathring{B}^{X_0}(x,r)\subset \pa_* X_0.
\]
\end{prop}

\psmall
\begin{proof}[Proof of Theorem \ref{thm:never-happen} assuming  Proposition \ref{prop:paC-subset-paX}]
Let $r$ be as in Proposition \ref{prop:paC-subset-paX}. 
We first remark that 
\beq\label{eq:h(C)subsetSigma}
  \mathring{B}^{X_0}(x,r)\setminus\pa_*X_0
     \subset \ca E_0(x,r,\theta).
\eeq
Actually for any $z\in \mathring{B}^{X_0}(x,r)\setminus\pa_*X_0$, choose a lift 
$\tilde z\in C_0\setminus \pa C_0$ of $z$ and set $s=|p,\tilde z|_{C_0}$.
By Perelman's topological stability theorem(\cite{Pr:alexII}, \cite{Per}), taking small
enough $r$ if necessary,
we may assume that 
there is a homeomorphism
\[
 h: (S^{C_0}(p,s), S^{C_0}(p,s)\cap\pa C_0)\to 
     (\Sigma_p(C_0), \pa\Sigma_p(C_0)).
\]
Let $\tilde\xi$  be the farthest point of 
$\Sigma_p(C_0)$ from $\pa\Sigma_p(C_0)$.
Take a curve $c$ joining  $h(\tilde z)$ to $\tilde\xi$ in $\Sigma_p(C_0)\setminus\pa\Sigma_p(C_0)$.
Note that $\eta\circ h^{-1}(\tilde\xi)\in
{\rm Cone}(x;r,\alpha)\subset \ca E_0(x,r,\theta)$.
Therefore if $z\notin\ca E_0(x,r,\theta)$,
then there would exist a point  on $\eta\circ h^{-1}\circ c$ meeting $\pa\ca E_0(x,r,\theta)$.
This contradicts Proposition \ref{prop:paC-subset-paX} since $\eta\circ h^{-1}\circ c\subset {\rm int}_* X_0$.
\psmall

We show that $r$ satisfies the conclusion of Theorem \ref{thm:never-happen}.
Suppose this does not hold.
Then there is a point 
$y\in\mathring{B}^{X_0}(x,r)\setminus({\rm int} X_0^1\setminus\ca C)=\mathring{B}^{X_0}(x,r)\cap (\ca C\cup\ca S^1\cup X_0^2)$.
If $y\in\ca S^1$, then there is a point of $X_0^2$ arbitrary
close to $y$.
Therefore from the beginning, we may assume
$y\in X_0^2\cup\ca C$.
Lemma \ref{lem:CsubsetX01} and  \eqref{eq:h(C)subsetSigma}
imply   
\[
y\in  \pa\ca E_0(x,r,\theta)\cap
 (\ca S^2\cup \ca C)\subset \pa_*X_0.
\]
Choose $v\in {\rm int}\,\Sigma_y(X_0)$ such that
for a lift $q\in C_0$ of $y$ and $\tilde v\in{\rm int}\Sigma_q(C_0)$ with $d\eta_0(\tilde v)=v$, there is a 
$C_0$-geodesic $\tilde\gamma:[0,\delta_0]\to C_0$ in the direction $\tilde v$.
%%%%%
Choose $a>0$ and $\e>0$ such that 
$\eta_0^{-1}(\Theta_y(v,a,\e))\subset
 C_0\setminus \pa C_0$.
Then obviously we have 
\beq\label{eq:Theta-int*}
\Theta_y(v,a,\e)\subset {\rm int}_* X_0.
\eeq
Lemma \ref{lem:CsubsetX01} and \eqref{eq:h(C)subsetSigma} again imply
\beq\label{eq:ThetaE0}
  \Theta_y(v,a,\e)\subset\ca E_0(x,r,\theta)
    \subset {\rm int} X_0^1\setminus\ca C.
\eeq
Since $\pa X_0\setminus\pa_* X_0\subset
\ca S^1\cup\ca C$ by Lemma \ref{lem:criteria-int-pa}, it follows from 
\eqref{eq:Theta-int*} and \eqref{eq:ThetaE0}
that
\[
\Theta_y(v,a,\e)\subset{\rm int} X_0.
\]
      
Let $\gamma:=\eta_0\circ\tilde\gamma$.
By  Proposition \ref{prop:paC-subset-paX},
$\gamma$ is the $X_0$-geodesic in the direction $v$ satisfying 
 \beq\label{eq:gammaintX0}
      \gamma(0,\delta)\subset
         \Theta_y(v,a,\e)\subset {\rm int}X_0
\eeq
for any small enough $0<\delta<\min\{ \delta_0,\e\}$.

\pmed
 %%%%%%%%%%%%%%%%%%%%
\begin{center}
\begin{tikzpicture}
[scale = 0.45]

\draw [-, thick] (0,0) to  (6, 3);
\draw [-, thick] (0,0) to  (6,-3);
\draw [-, thick] (6, -3)  [out=50, in=-50] to
(6, 3) ;
\fill (0, 0) circle (2pt) node [left] {\small {$x$}};
\fill (3, -1.5) circle (2pt) node [below] {\small {$y$}};
\fill (3.8,0.5) circle (0pt) node [above] {\small {$\gamma$}};
\draw [-, thin]  (3,-1.5) to [out=45, in=-90] (3.8,0.5);
\draw [dotted, thick]  (3,-1.5) to [out=38, in=250] (4,-0.2);

\fill (4.8,-2.5) circle (0pt) node [below] {\scriptsize{$\pa_*X_0$}};
\fill (7.5,0) circle (0pt) node [right] {\small{$B^{X_0}(x,r)$}};
\filldraw[fill=gray,opacity=.1] 
 (3,-1.5) to  [out=45, in=-90] (3.8,0.5)
 to (4,-0.2)
to   [out=250, in=38] (3,-1.5) ; 

\end{tikzpicture}
\end{center}
\vspace{-0.5cm}  
\begin{figure}[htbp]
  \centering
  \caption{}
  \label{fig:cones}  
\end{figure}  

On the other hand, since  $y\in \ca S^2\cup\ca C$,
as in the proof of Lemma \ref{lem:S2nbdW},
we have $\gamma(\e)\in \ca D(x,r,\theta)$
for small enough $\e>0$, which contradicts 
\eqref{eq:ThetaE0} and \eqref{eq:gammaintX0}.
This completes the proof  of Theorem \ref{thm:never-happen}.  
\end{proof}

\psmall

Since the topological boundary $\pa\ca E_0(x, r,\theta)$ 
could be quite wild, we consider a subset
of $\pa\ca E_0(x, r,\theta)$ that is easier to handle.

Let $\pa_* \ca E_0(x, r,\theta)$ denote the set of points 
$u\in\pa\ca E_0(x, r,\theta)$ such that 
$|b,u|_{X_0^{\rm int}}=|b,\pa\ca E_0(x, r,\theta)|_{X_0^{\rm int}}$ for some 
$b\in\ca E_0(x, r,\theta)$.
In this case, we have 
\begin{align} \label{eq:B(x,r,theta)-cond}
 \Theta(u,b)\subset \ca E_0(x, r,\theta).
\end{align}
Clearly, $\pa_* \ca E_0(x, r,\theta)$
is dense in $\pa\ca E_0(x, r,\theta)$.
\pmed
\n  
{\bf Rough outline of the proof of Proposition \ref{prop:paC-subset-paX}.}\,
The proof is done by contradiction.
Then we have a sequence
$u_m$ in $\pa_* \ca E_0(x, r,\theta)\setminus
\pa_* X_0$ converging to $x$.
It is verified that $u_m\in X_0^1$ and 
	$f_*={\rm id}$ on $\Sigma_{\tilde u_m}(C_0)$ for the lift $\tilde u_m$ of $u_m$ (Lemma \ref{lem:f=id}).
Using the presence of almost parallel domains
arbitrary close to $u_m$,
we show that $u_m\in{\rm int} \,Y$ and 
${\rm rad}(\xi^+_{u_m})>\pi/2+c$ 
for large enough $m$, where $c>0$ is a uniform constant (Sublemmas \ref{slem:z=intY} and \ref{slem:alpha>pi/2}).
Again this yields a contradiction from the
presence of almost parallel domains
arbitrary close to $u_m$.

We begin with  

\begin{lem} \label{lem:mayassumeX2}
$$
\pa_*\ca E_0(x, r,\theta)\subset X_0^1.
$$
\end{lem}
\begin{proof}

Suppose there is $u\in \pa_*\ca E_0(x, r,\theta)\cap X_0^2$ and  choose 
$b\in\ca E_0(x, r,\theta)$ for $u$ as in 
\eqref{eq:B(x,r,theta)-cond}.
Let $\gamma:[0,1]\to X_0$ be an $X_0$-minimal geodesic from $u$ to $b$.
Lemma \ref{lem:CsubsetX01} and 
\eqref{eq:B(x,r,theta)-cond} show $u\in\ca S^2$.
If $u\in{\rm int}_*X_0$, then by the proof of Lemma \ref{lem:S2nbdW},
there is a small neighborhood $U$ of $\gamma(\delta)$
such that $U\subset \ca D(x, r,\theta)$ for any small $\delta>0$.
This is a contradiction
since $U\subset \ca E(x, r,\theta)$ if 
$U$ is small enough.
If $u\in\pa_*X_0$, then we take a point $b'$ 
such that 
\[
\Theta_{u}(\Uparrow_{u}^{b'},|u,b'|,\pi/10)\subset \Theta(u,b)\cap {\rm int}_*X_0.
\]
For an $X_0$-geodesic $\sigma$ joining 
$u$ to $b'$, one can find $\delta>0$ and a small neighborhood $U$ of $\sigma(\delta)$
such that $U\subset \ca D(x, r,\theta)$
as before. Thus we have a contradiction
in this case, too.
\end{proof}

\begin{lem}  \label{lem:f=id}
For any $u\in \pa_*\ca E(x, r,\theta)$,
take $\tilde u\in C_0$ with $\eta_0(\tilde u)=u$.
Then $f_*$ is the identity on $\Sigma_{\tilde u}(C_0)$.

In particular, $u\in\pa X_0$ if and only if $\tilde u\in\pa C_0$, and hence $u\in\pa_* X_0$ in this case.
\end{lem}
\begin{proof}
Let 
$\Omega:=\{ \tilde\xi\in\Sigma_{\tilde u}(C_0)|\angle(\tilde\xi,\Uparrow_{\tilde u}^{\tilde b})\le\pi/4\}$, 
where $\tilde b=\eta_0^{-1}(b)$.
In view of \eqref{eq:B(x,r,theta)-cond},
Lemma \ref{lem:CsubsetX01} yields that
$f_*$ is the identity on $\Omega$, and hence is the identity on $\Sigma_{\tilde u}(C_0)$.
\end{proof}

\begin{proof} [Proof of Proposition \ref{prop:paC-subset-paX}]
We proceed by contradiction. Suppose that $\pa \ca E_0(x, r,\theta)$ meets 
${\rm int}_* X_0$ for any $r>0$, and take 
sequences $u_m\in\pa_* \ca E_0(x, r,\theta)\cap {\rm int}_* X_0$
and  $b_m\in \ca E_0(x, r,\theta)$ converging to $x$ such that $|b_m,u_m|=|b_m,\pa\ca E_0(x, r,\theta)|$.
Choose $\e_m>0$ such that

\beq\label{eq:zn-en}
B^{X_0}(u_m, \e_m)\subset {\rm int}_* X_0.
\eeq
Take a sequence $\{ y_{mi}\}_{i=1}^\infty$ in $\ca D(x,r,\theta)$ converging to $u_m$,
and
let $y_{mi}$ have $(x,r,\theta)$-almost parallels
$U_{mi}:=U(y_{mi}, s_i)$ and  $V_{mi}:=U(z_{mi},s_i)$.
Since $\lim_{i\to\infty} s_i=0$, from Definition \ref{defn:DC}(3), we have 
$\lim_{i\to\infty}|y_{mi},z_{mi}|=0$.

\begin{slem} \label{slem:z=intY}
We may assume that $u_m\in{\rm int}\,Y$ for large enough $m$.
\end{slem}
\begin{proof}
Suppose that $u_m\in\pa Y$ for any large $m$.
Take a nearest point $w_{mi}$ of $\pa Y$ from $z_{mi}$.
Since $u_m\in {\rm int}_* X_0$, we have 
$\Sigma_{u_m}(X_0)\subset\pa\Sigma_{u_m}(Y)$ 
from Lemma \ref{lem:criteria-0}.
Take  large enough $i=i_m$ such that 
\beq\label{eq:ynun-en}
y_{mi_m}, z_{mi_m}\in B(u_m,\e_m^2), \quad w_{mi}\in B(u_m,2\e_m^2).
\eeq
Set 
$y_m:=y_{mi_m}$, $z_m:=u_{mi_m}$, $w_m:=w_{mi_m}$ for simplicity.
We proceed as in Case A) of the proof of Lemma \ref{lem:CsubC}.
Namely  considering  the convergence
\begin{align} \label{eq:rescaling-zwY}
   \left(\frac{1}{|z_m, w_m|}Y, w_m\right)  \to (Y_\infty, w_\infty),
\end{align}
we obtain the following:
\begin{itemize}
\item  \eqref{eq:rescaling-zwY} is caused by an inradius collapse$\,;$
\item $w_\infty\in (X_0)_\infty$.
\end{itemize}
Since the argument is the same as before, we omit the detail.
Thus we have $w_\infty\in (X_0)_\infty\cap \pa Y_\infty$
and  $\gamma_{w_\infty}^+\subset \pa Y_\infty$.
%%%%%%%%

\pmed
%%%%%%%%%%%%%%%%%%%%
\begin{center}
\begin{tikzpicture}
[scale = 0.5]

\filldraw [fill=lightgray, opacity=.1] 
 (0,0) to  (1.42,1.42)
to [out=135, in=0]  (0,2)
to [out=180, in=45] (-1.42,1.42)
to  (0,0);

\filldraw [fill=lightgray, opacity=.1] 
(-0.2,-0.2) to[out=-80, in=90]   (-0.1,-0.7)
to [out=-90, in=80]     (-0.2,-1.2) 
to
(0.2,-1.2) to[out=100, in=-90]   (0.1,-0.7)
 to[out=-100, in=90]  (0.2,-0.2);

\draw [-, very thick] (-8.5,-2) to  (8.5,-2);
\draw [ -, thick] (-5.5,-2) to[out=0, in=180]   (0,0);
\draw [ -, thick] (5.5,-2) to[out=180, in=0]   (0,0);
\draw [ -, thick] (5.5,-2) to[out=0, in=180]   (6.75,-1.7);
\draw [ -, thick] (6.75,-1.7) to[out=0, in=180]   (8,-2);
\draw [ -, thick] (-5.5,-2) to[out=180, in=0]   (-6.75,-1.5);
\draw [ -, thick] (-6.75,-1.5) to[out=180, in=0]   (-8,-2);

\fill (-6.75,-1.5)  circle (2pt) node [above] {{\tiny $u_{m+1}$}};

\draw [-, thick] (0,0) to  (0,2);
\draw [-] (0,0) to  (1.42,1.42);
\draw [-] (0,0) to  (-1.42,1.42);
\draw [dotted, thick] (0,2) to[out=0, in=135]    (1.42,1.42);
\draw [dotted, thick] (0,2) to[out=180, in=45]    (-1.42,1.42);

\draw [ -] (-0.2,-0.2) to[out=-80, in=90]   (-0.1,-0.7);
\draw [ -] (-0.2,-1.2) to[out=80, in=-90]   (-0.1,-0.7);
%%%
\draw [ -] (0.2,-0.2) to[out=-100, in=90]   (0.1,-0.7);
\draw [ -] (0.2,-1.2) to[out=100, in=-90]   (0.1,-0.7);
%%%%%%

\fill (0,0.2) circle (0pt) node [right] {{\small $u_m$}};
\fill (0,2) circle (0pt) node [above] {{\small $b_m$}};
\fill (-0.1,-0.7) circle (1.2pt);
\fill (-0.1,-0.7) circle (0pt) node [left] {{\tiny $y_m$}};
\fill (0.1,-0.7) circle (1.2pt);
\fill (0.1,-0.7) circle (0pt) node [right] {{\tiny $z_m$}};

\fill (0,-2.1) circle (0pt) node [below] {{\small $\pa_* X_0$}};
\fill (3,-1) circle (0pt) node [right] {{\small $\pa \ca E_0(x,r,\theta)$}};
\fill (4,1.5) circle (0pt) node [right] {$\ca E_0(x,r,\theta)$};
\fill (-1.5,1.2) circle (0pt) node [left] {{\tiny $\Theta(u_m,b_m)$}};

\end{tikzpicture}
\end{center}  
\vspace{-0.7cm}  
\begin{figure}[htbp]
  \centering
  \caption{}
  \label{fig:final-stage}  
\end{figure}  
\psmall

%%%%%%%%%%%%
%Note that 
%${\rm rad}(\xi_{z_\infty}^+)=\pi/2$. 
By \eqref{eq:zn-en} and \eqref{eq:ynun-en}, 
$(C_0)_\infty$ has no boundary

%and hence by Lemma \ref{lem:inrad-collapse},
%\eqref{eq:rescaling-zwY} is caused by an inradius collapse.

By Theorem \ref{thm:inradius-collapse},
we have one of the following three cases:
\begin{enumerate}
\item[(a)] $(X_0)_\infty=(X_0)_\infty^1\,;$
\item[(b)]  Both $(X_0)_\infty^1$ and $(X_0)_\infty^2$ are nonempty$\,;$ 
\item[(c)]  $(X_0)_\infty=(X_0)_\infty^2$.
\end{enumerate}
In Case (a), $Y_\infty=(C_0)_\infty\times \mathbb R_+$, and hence $\pa Y_\infty=(C_0)_\infty$. This is a contradiction 
to $\gamma_{w_\infty}^+\subset \pa Y_\infty$.
In Case (b), suppose $w_\infty\in (X_0)_\infty^1$.
Then $\Sigma_{w_\infty}(Y_\infty)=\Sigma_{\tilde w_\infty}(C_\infty)/f_*$, where $f_*$ is not the identity
(Theorem \ref{thm:inradius-collapse}), and hence it has no boundary.
This is a  contradiction 
to $w_\infty\in \pa Y_\infty$.
In the rest of cases, we may assume $w_\infty\in (X_0)_\infty^2$. Then 
$\Sigma_{w_\infty}(Y_\infty)$ is the spherical suspension over 
$\Sigma_{w_\infty}((C_0)_\infty)$, and hence it has no boundary.
This is again a  contradiction, and completes the proof of 
Sublemma \ref{slem:z=intY}.
\end{proof}
\pmed

Choose $\delta_m>0$ satisfying
\beq\label{eq:delta-intY}
 B^Y(u_m,\delta_m)\subset {\rm int}\,Y,
\quad \delta_m\ll |u_m,b_m|.
\eeq
Set 
\[
   \alpha_m:={\rm rad}(\xi_{u_m}^+).
\]
If $\alpha_m=\pi/2$,
then Lemma \ref{lem:f=id} implies that 
$\Sigma_{u_n}(Y)=\Sigma_{\tilde u_n}(C)/f_*=\Sigma_{\tilde u_n}(C)$.
It turns out that $u_n\in\pa Y$, which contradicts 
Subelmma \ref{slem:z=intY}.
Thus we have $\alpha_m>\pi/2$.

\begin{slem} \label{slem:alpha>pi/2}
We have  $\alpha:=\liminf_{m\to\infty} \alpha_m>\pi/2$.
\end{slem}
\begin{proof}
Take large enough $i=i_m$ such that 
\beq \label{eq:wangle-gamma(yza)0}
 y_{mi_m}, z_{mi_m}\in B(u_m,\delta_m^2),
\eeq
and set
$y_m:=y_{mi_m}$, $z_m:=z_{mi_m}$. 
By contradiction, suppose  $\alpha=\pi/2$, and  consider the convergence
\beq \label{eq:rescal-ynun}
 \left(\frac{1}{|y_{m}, u_{m}|}Y, u_{m} \right)  \to (Y_\infty, u_\infty).
\eeq
Note that $Y_\infty$ has no boundary from \eqref{eq:delta-intY} and \eqref{eq:wangle-gamma(yza)0}.
On the other hand,  we have ${\rm rad}(\xi_{u_\infty}^+)=\pi/2$ in the limit.
This implies $\Sigma_{u_\infty}(X_\infty)=\Sigma_{u_\infty}((X_0)_\infty)$, that is,  the  convergence \eqref{eq:rescal-ynun} is caused by an inradius collapse (see Definition \ref{defn:inradius-conv} and Lemma 
\ref{lem:caused-by}).

Let 
$\xi_\infty\in\Sigma_{u_\infty}((X_0)_\infty)$ be the direction
defined as the limit of the geodesics $u_m b_m$ under 
\eqref{eq:rescal-ynun}.
Let $\Omega_\infty$ be the cone domain in 
$(X_0)_\infty$ around $\xi_\infty$ of angle $\pi/4$. 
Set $\mu_m:=|y_m,u_m|$, and 
let 
$\eta_{0,\infty}:(C_0)_\infty\to (X_0)_\infty$ 
be the limit of 
$\eta_0:(C_0/\mu_m,\tilde u_m)\to (X_0/\mu_m,u_m)$.
Since $\eta_0$ is injective and isometric on 
$\eta_0^{-1}(\Theta(u_m,b_m))$,
 $\eta_{0,\infty}$ must be injective on 
$\eta_{0,\infty}^{-1} (\eta_{0\infty}(\Omega_\infty))$.
Since
\eqref{eq:rescal-ynun} is caused by an inradius collapse,
it follows from Theorem \ref{thm:inradius-collapse}  that 
$\eta_{0,\infty}$ is injective on 
$(C_0)_\infty$. This yields 
$(X_0)_\infty\subset \pa Y_\infty$,
a contradiction. This completes the proof of Sublemma
\ref{slem:alpha>pi/2}.
\end{proof}

%%%%%%%%%%%%%%%%%%
\psmall
Set $v_m:=\gamma_{u_m}^+(t_0)$.
By Sublemma \ref{slem:alpha>pi/2},
take a point $a_m\in X$ such that
$\wangle v_m u_m a_m>\pi/2+c$
for a uniform constant $c>0$.
Then we have 
\beq\label{eq:wangle(vya)}
   \wangle v_m y_{mi} a_m>\pi/2+c/2,\quad
\wangle v_m z_{mi} a_m>\pi/2+c/2,
\eeq
for large enough $i$.
On the other hand, from Lemmas \ref {lem:calDC}  and \ref{lem:almostperp+}, 
\eqref{eq:wangle(vya)} implies   
\begin{align*}
  & \wangle z_{mi} y_{mi} a_m
   \le \angle z_{mi} y_{mi} a_m< \pi/2-c/3,\\ 
&\wangle y_{mi} z_{mi} a_m
   \le \angle y_{mi} z_{mi} a_m< \pi/2-c/3.
\end{align*}
Since $\wangle y_{mi}a_mz_{mi}<o_i$,
we have a contradiction.
This completes the proof of Proposition \ref{prop:paC-subset-paX}.  
\end{proof} 
\psmall

 Applying Theorem \ref{thm:never-happen},
we obtain the closedness of $\ca S^1\cup\ca C$.

\begin{thm}\label{thm:CS1=closed}
$\ca S^1\cup\ca C$ is closed in $N_0$
\end{thm}

\begin{proof}
Let a sequence $y_i$ in $\ca C$ converge 
to a point $x\in X_0$.
Since ${\rm rad}(\xi_{y_i}^+)=\pi/2$,
we have ${\rm rad}(\xi_x^+)=\pi/2$,
and hence $x\in X_0^1$. Let $p:=\eta_0^{-1}(x)$.
In view of Lemma \ref{lem:closed-S1},
it suffices to show that $f_*$ is not the 
identity on $\Sigma_p(C_0)$.
Suppose that $f_*$ is  the identity.
Then $\Sigma_x(X_0)\subset\pa\Sigma_x(Y)$.
By Lemma \ref{lem:CsubC}, $\Sigma_x(X_0)$
must have nonempty boundary.
Now we can apply Theorem \ref{thm:never-happen}
to conclude that $B^{X_0}(x,r)\subset {\rm int}\,X_0\setminus\ca C$ for small enough $r>0$. This is a contradiction.
\end{proof}

\setcounter{equation}{0}

%%%%%%%%%%%%%%%%%%%%%%%%%%%%%%%%%%%%%%%
\pmed\n
\section{Hausdorff dimensions of 
boundary singular sets.}\,  \label{sec:dimension}
%%%

In this section, we prove Theorems \ref{thm:dim(metric-sing)}  and  \ref{thm:nowhere}.
 Recall $m:=\dim Y=\dim X_0+1$.

\pmed
\n
{\bf Criterion for closedness of  $\pa X_0$.}\,\,
First we provide a criterion for the 
closedness of the boundary $\pa X_0$.

Let $k_0=\dim X_0-1$.

\begin{lem}\label{lem:charact-paX0}
 $\pa X_0$ can be written as the union
\beq \label{eq:pa0X0}
      \pa X_0 = \eta_0(\pa C_0)\cup\ca C(k_0-1)\cup\ca S^1(k_0-1).
\eeq
In particular, we have $\dim_H \overline{\pa X_0}\le k_0$.
\end{lem}
\begin{proof}
\eqref{eq:pa0X0} immediately follows from Lemma \ref{lem:criteria-int-pa}.
\eqref{eq:pa0X0}  and Theorems
\ref{thm:CS1=closed}, \ref{thm:dim(metric-sing)} (1) imply
$\dim_H \overline{\pa X_0}\le k_0$.
\end{proof}

\begin{lem}\label{lem:partial-closed}
$\pa X_0$ is a closed subset if and only if 
\begin{align} \label{eq:kcapm-1} 
\text{ $\ca C(k_0-1)\cup\ca S^1(k_0-1))\cup \pa Y$ is closed}.  
\end{align}
\end{lem}

\begin{proof}
 Let  $x_i$ be a sequence in $\pa X_0$ converging to a point $x\in X_0$.
Note that if  $x_i\in\ca C(k_0-1)\cup\ca S^1(k_0-1)$, then  $x\in \ca C\cup\ca S^1$ by 
Theorem \ref{thm:CS1=closed}.
Hence if $x\in\pa Y$, then Lemma  \ref{lem:criteria-0} implies $x\in\pa X_0$.
Since $\pa_* X_0=\eta_0(\pa C_0)$ is 
closed, in view of \eqref{eq:pa0X0},
it follows  that
$\pa X_0$ is closed if and only if the limits $x$
of all sequences $x_i\in\ca C(k_0-1)\cup\ca S^1(k_0-1)$ satisfy the following condition:
\[
     \text{if $x\in {\rm int} \,Y$, then 
$x\in\ca C(k_0-1)\cup\ca S^1(k_0-1)$.}
\]
 Note that the last condition is nothing but 
\eqref{eq:kcapm-1}.
This completes the proof.
\end{proof}

The closedness of the boundary of any Alexandrov space with curvature bounded below follows from Perelman's topological stability (\cite{Pr:alexII}).

\begin{prob}
Determine if  $\pa X_0$  is closed.
\end{prob}

Before proving Theorem \ref{thm:dim(metric-sing)}, we begin with 
the following result in Alexandrov geometry.

\pmed

\begin{lem}\label{lem:convexity}
For given $k\in \N$ and $v>0$, there exist
$C=C(k,v)>0$ and $\e=\e(k,v)>0$ satisfying the following:
Let $\Sigma$ be a $k$-dimensional 
Alexandrov space  with curvature $\ge 1$ having 
$\ca H^k(\Sigma)\ge v$.
Suppose that a group $G$ of isometries of $\Sigma$
has an orbit $Gp$ of diameter $<\e$ for some 
$\e\le\e_0$.
Then there is a $G$-fixed point $q$ with
$|p,q|<C\e$.
\end{lem}
\begin{proof} 
By \cite{FMY} (cf. \cite{PtPt:extremal}), $\Sigma$ is {\it locally $C$-Lipschitz
contractible} for some $C=C(k,v)$
(see \cite{FMY} for the definition).
Let $\e_0:=C^{-1}$. Since $\diam(Gp)<\e\le C^{-1}$,
there is a convex neighborhood
$U$ containing $Gp$ with $\diam(U)<C\e$
(\cite{FMY}).
Consider the convex set $K:=\bigcap_{g\in G} gU$. Replacing $U$ by a slightly larger convex
neighborhood if necessary, we may assume
that $K$ has nonempty boundary.
Since $K$ is $G$-invariant, the farthest point
 $q$ of $K$ from the boundary $\pa K$ is a required point.
\end{proof}

\psmall

%%%%%%%%%%%%%%%%%%%
In Theorem \ref{thm:dim(metric-sing)}, 
we already know  an example 
with $\dim \ca C(k)=k+1$ or $\dim \ca S^1(k)=k+1$
for any $0\le k\le m-3$ (see Example \ref{ex:non-closed}(1)).
%%%%%%%%%%%%%%%%%%%%%

%

\begin{proof}[Proof of Theorem \ref{thm:dim(metric-sing)}]
(1)\, Note that 
\beq \label{eq:Nsing=eta(C0)}
X_0^{\rm sing}=\eta_0(C_0^{\rm sing})\cup\ca S^1\cup\ca C.
\eeq
Theorem \ref{thm:dim-sing} shows 
$\dim_H C_0^{\rm sing}\le m-2$, and hence
$\dim_H \eta_0(C_0^{\rm sing})\le m-2$.
%%%%%%%%%%%%%%%%%%
For any $x\in \ca S^1\cup\ca C$, set $p:=\eta_0^{-1}(x)$.  
We already know that $\Sigma_x(X_0)=\Sigma_p(C_0)/f_*$, where
$f_*$ is not the identity by Theorem \ref{thm:S1-nontrivial-f*}.
It follows  that  
\beq
     \ca H^{m-2}(\Sigma_x(X_0))\le \omega_{m-2}/2,
\eeq
where $\omega_{m-2}$ denotes the volume of $(m-2)$-dimensional unit sphere.

We show
\beq \label{eq:volume=1/2}
\ca H^{m-1}(\Sigma_x(D(Y)))\le \omega_{m-1}/2.
\eeq
Actually,  if $x\in{\rm int}\,Y$, then  we have  
$\ca H^{m-1}(\Sigma_x(Y))=\ca H^{m-1}(\Sigma_p(C)/f_*)\le \omega_{m-1}/2$.
If $x\in\pa Y$,  Lemma \ref{lem:criteria-0} implies $p\in\pa C_0$, and hence we have 
$\ca H^{m-2}(\Sigma_p(C_0))\le \omega_{m-2}/2$ and 
$\ca H^{m-1}(\Sigma_x(Y))\le \omega_{m-1}/4$. 
In particular, we obtain   
 $\ca S^1\cup \ca C\subset  D(Y)^{\rm sing}$. 
Theorem \ref{thm:dim-sing}  and Lemma \ref{lem:XbiLipX} then yield
$\dim_H(\ca S^1\cup \ca C) \le m-2$.
\psmall\n
(2)\,
For $0\le k\le m$, let $Y(k,\delta)$ denote the set of $(k,\delta)$-strained points $y\in Y$.
By \cite[Theorem 10.7]{BGP}, we have
\beq\label{eq:dimY-Y(k,delta)}
\dim_H (Y\setminus Y(k,\delta))\le k-1.
\eeq
Therefore  we only have to show 
\beq \label{eq:S1cupCsubset Y-Y}
\ca S^1(k)\cup\ca C(k)\subset Y\setminus Y(k+2,\delta)
\eeq
for small enough $\delta$.
Suppose \eqref{eq:S1cupCsubset Y-Y} does not hold. Then there exists a 
sequence $x_i$ in 
$(\ca S^1(k)\cup\ca C(k))\cap Y(k+2,\delta_i)$
with $\lim_{i\to\infty}\delta_i=0$.
%%%%%
Choose a  $(k+2,\delta_i)$-strainer $\{ (a_{ij},b_{ij})\}_{j=1}^{k+2}$ of $Y$ at $x_i$, and 
set $\xi_{ij}:=\uparrow_{x_i}^{a_{ij}}$ and 
$\eta_{ij}:=\uparrow_{x_i}^{b_{ij}}$\, $(1\le j\le k+2)$.
Let $f_{i*}$ denote the isometric involution on
$T_{p_i}(C_0)$ induced by that on $\Sigma_{p_i}(C_0)$, where $\eta_0(p_i)=x_i$.
Since $\Sigma_{x_i}(Y)=\Sigma_{p_i}(C)/f_{i*}$, 
we may assume
$\{ (\xi_{ij},\eta_{ij})\}_{j=1}^{k+2}\subset\Sigma_{x_i}(X_0)$ 
by slightly changing $\xi_{ij},\eta_{ij}$ if necessary.
Let $\tilde \xi_{ij}, \tilde \eta_{ij}\in\Sigma_{p_i}(C_0)$ be arbitrary lifts of $\xi_{ij},\eta_{ij}$.
From
\[
 \angle(\tilde\xi_{ij},\tilde\eta_{ij})\ge    
 \angle(\xi_{ij},\eta_{ij})>\pi-\delta_i,
\]
we have $\diam(\eta_0^{-1}(\xi_{ij}))<2\delta_i$
and $\diam(\eta_0^{-1}(\eta_{ij}))<2\delta_i$.
In view of Lemma \ref{lem:convexity},
slightly changing $\xi_{ij},\eta_{ij}$ again, 
we may assume that 
$\tilde \xi_{ij}, \tilde \eta_{ij}\in\tilde{\ca F}_{p_i}$. This observation also shows that 
$\{(\tilde\xi_{ij},\tilde\eta_{ij})\}_{j=1}^{k+2}$ is a 
$(k+2,\delta_i)$-strainer of $\Sigma_{p_i}(C_0)$
in the global sense (see \cite[9.1]{BGP}).

Passing to a subsequence, we may assume
that 
$(\Sigma_{p_i}(C_0), \left<f_{i*}\right>)$ 
converges to a pair 
$(\tilde\Sigma_\infty,\left< f_\infty\right>)$. 
Let $\tilde\xi_{\infty,j}, \tilde\eta_{\infty,j}\in \tilde\Sigma_\infty$ be the limits of $\tilde \xi_{ij}, \tilde \eta_{ij}$ under this convergence.
Note that $\{(\tilde\xi_{\infty,j}, \tilde\eta_{\infty,j})\}_{j=1}^{k+2}$ 
is a global $(k+2,0)$ strainer of 
$\tilde\Sigma_\infty$.
It follows from the splitting theorem that 
$\tilde\Sigma_\infty$ admits an isometric 
embedding $\iota:\mathbb S^{k+1} \to \tilde\Sigma_\infty$.
Since $\{(\tilde\xi_{\infty,j}, \tilde\eta_{\infty,j})\}_{j=1}^{k+2}$ are fixed by $f_\infty$, 
$\iota(\mathbb S^{k+1})$
is also fixed by $f_\infty$.

Let $\Sigma_\infty$ denote the limit of 
$\Sigma_{x_i}(X_0)$. Then we have 
$\Sigma_\infty=\tilde\Sigma_\infty/f_\infty$.
Under the convergence 
$\Sigma_{x_i}(X_0)\to\Sigma_\infty$, the $k$-dimensional extremal subset
$\ca F_{x_i}$ converges to a  
$k$-dimensional extremal subset
$\ca F_\infty$ of $\Sigma_\infty$ (see \cite{Kap:stab}).
%%%
 Let 
$\eta_\infty:\tilde\Sigma_\infty\to\Sigma_\infty$ 
be the projection, and set 
$\tilde{\ca F}_\infty:=
\eta_\infty^{-1}(\ca F_\infty)$.

\begin{slem} \label{slem:SsubsetF}
\[
    \tilde{\ca F}_\infty={\rm Fix}(f_\infty).
\]
\end{slem}
\begin{proof}
Since any limit of $f_{i*}$-fixed point is fixed by $f_\infty$, we only have to show that ${\rm Fix}(f_\infty)\subset\tilde{\ca F}_\infty$. 
For any $u\in{\rm Fix}(f_\infty)$, take $u_i\in\Sigma_{p_i}(C_0)$ converging to $u$.
From $f_\infty u=u$, we have 
$\lim_{i\to\infty}|f_{i*}(u_i),u_i|=0$.
From Lemma \ref{lem:convexity}, we get
a point $v_i\in\ca F_{p_i}$ near $u_i$ with $\lim_{i\to\infty}|v_i,u_i|=0$. This implies $u\in \tilde{\ca F}_\infty$.
\end{proof}

It follows from Sublemma \ref{slem:SsubsetF} that $\iota(\mathbb S^{k+1})\subset\tilde{\ca F}_\infty$,
and hence $\dim \tilde{\ca F}_\infty\ge k+1$.
This is a contradiction since 
$\tilde{\ca F}_\infty$ is $k$-dimensional.
\par\n
(3)\, This is a special case of (2).
Theorem \ref{thm:dim-sing} shows
$\dim_H (C_0^{\rm sing} \cap {\rm int} \,C_0)\le m-3$.
Therefore in view of \eqref{eq:Nsing=eta(C0)}
and \eqref{eq:dimY-Y(k,delta)}, it suffices to 
show that 
\beq\label{eq:SCsubset Y(general)}
(\ca S^1\cup\ca C)\cap {\rm int} X_0\subset
Y\setminus 
Y(m-2,\delta).
\eeq
Since $\ca S^1(m-3)\cup\ca C(m-3)\subset 
\pa X_0$ by Lemma \ref{lem:charact-paX0},
\eqref{eq:S1cupCsubset Y-Y}
implies \eqref{eq:SCsubset Y(general)}.
 This completes the proof.
\end{proof}

%%%%%%%%%%%%%%%%%%%%%%

\begin{proof} [Proof of Theorem \ref{thm:nowhere}]
 Let us assume that ${\rm int}\, X_0^2$ is nonempty.
Fix a point $x_0\in {\rm int}\, X_0^2$, and consider the component 
$X_{0,\alpha}$ of $X_0$ containing $x_0$. By our hypothesis on the uniform positive lower bound
for ${\rm inrad}(M_i)$, 
any neighborhood of $X_{0,\alpha}$ meets ${\rm int}\, X$. 
Using Lemma \ref{lem:single-interior}, we take a point $y_0\in{\rm int}\,X_0^1\cap X_{0,\alpha}$. 
Let $\gamma$ be a minimal $X_0$-geodesic joining $x_0$ to $y_0$, and let $z$ be the first  point of $\gamma$ 
meeting $\pa({\rm int}\, N_0^2)$.
By Lemma \ref{lem:non-extend}, we have $z\in\ca S^2$. 
Let $\{ \tilde z_1,\tilde z_2\}:=\eta_0^{-1}(z)$.
From Sublemma \ref{slem:eta-ball}, take a small $r>0$ with $r\ll\min\{ |z,x_0|, |z,y_0|, |\tilde z_1,\tilde z_2|\}$
satisfying 
\[
                 B^{X_0}(z, r)\subset \eta_0(B(\tilde z_1, 2r))\cup \eta_0(B(\tilde z_2,2r)).
\]
Take $x\in{\rm int}\,X_0^2\cap B^{X_0}(z, r/3)$.
Using Lemma \ref{lem:S2character}, take
 $y\in{\rm int}\,X_0^1\cap B^{X_0}(z, r/3)$,
and set $q:=\eta_0^{-1}(y)$.
Here we may assume $q\in B^{C_0}(\tilde z_1, 2r/3)$. 
Find $p\in\eta_0^{-1}(x)$ such that $p\in B^{C_0}(\tilde z_1, 2r/3)$.
Fix  $r_1>0$ such that 
$$
     B^{X_0}(x, r_1)\subset {\rm int}\, X_0^2,
$$
and set $d:=d^{C_0}(p,q)$, and  
$S:=S^{C_0}(p,d)\cap B^{C_0}(q, r_2)$, where $r_2\ll r$ so that $S\subset {\rm int}\,C_0^1$.
By the coarea formula for 
$d_p^{C_0}$, replacing $y$ if necessary, we may assume $\ca H^{m-2}(S)>0$.

For every $w\in S$, let $\gamma_w$ denote a minimal geodesic in $C_0$
joining $p$ to $w$. 
Applying Lemma \ref{lem:non-extend}, we can find 
a point, say $\varphi(w)$, of $\gamma_w$  
satisfying $\varphi(w) \in \tilde{\ca S}^2$. 
Note that  
\[
                     r_1<|p,\varphi(w)|<d.
\]
For arbitrary $w_1, w_2\in S$, consider a comparison triangle
$\tilde\triangle pw_1w_2$ in $M_{\tilde \kappa}^2$,
where $\tilde\kappa$ is the lower curvature bound of $C_0$
(see Proposition \ref{prop:extendAS}).
Let $\tilde \varphi(w_k)$\, $(k=1,2)$\, be the point on the edge
$\tilde p\tilde w_k$ corresponding to $\varphi(w_k)$.
By the curvature condition for $C_0$, we have 
\[
       |\varphi(w_1),\varphi(w_2)|  \ge  |\tilde\varphi(w_1),\tilde\varphi(w_2)|\ge c|w_1,w_2|,
\]
where $c=c(\tilde\kappa, d, r_1)>0$, which implies 
\[
     \dim_H \varphi(S) \ge \dim_H  S =m-2.
\]
Therefore together with Proposition \ref{prop:eta-2cover}, we conclude that 
\[
              \dim_H\ca S^2 = \dim_H {\tilde{\ca S}}^2\ge m-2.
\]

\pmed

\begin{center}
\begin{tikzpicture}
[scale = 0.6]
\fill (-3,0) circle (3pt);
\fill (5,0) circle (3pt);
\draw [-, thick] (-3,0) circle  (1);
\draw [-, thick] (5,0) circle  (1);
\fill (-4.3,0) circle (0pt) node [left] {\small{${\rm int X_0^2}$}};
\fill (6.3,0) circle (0pt) node [right] {\small{${\rm int X_0^1}$}};

\draw [-, thick] (5,0) to [out=90, in=-80] (4.88,0.96);
\draw [-, thick] (5,0) to [out=-90, in=80] (4.88,-0.96);
%\fill (4.5,0) circle (3pt);
\draw [-, thick] (-3,0) -- (5,0);
\draw [-, thick] (-3,0) -- (4.97,0.24);
\draw [-, thick] (-3,0) -- (4.94,0.48);
\draw [-, thick] (-3,0) -- (4.91,0.72);
\draw [-, thick] (-3,0) -- (4.88,0.96);
\draw [-, thick] (-3,0) -- (4.97,-0.24);
\draw [-, thick] (-3,0) -- (4.94,-0.48);
\draw [-, thick] (-3,0) -- (4.91,-0.72);
\draw [-, thick] (-3,0) -- (4.88,-0.96);

\draw[-,very thick] [-, thick] (1,0) to [out=90, in=-120] (1.5,1.5);
\draw[very thick] [-, thick] (1,0) to [out=-90, in=120] (1.5,-1.5);
\fill (1.5,1.5) circle (0pt) node [left] {\small{$\ca S^2$}};

\end{tikzpicture}
\end{center}  
\vspace{-0.5cm}  
\begin{figure}[htbp]
  \centering
  \caption{}
  \label{fig:measure-S2}  
\end{figure}  
\psmall\n
This completes the proof of 
Theorem \ref{thm:nowhere}.
\end{proof}

\begin{rem} \label{rem:Cantor}
In the following example, we show that  in  
Theorem \ref{thm:nowhere}, one can not expect that the $(m-1)$-dimensional 
Hausdorff measure of $\ca S^2$ is zero.
Actually,  $\ca H^{m-1}(\ca S^2)/\ca H^{m-1}(N_0)$ can be close to $1$.
This shows that Theorem \ref{thm:nowhere} is sharp in that sense.
\end{rem}

\begin{ex} \label{ex:Cantor}
We define  a sequence of flat surfaces $M_n$ with wild boundary 
converging to a two-dimensional space $N$ by making use of construction of 
$\epsilon$-Cantor set (cf. \cite{AlBu}).

\psmall\n
Step 1. \, Given $\epsilon\in(0,2\pi)$, let $\delta=2\pi-\epsilon$.
Following \cite{AlBu}, we construct an $\epsilon$-Cantor set in $[0,2\pi]$.
 We start with $I_0 := [0, 2\pi]$ and remove from $I_0$ an open interval $J_{0,1}=(\pi-\delta/4,\pi+\delta/4)$ around the center of $I_0$ of length $\delta/2$. 
Setting $J_1 :=J_{0,1}$, we define 
\[
    I_1 :=I_0\setminus J_1,
\]
where $I_1$ consists of $2^1$ disjoint closed intervals $I_{1,1}=[0,\pi-\delta/4]$ and $I_{1,2}=[\pi+\delta/4,2\pi]$ with 
\[
     L(I_{1,j}) =\pi -\delta/4\, \,(j=1,2),  \quad  L(I_1) =2\pi -\delta/2.
\]
Next, remove from each $I_{1,j}$ an open interval $J_{1,j}$around the center of $I_{1,j}$ of length $\delta/2^{2+1}$. 
Setting $J_2 := \bigcup_{j=1}^2 J_{1,j}$, we define 
\[
    I_2 :=I_1\setminus J_2,
\]
which consists of $2^2$ disjoint closed intervals $\{ I_{2,j}\}_{j=1}^{2^2}$  with 
\[
     L(I_{2,j}) =(\pi -\delta/4-\delta/8)/2, \quad  L(I_2) =2\pi -\delta/2-\delta/4.
\]
Thus, inductively assuming that $I_{n-1}$, $I_{n-1,j}$ and $J_n$ are defined, 
we define 
\[
    I_n :=I_{n-1}\setminus J_n,
\]
which consists of $2^n$ disjoint closed intervals $\{ I_{n,j}\}_{j=1}^{2^n}$  with 
\begin{align*}
     L(I_{n,j}) &=(\pi -\delta/2^2-\delta/2^3 - \cdots \delta/2^{n+1})/2^{n-1}, \\
    L(I_n) &=2\pi -\delta/2-\delta/2^2-\cdots - \delta/2^n, \\
    L(J_n) & = \delta/2^n.
\end{align*}
Finally we set
\[
         I_{\infty}:=\bigcap_{n=0}^{\infty}\, I_n, \quad J_{\infty}:=\bigcup_{n=1}^{\infty}\,J_n.
\]
Note that $\ca H^1(I_{\infty}) =\lim_{n\to\infty} L(I_n)=2\pi-\delta=\e$
and $\ca H^1(J_{\infty})=\delta$.
The set $I_{\infty}$ is called an $\e$-Cantor set.
\psmall\n
Step 2.\, Inductively define a periodic $C^{\infty}$-function $g_n:\mathbb R\to [0,1]$ 
with period $2\pi$ in such a way that 
\begin{enumerate}
\item $\displaystyle{\supp (g_n|_{[0,2\pi]}) = \bigcup_{k=1}^nJ_n;}$
\item $|g_n''|\le c$, where $c$ is a uniform positive constant independent of $n;$
\item $g_n=g_{n-1}$ on $\displaystyle{\bigcup_{k=1}^{n-1} J_k};$
\end{enumerate}
Let $D_n$ be the domain bounded by the curve $y=g_n(x)$ and the line $y=-1/n$,
and let $M_n$ be the quotient of $D_n$ by the infinite cyclic group 
$\Gamma$ generated by $\gamma(x,y)=(x+2\pi, y)$.
The second fundamental form of $\pa M_n$ 
satisfies $|\Pi_{\pa M_n}|\le \lambda=\lambda(c)$,
and therefore $M_n\in\ca M_b(2,0,\lambda, \pi+2)$.
As $n\to\infty$, $g_n$ converges to a $C^1$-function $g_{\infty}$,
and $M_n$ converges to $N=D_{\infty}/\Gamma$,
where $D_{\infty}$ is the closed set bounded by $y=g_{\infty}(x)$ and $y=0$.
Let $\pi:D_{\infty}\to N$ be the projection.
Obviously $N_0^1$ coincides with the $\pi$-image of 
$\{ (x,y)\,|\,x\in J_{\infty}, \text{$y=0$ or $y=g_{\infty}(x)$}\}$,
which is open and dense in $N_0$.
Moreover, $\ca S^1$ is empty and  $\ca S^2=\pa N_0^2=N_0^2 = \pi(I_{\infty})$
Thus we conclude that 
\[
     \ca H^1({\ca S}^2) = {\ca H}^1(I_{\infty}) =2\pi -\delta.
\]
It also  follows from the above condition (2) that $\lim_{\delta\to 0}\ca H^1(N_0)=2\pi$
and therefore 
\[
                 \lim_{\delta\to 0}\, \ca H^1(\ca S^2)/\ca H^1(N_0)=1.
\]
\end{ex}

\pmed
As  the following example shows,  the set $\ca S^2$ is much smaller than $\ca S^1$ in some cases.
\psmall 
\begin{ex} \label{ex:smallS2}
For any positive integer $n$, consider the finite set
\[
   Q_n :=\left\{ (1/k, \ell/k)\in\mathbb R^2\,|\,\ell\in\mathbb Z, k\in\mathbb N,  |\ell|\le k, 1\le k\le n\,\right\}.
\]
Let $r$ denote the reflection of $\mathbb R^2$ with respect to $x=1$.
Let $g_n:[0,2]\times\R \to \R_+$ be a smooth function such that
\benu
 \item $(g_n|_{[0,2]\times [-1,1]}^{-1}(0)=Q_n\cup r(Q_n)\,;$
 \item $|\nabla\nabla g_n|\le C$ for some uniform constant $C\,;$
 \item $g_n(x, y+2)=g_n(x,y)\,;$
 \item $g_n(2-x, y)=g_n(x,y)$ for all $0\le x\le 2\,;$
 \item $g_n$ takes a constant $\e_n$ on $([0, 1/2n]\cup[2-1/2n, 2])\times \mathbb R$,
       where $0<\e_n\ll 1/n$.
\eenu
%%%%%
Set $h_n:= g_n+\e_n$, and consider the following closed domain bounded by $z=0$ and $z=h_n(x,y)$:
\[
L_n:=\left\{ (x,y,z)\in\mathbb R^3\,|\, 0\le z\le h_n(x,y), \frac{1}{3n}\le x\le2-\frac{1}{3n} 
\,\right\}
\]
For $\delta_n=2\e_n/\pi$, 
%%%%%
let us consider the disk $D(\e_n, 2\e_n/\pi)$
defined in Example \ref{ex:basic-disk}
such that $J_n:=\pa_* D(\e_n, 2\e_n/\pi)$ is 
an arc of length $2\e_n$.
%%%%%%%%%%%%
Note  that $\pa L_n$ consists of two copies of $[0,2\e_n]\times \mathbb R$.
Therefore we can glue $L_n$ and two copies of $D(\e_n, 2\e_n/\pi)\times \mathbb R$
along $\pa L_n$ and  $J_n\times \mathbb R$ 
naturally
to get a
complete three-manifold $\hat M_n$ with boundary. 
Let $M_n$ be the quotient of $\hat M_n$ by the isometric $\Z$-action induced by $y\to y+2$.
Note that $M_n$ is a compact manifold diffeomorphic to $D^2\times S^1$,
which belongs to $\ca M_b(3, 0, \lambda, d)$ for 
some $\lambda,d$.
Let $\hat N\subset \mathbb R^3$ and $N$ be the limits of $\hat M_n$ and $M_n$ respectively, and
let $\pi:\hat N\to N$ be the projection.
Let $Q_{\infty}:=\cup_{n=1}^{\infty} Q_n$.
Note that the set of type $2$ singular set $\ca S^2$ of $N_0$ coincides with
$\pi$-image of $(Q_{\infty}\cup r(Q_{\infty}))\times \{ 0\}$, and hence 
$\dim_H \ca S^2 = 0$.
On the other hand,  the set of type $1$ singular set $\ca S^1$ of $N_0$ coincides with
$\pi$-image of the lines $(x,z)=(0,0)$ and $(x,z)=(2,0)$, and hence 
$\dim_H \ca S^1 = 1$.
\end{ex} 
\pmed

%$\partial_{\sharp} X_0$, $\partial_{\#} X_0$, $\partial_0 X_0$
%%%%%%%%%%%%%%%%%%%%%%%%%%%%%%%%%%%%%%

\end{document}